\documentclass[a4paper, 11pt]{article}
\usepackage{graphicx}
\usepackage{pdfpages}
\usepackage[utf8]{inputenc}
\usepackage{amsmath, amsthm, amssymb, amsfonts}
\usepackage{bm}
\usepackage{mathtools}
\usepackage{enumerate, enumitem}
\usepackage{float}
\usepackage{makecell, siunitx, multirow, threeparttable, caption, stackengine}
\usepackage{pdflscape}
\usepackage[nottoc]{tocbibind}
\usepackage[titletoc]{appendix}
\usepackage{afterpage,natbib,lipsum}
\usepackage[unicode, psdextra, linkcolor=blue, citecolor=blue, colorlinks=true]{hyperref}
\usepackage[top = 3.5cm, bottom = 3.5cm, left = 3cm, right = 3cm]{geometry}
\usepackage{accents}
\usepackage{xcolor}
\usepackage{authblk}
\usepackage{array}
\usepackage{bookmark, xpatch}
\usepackage[symbols,nogroupskip,sort=none,numberedsection=autolabel]{glossaries-extra}

\numberwithin{equation}{section}
\theoremstyle{plain}
\newtheorem{theorem}{Theorem}[section]
\newtheorem{lemma}{Lemma}[section]

\newtheorem{proposition}{Proposition}[section]

\theoremstyle{definition}

\newtheorem{example}{Example}[section]

\newtheorem{remark}{Remark}[section]

\newcommand{\ba}{\mathbf{a}}
\newcommand{\bA}{\mathbf{A}}
\newcommand{\bB}{\mathbf{B}}
\newcommand{\bF}{\mathbf{F}}
\newcommand{\bh}{\mathbf{h}}
\newcommand{\bH}{\mathbf{H}}
\newcommand{\bI}{\mathbf{I}}
\newcommand{\bV}{\mathbf{V}}
\newcommand{\bx}{\mathbf{x}}
\newcommand{\bX}{\mathbf{X}}
\newcommand{\by}{\mathbf{y}}

\newcommand{\bmu}{\bm{\mu}}
\newcommand{\bbeta}{\bm{\beta}}

\newcommand{\bSigma}{\bm{\Sigma}}

\newcommand{\ErrF}{\mathrm{ErrF}}
\newcommand{\ErrR}{\mathrm{ErrR}}
\newcommand{\ErrT}{\mathrm{ErrT}}
\newcommand{\OptF}{\mathrm{OptF}}
\newcommand{\OptR}{\mathrm{OptR}}
\newcommand{\dfF}{\mathrm{df}_{\rm F}}
\newcommand{\dfR}{\mathrm{df}_{\rm R}}

\newcommand{\trace}{{\rm tr}}

\newcommand{\E}{\mathrm{E}} 
\newcommand{\var}{\mathrm{Var}}
\newcommand{\cov}{\mathrm{Cov}}
\newcommand{\prob}{\mathrm{P}}

\DeclareMathOperator*{\argmin}{arg\,min}

\newcommand{\R}{\mathbb{R}}

\newcommand{\T}{\top}
\newcommand{\comp}{\mathrm{c}}

\newcommand{\cS}{\mathcal{S}}
\newcommand{\ubar}[1]{\underaccent{\bar}{#1}}

\title{Predictive Model Degrees of Freedom in Linear Regression}
\author{Bo Luan}
\author{Yoonkyung Lee}
\author{Yunzhang Zhu}
\affil{The Ohio State University}
\date{}

\begin{document}
\maketitle

\begin{abstract}
	Overparametrized interpolating models have drawn increasing attention from machine learning. Some recent studies suggest that regularized interpolating models can generalize well. This phenomenon seemingly contradicts the conventional wisdom that interpolation tends to overfit the data and performs poorly on test data. Further, it appears to defy the bias-variance trade-off. As one of the shortcomings of the existing theory, the classical notion of model degrees of freedom fails to explain the intrinsic difference among the interpolating models since it focuses on estimation of in-sample prediction error. This motivates an alternative measure of model complexity which can differentiate those interpolating models and take different test points into account. In particular, we propose a measure with a proper adjustment based on the squared covariance between the predictions and observations. Our analysis with least squares method reveals some interesting properties of the measure, which can reconcile the “double descent” phenomenon with the classical theory. This opens doors to an extended definition of model degrees of freedom in modern predictive settings.
	\vspace{8pt}
	
\end{abstract}

\section{Introduction} \label{sec: introduction}

Overparameterized machine learning models have drawn increasing attention in recent years. These models are usually complex enough to achieve zero (or nearly zero) training error, yet they could still generalize well. One typical example is the state-of-the-art deep neural networks. \cite{zhang2016understanding} showed in an experiment that a well-designed interpolating deep neural network can perform respectably well even when considerable label noise is added. Similar interpolation performance has been observed in other machine learning methods such as random forests \citep{wyner2017explaining, belkin2018reconciling}, AdaBoost \citep{wyner2017explaining}, underdetermined least squares \citep{belkin2019two,hastie2019surprises, bartlett2020benign} and nearest neighbors \citep{belkin2018overfitting, xing2018statistical, xing2019benefit}. This phenomenon seemingly contradicts the conventional wisdom that interpolation tends to overfit the training data and performs poorly on test data. Further, it appears to defy the bias-variance trade-off.

In many situations, interpolating models may still perform poorly. It is easy to demonstrate varying performance of interpolating models numerically. For example, in our study of least squares method, interpolation occurs when the number of features $p$ for a linear regression model equals the sample size $n$ and continues beyond $n$ when the minimum-norm least squares method is applied. We call $n$ the interpolation threshold in this setting. As shown in the left panel of Figure \ref{fig: double_descent}, the prediction risk of these linear models follows a double U-shape pattern when viewed as a function of $p$. \cite{belkin2018reconciling} used the term ``double descent'' to describe these two side-by-side U-shape risk curves. They pointed out that models that are just beyond the interpolation threshold often have remarkably high risk and better models could emerge well beyond the interpolation threshold. This raises the necessity of differentiating and selecting from interpolating models.

Classical model selection criteria such as Mallows's $C_p$ \citep{mallows1973some}, AIC \citep{akaike1973information} and BIC \citep{schwarz1978estimating} all seek to balance the training error with model complexity. However, the interpolating models in our least squares example shown in Figure \ref{fig: double_descent} cannot be differentiated by these criteria as they all have the same training error and model complexity. Here, model complexity is defined through the map from the observations $\by=(y_1,\ldots,y_n)^\T$ to the fitted values $\hat{\bmu}=(\hat{\mu}_1,\ldots,\hat{\mu}_n)^\T$. This map is common for all interpolating models and corresponds to the identity map. Then a natural question is how to effectively quantify the complexity of statistical models, including those interpolating ones in a way that captures the difference in the prediction risk. Such a measure of model complexity could be useful in risk estimation to help differentiate and select models.
\begin{figure}[t!]
	\centering
	\includegraphics[scale = 0.46]{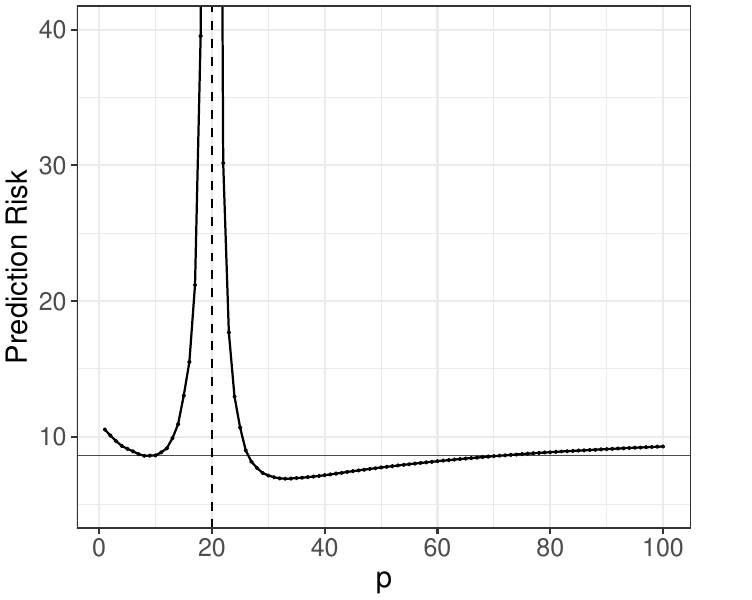}
	\includegraphics[scale = 0.46]{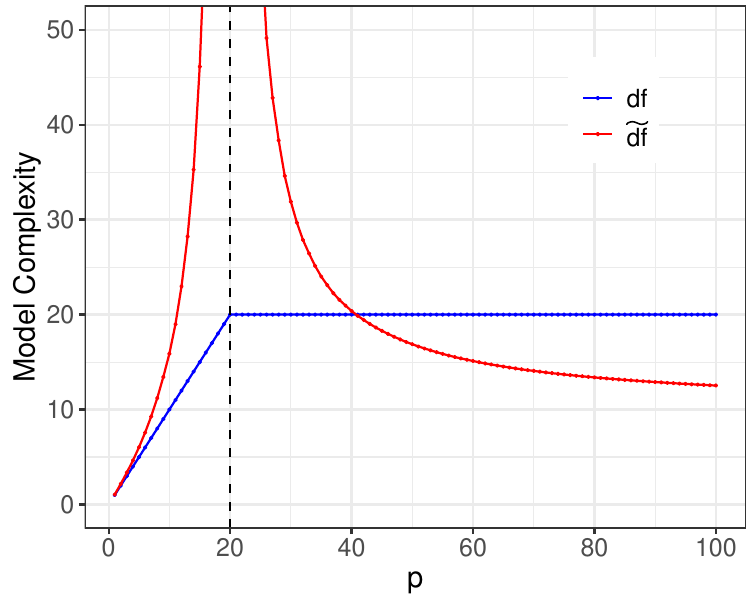}
	\caption{The ``double descent" phenomenon with least squares method (left) and model degrees of freedom (right) when sample size $n = 20$ and the number of variables $d = 100$. The vertical dashed line marks the interpolation threshold. The blue line is for the classical model degrees of freedom, and the red line is for the predictive model degrees of freedom.}
	\label{fig: double_descent}
\end{figure}

To answer this question, we examine the notion of model degrees of freedom in classical statistical theories. \cite{efron2004estimation} formally defined it as 
$$
	\mathrm{df} = \sum_{i=1}^n \frac{\cov(y_i,\hat{\mu}_i)}{\sigma_\varepsilon^2},
$$
where $\sigma_\varepsilon^2$ is the error variance. This classical model degrees of freedom, indicated by the blue line in the right panel of Figure \ref{fig: double_descent}, fails to explain the intrinsic difference among interpolating models, since it depends on in-sample prediction only. One of its underlying assumptions is that the values of covariates in test data are considered fixed and the same as those in the training data (Fixed-X setting). While this assumption may be valid in an experimental setting, it is less realistic in a predictive setting where new feature values typically arise. This motivates us to consider out-of-sample prediction for definition of model complexity instead, where test features are different from those in the training data (Random-X setting).

In this work, we aim to extend Efron's theory to the analysis of out-of-sample prediction risk and provide an extended model degrees of freedom for linear regression methods that is applicable to both non-interpolating and interpolating models. We call it the \textit{predictive model degrees of freedom}. We show that it adjusts the classical model degrees of freedom with the squared covariance between the observations and predicted values. In particular, for an interpolating model, it is of the form:
$$
	\widetilde{\mathrm{df}} = \mathrm{df} + \frac{n}{2}\sum_{i = 1}^n \E_{\bx_\ast}\left(\frac{\cov^2(y_i, \hat{\mu}_\ast \vert \bx_\ast)}{(\sigma_\varepsilon^2)^2}\right) - \frac{n}{2},
$$
where the expectation is taken over a random test point $\bx_\ast$, and $\hat{\mu}_\ast$ is a prediction of $y_\ast$ at $\bx_\ast$. As illustrated in the right panel of Figure \ref{fig: double_descent} by the red line, the predictive model degrees of freedom can indeed differentiate among interpolating models. Further, by aligning the prediction risk against the predictive model degrees of freedom, we demonstrate that the double descent phenomenon can be well reconciled with the classical theory. See Section \ref{subsec: double_descent}.

As a standard example of linear regression procedures, we look into the least squares method in detail. Using the predictive model degrees of freedom, we propose several risk estimators for model evaluation. Our analysis shows that, when the true model is indeed linear and the covariates are normally distributed, one estimator has a close connection with Hocking's $S_p$ criterion \citep{hocking1976biometrics, thompson1978selection} and $\hat{U}_{np}$ statistic \citep{breiman1983many}. Another estimator, obtained in a more general setting, is shown to be equivalent to the generalized covariance penalty criterion $\mathrm{RCp}^+$ defined in \cite{rosset2020fixed}. We find that this estimator could be negative around the interpolation threshold due to its large variance. To circumvent the issue, we develop an improved version that corrects the original estimator whenever it is negative.

We assess the performance of the proposed estimator through a number of numerical studies. Our prediction risk estimator exhibits smaller variance than the leave-one-out cross validation error, especially around the interpolation threshold. It also tends to favor more parsimonious models compared to classical criteria such as Mallows's $C_p$ and AIC. We believe that this is quite reasonable as out-of-sample prediction generally involves more uncertainty than in-sample prediction. Such uncertainty is reflected in the risk estimator through the predictive model degrees of freedom, which results in the selection of a simpler model.

The paper is organized as follows. In Section \ref{sec: review_classical}, we briefly review the classical theory on model degrees of freedom and prediction error estimation. We then propose the predictive model degrees of freedom for linear procedures in Section \ref{sec: new_df}. Sections \ref{sec: ls_method}, \ref{sec: pred_err_est} and \ref{sec: numerical_studies} study the predictive model degrees of freedom for least squares method in the context of subset regression, with the focus on properties, prediction error estimation and numerical studies, respectively. Section \ref{sec: grad_descient} discusses linear interpolating models. For the sake of conciseness, most of the proofs and details of examples will be deferred to the appendix.

\section{Classical Optimism Theory} \label{sec: review_classical}

In this section, we review the classical optimism theory presented in \cite{efron1986biased,efron2004estimation}. To begin with, we first introduce the notation and assumptions that will be used in this paper. Let $\lbrace (\bx_i, y_i) \in \R^d \times \R \vert i = 1,\ldots,n \rbrace$ be the training data generated under the following model assumptions:
\begin{enumerate}[label = A\arabic*.]
	\setcounter{enumi}{-1}
	\item
	$$
		y_i = \mu(\bx_i; \bbeta) + \varepsilon_i, \quad i = 1, \ldots, n,
	$$
	where $\mu(\cdot;\bbeta) \mathpunct{:} \R^d \to \R$ is the mean regression function with unknown parameter $\bbeta$.
	
	\item $\bx_1,\ldots,\bx_n$ are i.i.d. with $\E(\bx_i) = \mathbf{0}$ and $\var(\bx_i) = \bSigma$.
	
	\item $\varepsilon_1,\ldots,\varepsilon_n$ are i.i.d. with $\E(\varepsilon_i)=0$ and $\var(\varepsilon_i)=\sigma_\varepsilon^2$.
	
	\item $\bx_i$ and $\varepsilon_i$ are independent.
\end{enumerate}
We assume $\bSigma$ and $\sigma_\varepsilon^2$ are known throughout the paper unless noted otherwise. For brevity, we abbreviate $\mu(\bx_i;\bbeta)$ as $\mu_i$ and write $\bmu = (\mu_1,\ldots,\mu_n)^\T$,  $\by=(y_1,\ldots,y_n)^\T$ and $\bX = (\bx_1,\ldots, \bx_n)^\T$. For a given modeling procedure, let $\hat{\mu}_i$ be the fitted value of $y_i$ and $\hat{\bmu} = (\hat{\mu}_1, \ldots, \hat{\mu}_n)^\T$. 

We define the training error of $\hat{\bmu}$ given $(\bX,\by)$ as
\begin{equation}\label{eq: ErrT}
	\ErrT_\mathbf{X,y} = \frac{1}{n} \Vert \by - \hat{\bmu} \Vert^2.
\end{equation}
The classical risk analysis relies on the assumption that $\bx_1,\ldots,\bx_n$ are fixed and the covariate values are the same in the training and test data. Let $\tilde{\by}$ be an independent copy of $\by$ given $\bX$. The in-sample prediction error is defined as
\begin{equation}\label{eq: ErrF}
	\ErrF_{\bX,\by} = \frac{1}{n} \E(\Vert \tilde{\by} - \hat{\bmu} \Vert^2 \vert \bX,\by),
\end{equation}
where the expectation is with respect to $\tilde{\by}$. 

The training error generally underestimates the in-sample prediction error, since the data for model fitting are reused for evaluation. \cite{efron1986biased} used the term \textit{optimism} to refer to such a downward bias, indicating how optimistic the training error is as an estimate of the prediction error. Further, averaging over $\by$ given $\bX$, he defined the \textit{expected optimism} as
\begin{equation}\label{eq: expected_OptF}
	\OptF_{\bX} = \ErrF_{\bX} - \ErrT_{\bX},
\end{equation}
where $\ErrF_{\bX}$ and $\ErrT_{\bX}$ are the expectations of $\ErrF_{\bX, \by}$ and $\ErrT_{\bX, \by}$ with respect to $\by$ respectively. 

Note that $\ErrF_\bX$ and $\ErrT_\bX$ have bias-variance decomposition as shown below:
\begin{align}
	& \ErrF_\bX = \sigma_\varepsilon^2 + \frac{1}{n}\Vert \E(\hat{\bmu} \vert \bX) - \bmu \Vert^2 + \frac{1}{n}\E[\Vert \hat{\bmu}-\E(\hat{\bmu}\vert\bX)\Vert^2 \vert \bX], \label{eq: bias_var_ErrF}\\
	&\begin{aligned}
		\ErrT_\bX = \sigma_\varepsilon^2 
		& + \frac{1}{n}\Vert \E(\hat{\bmu} \vert \bX) - \bmu \Vert^2\\
		& + \frac{1}{n}\left(\E[\Vert \hat{\bmu}-\E(\hat{\bmu}\vert\bX)\Vert^2 \vert \bX] - 2\E[(\by-\bmu)^\T(\hat{\bmu}-\E(\hat{\bmu}\vert \bX)) \vert \bX]\right), \label{eq: bias_var_ErrT}
	\end{aligned}
\end{align}
where the three terms on the right hand side are the irreducible error, (squared) bias and variance, respectively. Subtracting \eqref{eq: bias_var_ErrT} from \eqref{eq: bias_var_ErrF} yields $\OptF_\bX$, which is the difference between the model variance on the test data and that on the training data, or the \textit{excess variance}. Further, we have
\begin{equation}\label{eq: OptF_cov_penalty}
	\OptF_{\bX} = \frac{2}{n}\E[(\by-\bmu)^\T(\hat{\bmu}-\E(\hat{\bmu}\vert \bX)) \vert \bX] = \frac{2}{n}\sum_{i=1}^n \cov(y_i, \hat{\mu}_i\vert \bX).
\end{equation}
In other words, the expected optimism can be expressed in terms of the covariance between the observations $y_i$ and their fitted values $\hat{\mu}_i$. 

Based on \eqref{eq: OptF_cov_penalty}, an unbiased estimator of $\ErrF_{\bX}$ is given by
\begin{equation}\label{eq: cov_penalty}
	\widehat{\ErrF} = \ErrT_{\bX, \by} + \frac{2}{n} \sum_{i=1}^n \cov(y_i, \hat{\mu}_i\vert \bX).
\end{equation}
This prediction error estimate can be interpreted as an adjusted training error with the covariance penalty that accounts for the flexibility of a model. \cite{efron2004estimation} then formally defined the degrees of freedom for a modeling procedure of $\hat{\bmu}$ as
\begin{equation}\label{eq: df_F}
	\dfF = \sum_{i=1}^n \frac{\cov(y_i, \hat{\mu}_i\vert \bX)}{\sigma_\varepsilon^2}.
\end{equation}

In some special cases, $\dfF$ can be calculated explicitly. For example, for a linear procedure that predicts $\by$ with $\hat{\bmu} = \bH \by$, where $\bH = (h_{ij})$ depends only on $\bX$, we have $\cov(y_i, \hat{\mu}_i \vert \bX) = \sigma_\varepsilon^2 h_{ii}$ and
\begin{equation}\label{eq: df_F_linear_procedures}
	\dfF = \sum_{i=1}^n h_{ii} = \trace(\bH).
\end{equation}
This also agrees with the model degrees of freedom for linear smoothers defined by \cite{tibshirani1987local} through the expected residual sum of squares. 

In practice, one may estimate the degrees of freedom and evaluate the prediction error estimate  $\widehat{\ErrF}$ based on it to differentiate a set of models and select from them. However, the focus on in-sample prediction has the limitation that we may not tell how a modeling procedure will generalize at new feature values that may arise in the future. This also impacts how we measure model complexity. As a case in point, it completely fails to differentiate among interpolating models, since the training error of these models are all zero and $\cov(y_i, \hat{\mu}_i \vert \bx_i) = \var(y_i \vert \bx_i) \equiv \sigma_\varepsilon^2$ is constant. These motivate us to study out-of-sample prediction and look for alternative model complexity measures.

\section{The Predictive Model Degrees of Freedom} \label{sec: new_df}

In this section, we investigate the out-of-sample prediction error for linear procedures in the Random-X setting. Let $(\bx_\ast,\varepsilon_\ast)$ be an independent copy of $(\bx_i,\varepsilon_i)$. Define $\mu_\ast = \mu(\bx_\ast; \bbeta)$, the true mean function value at $\bx_\ast$, and $y_\ast = \mu_\ast + \varepsilon_\ast$, a new realization of $y$ at $\bx_\ast$. Let $\hat{\mu}_\ast$ be the prediction of $y_\ast$ with a fitted model $\hat{\bmu}$. Note that $\hat{\mu}_\ast$ depends on $\bX$, $\by$ and $\bx_\ast$.

The out-of-sample prediction error of $\hat{\bmu}$ is defined as
\begin{equation} \label{eq: ErrR}
	\ErrR_{\bX,\by} = \E[(y_\ast - \hat{\mu}_\ast)^2 \vert \bX, \by],
\end{equation}
where the expectation is with respect to $(\bx_\ast, y_\ast)$. Following Efron's work, we then define the (Random-X) expected optimism as
\begin{equation} \label{eq: exp_OptR}
	\OptR_\bX = \ErrR_\bX - \ErrT_\bX,
\end{equation}
where $\ErrR_\bX$ is the expectation of $\ErrR_{\bX,\by}$ with respect to $\by$.

Consider a linear procedure with hat matrix $\bH$ such that $\hat{\bmu} = \bH \by$. For each $\bx_\ast \in \R^d$, there must exist $\bh_\ast \in \R^n$ that depends only on $\bX$ and $\bx_\ast$ such that
$$
	\hat{\mu}_\ast = \bh_\ast^\T \by.
$$
We call $\bh_\ast$ the \textit{hat vector} at $\bx_\ast$ reminiscent of the hat matrix in linear regression. Then, we can describe the bias-variance decomposition for $\ErrR_\bX$ and $\ErrT_\bX$ as
\begin{align*}
	& \ErrR_{\bX} = \sigma_\varepsilon^2 + \E[(\mu_\ast - \bh_\ast^\T \bmu)^2 \vert \bX] + \sigma_\varepsilon^2 \E(\Vert \bh_\ast\Vert^2 \vert \bX), \\ 
	& \ErrT_\bX = \sigma_\varepsilon^2 + \frac{1}{n}\Vert \bmu - \bH\bmu \Vert^2 + \frac{1}{n}\sigma_\varepsilon^2 \trace(\bH^\T \bH - 2 \bH). 
\end{align*}
Here $\sigma_\varepsilon^2$ is the irreducible error variance, the second term is the squared bias, and the third term is the variance. Consequently, \eqref{eq: exp_OptR} becomes
\begin{equation} \label{eq: OptR_decomposition}
	\OptR_\bX = \Delta B_\bX + \frac{2}{n}\sigma_\varepsilon^2 \left[\trace(\bH) + \frac{n}{2}\left(\E(\Vert \bh_\ast\Vert^2 \vert \bX) - \frac{1}{n}\trace(\bH^\T \bH)\right)\right],
\end{equation}
where $\Delta B_\bX = \E[(\mu_\ast - \bh_\ast^\T \bmu)^2 \vert \bX] - \frac{1}{n}\Vert \bmu - \bH\bmu \Vert^2$. We call $\Delta B_\bX$ the \textit{excess bias}, since it measures the extra amount of bias due to making out-of-sample prediction. We see that $\Delta B_\bX$ depends on both the true mean regression function $\mu(\cdot;\bbeta)$ and the modeling procedure $\hat{\bmu}$. The second term in \eqref{eq: OptR_decomposition} is the excess variance, which depends only on the procedure and the distribution of $\bx_\ast$. 

Using the parallel between \eqref{eq: OptF_cov_penalty} and \eqref{eq: OptR_decomposition} in the expected optimism and the degrees of freedom in the Fixed-X setting in \eqref{eq: df_F}, we define the model degrees of freedom under the Random-X setting as
\begin{equation} \label{eq: df_R}
	\dfR = \trace(\bH) + \frac{n}{2}\left(\E(\Vert \bh_\ast\Vert^2 \vert \bX) - \frac{1}{n}\trace(\bH^\T \bH)\right).
\end{equation}
We call $\dfR$ the \textit{predictive model degrees of freedom} as it is a more pertinent measure of model complexity in a genuine predictive setting than the classical one for in-sample prediction. Note that both $\dfR$ and $\dfF$ are defined via the excess variance of a procedure. The following subsection presents some interesting facts about $\dfR$.

\subsection{Properties}

We demonstrate some general properties of $\dfR$. We begin by giving two remarks about the definition.

\begin{remark}[Interpolating models]
	For interpolating models, $\bH = \bI_n$. Then the predictive model degrees of freedom in \eqref{eq: df_R} is simplified to
	\begin{equation}\label{eq: dfR_interpolating_models}
		\dfR = \frac{n}{2} + \frac{n}{2}\E(\Vert \bh_\ast\Vert^2 \vert \bX).
	\end{equation}
	For $\bx_\ast \neq \bx_i$, the hat vector $\bh_\ast$ usually varies across different interpolating models (see Examples \ref{ex: df_wgt_funcs} and \ref{ex: df_vs_bw} for example). Thus, the predictive model degrees of freedom can indeed differentiate interpolating models.
\end{remark}

\begin{remark}[Connection to $\dfF$]
	Since $\dfF = \mathrm{tr}(\bH)$ for linear procedures, we can rewrite \eqref{eq: df_R} as
	\begin{equation} \label{eq: df_connection}
		\dfR = \dfF + \frac{n}{2}\left(\E(\Vert \bh_\ast\Vert^2 \vert \bX) - \frac{1}{n}\trace(\bH^\T \bH)\right).
	\end{equation}
	Thus, the predictive model degrees of freedom $\dfR$ adjusts $\dfF$ with an additional term that accounts for out-of-sample prediction. In particular, if $\bx_\ast$ is drawn from the empirical distribution of $\bx_1,\ldots,\bx_n$ with $\prob(\bx_\ast = \bx_i) = \frac{1}{n}$ for $i=1,\ldots,n$, then the additional term vanishes and we have $\dfR = \dfF$.
\end{remark}

The following proposition provides an interesting representation of $\dfR - \dfF$.
\begin{proposition}[Covariance penalty representation] \label{prop: dfR_cov_representation}
	For a linear procedure, $\dfR - \dfF$ can be expressed as
	$$
		\dfR - \dfF = \frac{n}{2}\sum_{i=1}^n\left[\E\left(\frac{\cov^2(y_i, \hat{\mu}_\ast \vert \bx_\ast, \bX)}{(\sigma_\varepsilon^2)^2} \Bigg\vert \bX\right) - \frac{1}{n}\sum_{j=1}^n \frac{\cov^2(y_i, \hat{\mu}_j \vert \bX)}{(\sigma_\varepsilon^2)^2}\right].
	$$
\end{proposition}

The representation above extends the covariance penalty approach to the Random-X setting. It also reveals a key difference between such an approach used for in-sample prediction and that for out-of-sample prediction. For the former, a reasonable model fit to training data produces fitted values positively correlated with the corresponding observations. As a result, the model degrees of freedom $\dfF$ defined as the sum of covariances between the observations and in-sample predictions should be positive.

However, for the latter case, predictions on new feature values could be negatively correlated with the observations. For example, consider fitting a univariate linear regression model using least squares method without intercept on $\lbrace(x_i,y_i)\in\R\times\R \vert i = 1,\ldots,n \rbrace$. For a given $x_\ast\in\R$, we have $\hat{\mu}_\ast = \sum_{i=1}^n h_{\ast,i} y_i$, where $h_{\ast,i} = \frac{x_i x_\ast}{\sum_{i=1}^n x_i^2}$. Then for any $x_i$ such that $x_i x_\ast < 0$,
$$
	\cov(\hat{\mu}_\ast, y_i\vert \bX,x_\ast) = \sigma_\varepsilon^2 h_{\ast,i} < 0.
$$
The representation in the above proposition involves squared covariances reflecting the change in the relation between a response and its prediction for an independent test case.

For a modeling procedure $\hat{\bmu}$, \cite{ye1998measuring} defined the generalized degrees of freedom (GDF) as
$$
	\mathrm{GDF}(\hat{\bmu}) = \sum_{i=1}^n \frac{\partial \E(\hat{\mu}_i \vert \bX) }{\partial \mu_i}.
$$
It turns out that we can express $\dfR - \dfF$ in a similar way as the GDF.
\begin{proposition}[GDF representation] \label{prop: dfR_gdf_representation}
	For a linear procedure, $\dfR - \dfF$ can be expressed as
	$$
		\dfR - \dfF = \frac{n}{2} \sum_{i=1}^n\left[\E\left(\left(\frac{\partial \E(\hat{\mu}_\ast \vert \bx_\ast,\bX)}{\partial \mu_i}\right)^2 \Bigg\vert \bX\right) - \frac{1}{n} \sum_{j=1}^n \left(\frac{\partial \E(\hat{\mu}_j \vert \bX)}{\partial \mu_i}\right)^2\right].
	$$
\end{proposition}

\subsection{Examples}
We provide some examples where the predictive model degrees of freedom can be evaluated explicitly.

\begin{example}[Ridge regression] \label{ex: ridge_df}
	Assume $\bX \in \R^{n \times p}$. For the ridge regression problem
	\begin{equation}\label{eq: ridge_regression}
		\min_{\mathbf{b}} \frac{1}{n}\Vert \by - \bX \mathbf{b} \Vert^2 + \lambda \Vert \mathbf{b} \Vert^2, \quad \lambda > 0,
	\end{equation}
	we have $\bH = \bX (\bX^\T \bX + \lambda \bI)^{-1} \bX^\T$ and $\bh_\ast = \bX (\bX^\T \bX + \lambda \bI)^{-1} \bx_\ast$. Assume $\bX^\T \bX$ has the spectral decomposition $\mathbf{U} \bm{\bm{\Omega}} \mathbf{U}^\T$, where $\bm{\Omega} = \mathrm{diag}(\omega_1,\ldots, \omega_d)$. Let $\bSigma = \var(\bx_i)$ and $\bV = \mathbf{U}^\T \bSigma  \mathbf{U} = (v_{ij})$. Then it is easy to show that
	$$
		\dfR(\lambda) = \sum_{j=1}^d \frac{\omega_j^2 + (2\lambda + n v_{jj})\omega_j}{2(\omega_j+\lambda)^2} \,\text{ and }\; \frac{\partial \dfR}{\partial \lambda} = -\sum_{j=1}^n\frac{\lambda + 2n v_{jj} \omega_j}{(\omega_j+\lambda)^3} < 0.
	$$
	This implies that $\dfR$ is a decreasing function of $\lambda$. As shown in Figure \ref{fig: ex_ridge_df}, $\dfR$ and $\dfF$ exhibit very different relationships under the scenarios of $p < n$ and $p > n$ for small $\lambda$. As $\lambda \to 0$, the ridge regression estimator converges to the ordinary least squares estimator when $p \leq n$ and the minimum-norm least squares solution when $p > n$. Interestingly in this setting, the predictive model degrees of freedom for a model with more variables could be even smaller than that with fewer. We will study this phenomenon further for least squares method in Section \ref{sec: ls_method}. As $\lambda$ increases, the difference between $\dfR$ and $\dfF$ diminishes to 0 in both scenarios. 
	\begin{figure}[t!]
		\centering
		\includegraphics[scale = 0.5]{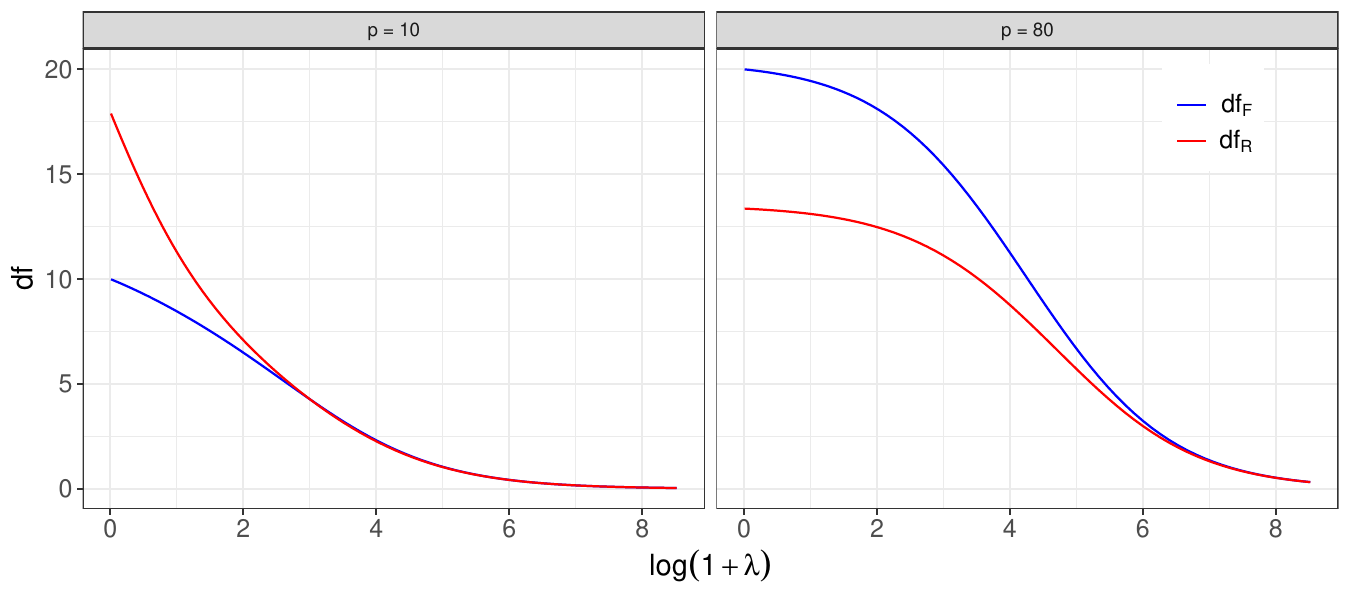}
		\caption{Comparison of $\dfR$ and $\dfF$ as a function of $\lambda$ in ridge regression with $n = 20$, $p = 10$ (left) and $n = 20$, $p = 80$ (right). $\bx_1, \ldots, \bx_n$ are assumed to be from $\mathcal{N}(\mathbf{0}, \bI_p)$.}
		\label{fig: ex_ridge_df}
	\end{figure}
\end{example}

\begin{example}\label{ex: df_wgt_funcs}
	Let $x_\ast, x_1, \ldots, x_n$ be i.i.d. from a distribution on a finite interval $[a,b]$ with continuous and positive density. Without loss of generality, assume $a<x_1<\cdots<x_n<b$. Let $z_{\ast,i} = \frac{x_\ast-x_i}{x_{i+1}-x_i}$. Consider the following interpolating scheme:
	$$
		\hat{\mu}_\ast = 
		\begin{cases}
			y_1, & a\leq x_\ast < x_1,\\
			K(z_{\ast,i}) y_i + (1 - K(z_{\ast,i})) y_{i+1}, & x_i\leq x_\ast < x_{i+1},\; i=1,\ldots,n-1,\\
			y_n, & x_n\leq x_\ast \leq b,
		\end{cases}
	$$
	where $K \mathpunct{:} [0,1]\to [0,1]$ is a nonincreasing weight function with $K(0)=1$ and $K(1)=0$. We now consider four choices of $K$:
	\begin{enumerate}
		\item[I.] Constant: $K(z) = \mathbf{1}_{\lbrace z<\frac{1}{2}\rbrace}$. This corresponds to 1-nearest neighbor regression.
		
		\item[II.] Linear: $K(z) = 1-z$.
		
		\item[III.] Quadratic: $K(z) = 1-z^2$.
		
		\item[IV.] Cosine: $K(z) = \cos\left(\frac{\pi}{2}z\right)$.
	\end{enumerate}
	The four interpolating schemes are illustrated in the left panel of Figure \ref{fig: ex_interpolant_illustration}. For $x_\ast \in [a,b]$, the corresponding hat vector $\bh_\ast$ is given by
	$$
		\bh_\ast = 
		\begin{cases}
			\mathbf{e}_1, & a \leq x_\ast < x_1,\\
			K(z_{\ast,i})\mathbf{e}_i + [1-K(z_{\ast,i})]\mathbf{e}_{i+1}, & x_i \leq x_\ast < x_{i+1}, \; i=1,\ldots,n-1,\\
			\mathbf{e}_n, & x_n \leq x_\ast \leq b,
		\end{cases}
	$$
	where $\mathbf{e}_i \in \R^n$ is the $i$th standard basis vector in the Euclidean space. The predictive model degrees of freedom can then be evaluated. For more details about the derivation, see Appendix \ref{det: df_wgt_funcs}. In Table \ref{tab: df_wgt_funcs}, we show the ratio of the predictive model degrees of freedom to the sample size $n$ as $n \to \infty$. Among the interpolating models with these four schemes, model I is the most ``complex'' with its predictive model degrees of freedom equal to $n$. One interpretation of it is that the constant weight scheme partitions the entire feature space $[a,b]$ into $n$ disjoint neighborhoods, with each one associated with a particular predicted value. Model II-IV are ``simpler'' than model I in terms of the predictive model degrees of freedom due to the implementation of averaging schemes. Model II is the simplest among the three. We may relate this to the geometric fact that a line segment makes the shortest path between two points. Model III and IV are similar in their degrees of freedom because $\cos(x_\ast-x_i)\approx 1-\frac{1}{2}(x_\ast-x_i)^2$ when $x_\ast$ is close to $x_i$.
	\begin{figure}[!t]
		\centering
		\includegraphics[scale = 0.45]{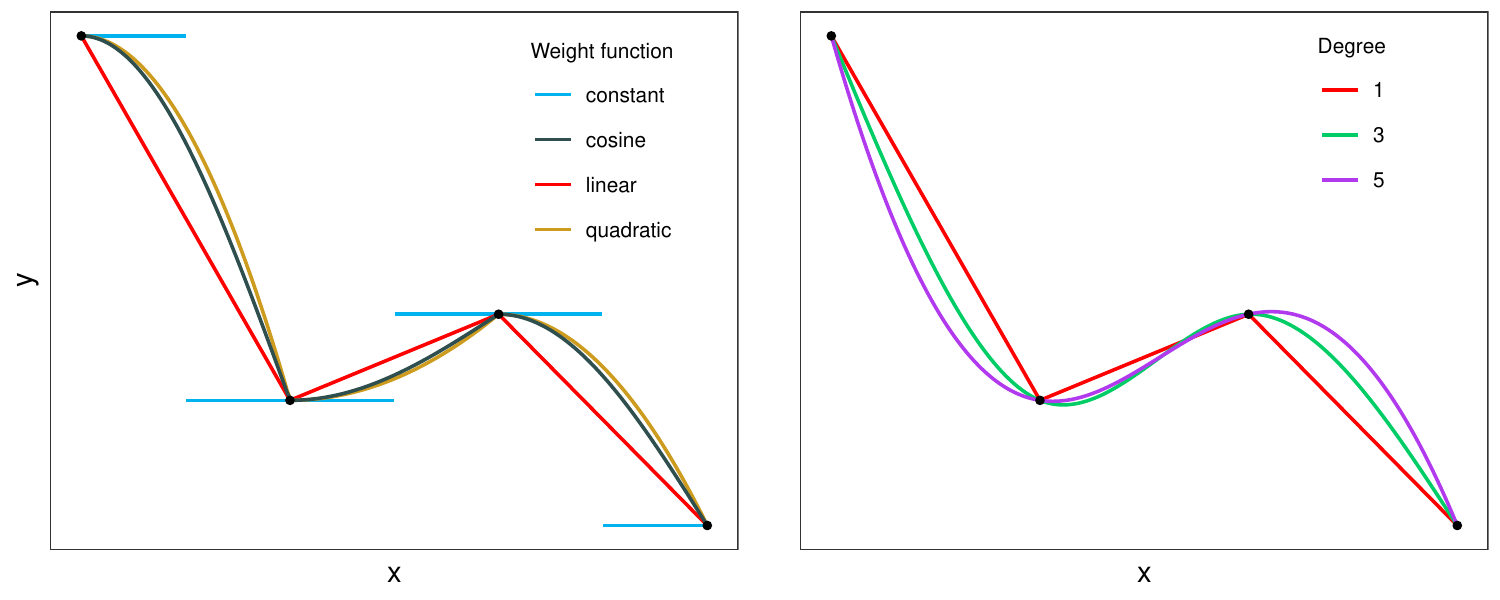}
		\caption{An illustration of interpolation with four weight schemes in Example \ref{ex: df_wgt_funcs} (left) and interpolating polynomial splines of degree 1, 3 and 5 in Example \ref{ex: spline_df} (right).}
		\label{fig: ex_interpolant_illustration}
	\end{figure}
	\begin{table}[h]
		\centering
		\caption{Ratio of $\dfR$ to $n$ for interpolating schemes I - IV as $n \to \infty$.}
		\label{tab: df_wgt_funcs}
		\begin{tabular}{c|*{4}{c}}
			\hline
			Weight &  Constant  & Linear & Quadratic & Cosine\\ \hline
			$\lim_{n \to \infty}\dfR/n$ & $1$ & $0.833$ & $0.867$ & $0.863$\\ \hline
		\end{tabular}
	\end{table}
	
\end{example}

\begin{example}[Interpolating splines] \label{ex: spline_df}
	We study the predictive model degrees of freedom for univariate interpolating polynomial splines. For simplicity, we assume $x_\ast \sim \mathrm{Uniform}(0,1)$ and $0 = x_1 < \cdots < x_n = 1$. For $s = 1,2\ldots$, let $\mathcal{M}^{s}[0,1] = \lbrace \mu \vert \int_{0}^{1} \mu^{(s)} dx < \infty \rbrace$ be the model space. Then, interpolating polynomial splines of degree $2s - 1$ can be defined. See \cite{gu2013smoothing} for more technical details. For any $x_\ast \in (0,1)$, the hat vector $\bh_\ast$ can also be obtained. See Appendix \ref{det: spline_df} for the derivation. Monte Carlo method is then used to approximate $\E(\Vert \bh_\ast \Vert^2 \vert x_1,\ldots,x_n)$ and $\dfR$. As an example, we set $n=21$ and $x_i = (i-1)/20$. We then estimate the ratio of $\dfR$ to $n$ based on 10,000 random samples from $\mathrm{Uniform}(0,1)$. As shown in Table \ref{tab: spline_df}, $\dfR$ increases with the polynomial degree and can exceed $n$. Note that the linear interpolating spline ($s=1$) is the same as the interpolant with a linear weight function in Example \ref{ex: df_wgt_funcs} over $[x_1, x_n]$, which leads to the same ratio of $0.833$.
	\begin{table}[h]
		\centering
		\caption{Ratio of $\dfR$ to $n$ for polynomial interpolating splines.}
		\label{tab: spline_df}
		\begin{tabular}{c|*{6}{c}}
			\hline
			Degree & 1 & 3 & 5 & 7 & 9 & 11\\ \hline
			$\dfR/n$ & $0.833$ & $0.932$ & $0.960$ & $0.991$ & 1.056 & 1.618\\ \hline
		\end{tabular}
	\end{table}

\end{example}

\begin{example}[Local constant smoother] \label{ex: df_vs_bw}
	In this example, we look into how the predictive model degrees of freedom changes as a function of the bandwidth $\omega$ for local constant smoother under the same setting as in Example \ref{ex: df_wgt_funcs}. Let $L_i=x_{i+1}-x_i$, $\bar{L}=\max_i L_i$ and $\ubar{L}=\min_i L_i$. Assume $\frac{1}{2}\bar{L} < \ubar{L}$. Consider the smoother
	$$
		\hat{\mu}_\ast = 
		\begin{dcases}
			y_1, & a\leq x_\ast < x_1,\\
			\frac{\sum_{i=1}^n \mathbf{1}_{\left\lbrace\vert x_\ast - x_i\vert \leq \omega\right\rbrace} y_i}{\sum_{i=1}^n\mathbf{1}_{\left\lbrace\vert x_\ast - x_i\vert\leq \omega\right\rbrace}}, & x_1 \leq x_\ast < x_n,\\
			y_n, & x_n\leq x_\ast \leq b,
		\end{dcases}
	$$
	which interpolates the training data when $\frac{1}{2}\bar{L} \leq \omega <\ubar{L}$. Define $x_0 = a - \omega$ and $x_{n+1} = b + \omega$. Note that, for $x_\ast \in [a,b]$, the $\omega$-neighborhood of $x_\ast$ may contain either one or two $x_i$'s in the training data.
	When $x_{i-1} + \omega < x_\ast < x_{i+1} - \omega$, only $x_i$ is in the neighborhood, whereas both $x_i$ and $x_{i+1}$ are in when $x_{i+1} - \omega \leq x_\ast \leq x_{i} + \omega$. Then, the hat vector for $x_\ast \in [a,b]$ is given by
	$$
		\bh_\ast =
		\begin{cases}
			\mathbf{e}_i, & x_{i-1} + \omega < x_\ast < x_{i+1} - \omega, \;i = 1,\ldots,n\\
			\frac{1}{2}\mathbf{e}_i + \frac{1}{2}\mathbf{e}_{i+1}, & x_{i+1} - \omega \leq x_\ast \leq x_{i} + \omega, \;i = 1,\ldots,n-1.
		\end{cases}
	$$
	Assume $x_\ast \sim \mathrm{Uniform}(a,b)$. In \ref{det: df_vs_bw}, we show that
	$$
		\dfR(\omega) =  n + \frac{n(x_n - x_1)}{4(b-a)} - \frac{n(n-1)}{2(b-a)}\omega.
	$$
	Thus, $\dfR$ decreases linearly in $\omega$.
	
	Figure \ref{fig: ex_df_vs_bw} compares $\dfR$ and $\dfF$ as a function of the bandwidth $\omega$ when $x_1,\ldots,x_n$ are equally spaced ($\bar{L} = \ubar{L} \equiv L$) with $x_1 = a$ and $x_n = b$. We see that $\dfR$ can indeed differentiate interpolating models when $\frac{L}{2} \leq \omega < L$. For $\omega \geq L$, the smoother does not interpolate the training data. In particular, when $L \leq \omega < x_n-x_1$, $\dfR$ is strictly decreasing and piecewise linear in $\omega$ while $\dfF$ is piecewise constant. When $\omega\geq x_n-x_1$, it can be shown that $\dfR\xrightarrow{\rm P} 1 = \dfF$ as $n\to\infty$.
	\begin{figure}[t!]
		\centering
		\includegraphics[scale = 0.6]{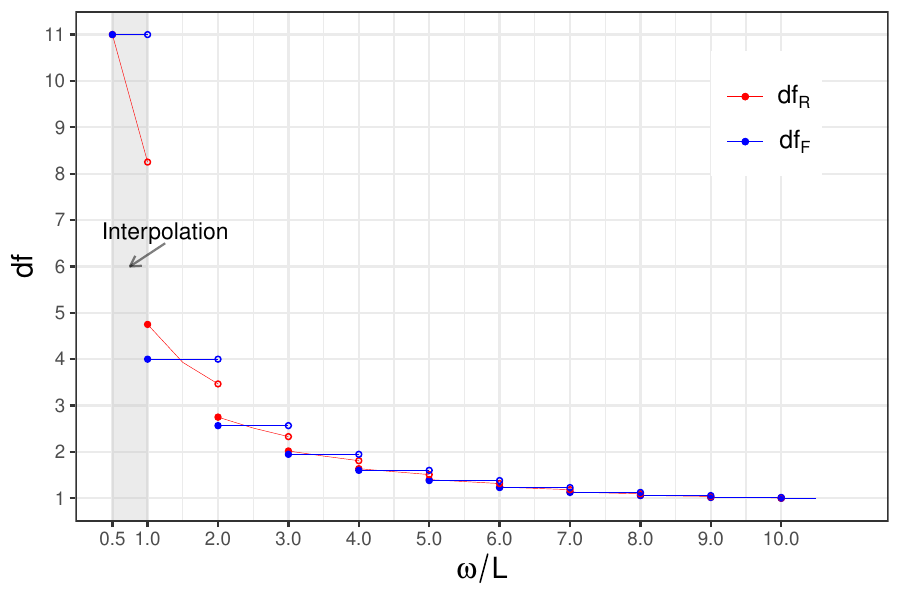}
		\caption{Comparison of $\dfR$ and $\dfF$ as a function of the bandwidth for univariate local constant smoother when $x_1,\ldots,x_n$ ($n=11$) are equally spaced ($\protect\ubar{L} = \bar{L} = L$) with $x_1 = a$ and $x_n = b$.}
		\label{fig: ex_df_vs_bw}
	\end{figure}
\end{example}

\section{Least Squares Method in Subset Regression} \label{sec: ls_method}
In this section, we provide an in-depth analysis of the predictive model degrees of freedom for the least squares method in the context of subset regression. Suppose that  $\cS$ is a subset of variable indices $\mathcal{D} \coloneqq \left\lbrace 1,\ldots, d\right\rbrace$ with $\vert \cS \vert = p$. For $i = 1,\ldots, n$, let $\bx_{i,\cS} = (x_{ij})_{j\in \cS}$ denote the subvector of $\bx_i$ corresponding to $\cS$ and $\bSigma_\cS = \var(\bx_{i,\cS})$, the submatrix of $\bSigma$ for the variables in $\cS$. Let $\bx_{(j)}$ be the $j$th column of $\bX$ and $\bX_\cS = (\bx_{(j)})_{j \in \cS} \in \R^{n \times p}$.

Let $\hat{\bbeta}(\cS)$ be a least squares estimator using $\bX_\cS$ as a design matrix.  When $p \leq n$ (underparameterized regime), we assume that $\bX_\cS$ has full column rank and apply the ordinary least squares method to get
$$
	\hat{\bbeta}(\cS) = (\bX_\cS^\T \bX_\cS)^{-1} \bX_\cS^\T\by.
$$
When $p > n$ (overparameterized regime), we consider the minimum-$\ell_2$-norm least squares method and obtain $\hat{\bbeta}(\cS)$ by solving the optimization problem:
$$
	\min_{\mathbf{b} \in \R^p} \Vert \mathbf{b} \Vert_2^2, \quad \text{subject to } \by = \bX_\cS \mathbf{b}.
$$
We assume that $\bX_\cS$ has full row rank in this case so that the solution is unique and can be explicitly expressed as
$$
	\hat{\bbeta}(\cS) = \bX_\cS^\T (\bX_\cS \bX_\cS^\T)^{-1} \by.
$$

For each $\bx_\ast \in \R^d$, using the identity $\hat{\mu}_\ast = \bx_{\ast, \cS}^\T \hat{\bbeta}(\cS) = \bh_\ast^\T \by$, we can define the hat vector $\bh_\ast$ for $\bx_\ast$ as 
\begin{equation} \label{eq: coef_ls}
	\bh_\ast = 
	\begin{dcases}
		\bX_\cS(\bX_\cS^\T \bX_\cS)^{-1}\bx_{\ast,\cS}, & p\leq n,\\
		(\bX_\cS \bX_\cS^\T)^{-1} \bX_\cS \bx_{\ast,\cS}, & p > n.
	\end{dcases}
\end{equation}
Under the assumptions A0---A3, it is then easy to obtain the predictive model degrees of freedom for the subset regression model as
\begin{equation}\label{eq: dfR_ls}
	\dfR(\cS) = 
	\begin{dcases}
		\frac{p}{2} + \frac{n}{2}\trace[(\bX_\cS^\T \bX_\cS)^{-1} \bSigma_\cS], & p\leq n,\\
		\frac{n}{2} + \frac{n}{2}\trace[\bX_\cS^\T(\bX_\cS \bX_\cS^\T)^{-2} \bX_\cS \bSigma_\cS], & p > n.
	\end{dcases}
\end{equation}
In the following analysis, we will use $\dfR(\cS)$, $\dfR(\bX_\cS)$ and $\dfR(p)$ interchangeably. It is also worth noting that, while \eqref{eq: dfR_ls} is derived under the assumption that $\E(\bx_\ast) = \mathbf{0}$, it can be generalized to an arbitrary mean $\bm{\nu} \in \R^d$ by simply replacing the covariance matrix $\bSigma_\cS$ with the second moment matrix, $\bSigma_\cS + \bm{\nu}_\cS \bm{\nu}_\cS^\T$.

In Sections \ref{subsec: monotonicity} and \ref{subsec: df_asymptotics}, we first study the monotonicity and asymptotics of the predictive model degrees of freedom $\dfR$. In Section \ref{subsec: double_descent}, we revisit the double descent phenomenon and show that it can be reconciled with the classical theory of the bias-variance trade-off by parameterizing the risk with the predictive model degrees of freedom.

\subsection{Monotonicity} \label{subsec: monotonicity}
In this subsection, we discuss the monotonicity of the predictive model degrees of freedom as a function of $p$ in the underparameterized and overparameterized regimes separately. In particular, we will show that $\dfR$ is strictly increasing when $p < n$ and generally decreasing when $p > n$.

\subsubsection{Underparameterized Regime}

As shown in \eqref{eq: dfR_ls}, $\trace[(\bX_\cS^\T \bX_\cS)^{-1} \bSigma_\cS]$ is pivotal to the predictive model degrees of freedom when $p < n$. We first state a useful result in linear algebra regarding this trace term with a rank-one change. The proof can be found in \ref{pf: useful_lemma}. 
\begin{lemma} \label{lemma: trace_monotonicity_under}
	Let $\bX \in\R^{n\times p}$, $\mathbf{w} \in \R^n$ and $\tilde{\bX} = (\bX, \mathbf{w})$. Assume $\mathrm{rank}(\bX) = p < n$ and $\mathrm{rank}(\tilde{\bX}) = p + 1$. Let $\bB \in\R^{p\times p}$ be symmetric and positive definite. For $\ba \in \R^p$ and $b^2 > 0$, define
	$$
		\tilde{\bB} =
		\begin{pmatrix}
			\bB &  \ba\\
			\ba^\T & b^2
		\end{pmatrix}.
	$$
	If $\tilde{\bB}\in\R^{(p+1)\times (p+1)}$ is positive semi-definite, then
	\begin{equation} \label{eq: df_ls_inequality}
		\trace[(\tilde{\bX}^\T \tilde{\bX})^{-1}\tilde{\bB}] \geq \trace[(\bX^\T \bX)^{-1}\bB].
	\end{equation}
	In particular, if $\tilde{\bB}$ is positive definite, the inequality above is strict.
\end{lemma}

As a direct application of the lemma, the following result characterizes the monotonicity of $\dfR$ in subset size $p$ in subset regression when $p < n$. 
\begin{theorem} \label{thm: df_ls_monotonicity}
	Let $\cS_1$ and $\cS_2$ be two subsets of $\mathcal{D}$. If $\cS_1 \subset \cS_2$, $\vert \cS_2 \vert \leq n$, and $\bSigma_{\cS_2}$ is positive definite, then
	\begin{equation} \label{eq: df_ls_monotonicity}
		\dfR(\cS_1) < \dfR(\cS_2).
	\end{equation}
\end{theorem}
\begin{proof}
	Let $j \in \mathcal{D} \backslash \cS_1$. It suffices to show \eqref{eq: df_ls_monotonicity} for $\cS_2 = \cS_1 \cup \left\lbrace j \right\rbrace$. Since $\bSigma_{\cS_2}$ is positive definite, by Lemma \ref{lemma: trace_monotonicity_under}, we have
	$$
		\trace[(\bX_{\cS_2}^\T \bX_{\cS_2})^{-1} \bSigma_{\cS_2}] > \trace[(\bX_{\cS_1}^\T \bX_{\cS_1})^{-1} \bSigma_{\cS_1}],
	$$
	which implies \eqref{eq: df_ls_monotonicity} immediately.
\end{proof}

Theorem \ref{thm: df_ls_monotonicity} says that, for a sequence of nested subsets $\cS_1 \subset \cdots \subset \cS_n$ such that $\vert \cS_p \vert = p$, $\dfR$ is strictly increasing in $p$ as long as $\bSigma_{\cS_n}$ is positive definite. Further, the following result gives the increment in $\dfR$ when a new variable is added.
\begin{theorem}\label{thm: dfR_increment}
	Let $\cS_1$ be a subset of $\mathcal{D}$ with $\vert \cS_1 \vert < n$. For $j \in \mathcal{D}\backslash \cS_1$, let $\cS_2 = \cS_1 \cup \lbrace j \rbrace$ and
	$$
		\bSigma_{\cS_2} =
		\begin{pmatrix}
			\bSigma_{\cS_1} & \bSigma_{\cS_1,j}\\
			\bSigma_{\cS_1,j}^\T & \sigma_j^2
		\end{pmatrix}.
	$$
	Assume that $\bX_{\cS_2}$ has full column rank, and $\bSigma_{\cS_2}$ is positive definite. Define
	$$
		\bm{\zeta} = \frac{\bx_{(j)} - \bX_{\cS_1} \bSigma_{\cS_1}^{-1 } \bSigma_{\cS_1, j}}{\sqrt{\sigma_j^2 - \bSigma_{\cS_1,j}^\T \bSigma_{\cS_1}^{-1} \bSigma_{\cS_1,j}}}.
	$$ 
	Then
	$$
		\dfR(\cS_2) - \dfR(\cS_1) = \frac{1}{2} + \frac{n}{2} \frac{\bm{\zeta}^\T \mathbf{C} \bm{\zeta} + 1}{\bm{\zeta}^\T (\bI_n - \bH) \bm{\zeta}}.
	$$
	where $\mathbf{C} = \bX_{\cS_1} (\bX_{\cS_1}^\T \bX_{\cS_1})^{-1} \bSigma_{\cS_1} (\bX_{\cS_1}^\T \bX_{\cS_1})^{-1}\bX_{\cS_1}^\T$ and $\bH = \bX_{\cS_1}(\bX_{\cS_1}^\T \bX_{\cS_1})^{-1} \bX_{\cS_1}^\T$.
\end{theorem}

Note that, for a given $j \in \mathcal{D}\backslash \cS$, $\bm{\zeta}$ is comprised of the normalized residuals from regressing the $j$th variable on the existing variables in $\cS_1$ as in partial regression. Theorem \ref{thm: dfR_increment} points out a key difference between the classical model degrees of freedom $\dfF$ and the predictive model degrees of freedom $\dfR$. Unlike $\dfF$ that always increases by 1 whenever a new variable is added, the increment in $\dfR$ depends on both the current design matrix and the new variable that is to be added, and thus varies from sample to sample. The following remark provides some more insights about $\dfR$ and its increment when $\bx_i$'s are multivariate normal. 

\begin{remark}[Normal covariates]
	Let $\cS$ be a subset of $\mathcal{D}$ with $\vert \cS \vert = p < n-1$. When $\bx_i$'s are multivariate normal, $(\bSigma_\cS^{-1/2} \bX_\cS^\T \bX_\cS \bSigma_\cS^{-1/2})^{-1}$ follows an inverse-Wishart distribution with scale matrix $\bI_p$ and degrees of freedom $n$ \citep{mardia1979multivariate}, which implies that
	$$
		\E(\trace[(\bX_\cS^\T \bX_\cS)^{-1}\bSigma_\cS]) = \E(\trace[(\bSigma_\cS^{-1/2}\bX_\cS^\T \bX_\cS \bSigma_\cS^{-1/2})^{-1}]) = \frac{p}{n-p-1}.
	$$
	Therefore,
	\begin{equation}\label{eq: dfR_normal_expectation}
		\E[\dfR(\cS)] = \frac{p}{2} + \frac{n}{2}\E(\trace[(\bX_\cS^\T \bX_\cS)^{-1}\bSigma_\cS]) = \frac{p}{2}\left(1+\frac{n}{n-p-1}\right).
	\end{equation}
	On the one hand, this inspires us to approximate $\dfR$ by 
	\begin{equation}\label{eq: df_approx_norm}
		\dfR \approx \frac{p}{2}\left(1+\frac{n}{n-p-1}\right)
	\end{equation}
	when features are jointly normal. On the other hand, the expected increment in $\dfR$ is given by
	$$
		\E[\dfR(p+1)] - \E[\dfR(p)] = \frac{1}{2} + \frac{n (n-1)}{2 (n-p-1)(n-p-2)},
	$$
	which is strictly increasing in $p$. Thus, on average, the increment in $\dfR$ grows as more variables are added.
\end{remark}
\vspace{8pt}

For two nested ordinary least squares models, Theorem \ref{thm: df_ls_monotonicity} assumes that variables in the smaller model are a proper subset of those in the larger one. In fact, similar results still hold for two models where the column space of one model is contained in that of the other. 

\begin{theorem} \label{thm: df_ls_monotonicity_general}
	Assume that $\bX \in\R^{n\times p}$ has full column rank, and $\bSigma = \var(\bx_i)$ is positive definite. For $s \leq p$, let $\mathbf{U} \in \R^{p \times s}$ be an arbitrary coefficients matrix of full column rank for linear combinations and define $\mathbf{Z} = \bX\mathbf{U}$. Then, for the ordinary least squares models based on $\bX$ and $\mathbf{Z}$ respectively,
	$$
		\dfR(\mathbf{Z}) \leq \dfR(\bX),
	$$
	and the equality holds if and only if $s = p$.
\end{theorem}

The above theorem allows us to compare the complexity of two models lying in two nested linear spaces but with different bases. The following example regards principal component regression for illustration of the fact.

\begin{example}\label{ex: pcr}
	Let $\bX\in\mathbb{R}^{n\times p}$ be the design matrix ($p\leq n)$. Assume that $\bX^\T \bX$ has spectral decomposition $\mathbf{U} \bm{\Lambda} \mathbf{U}^\T$, where $\bm{\Lambda}=\mathrm{diag}(\lambda_1,\ldots,\lambda_p)$ with $\lambda_1 \geq \cdots \geq \lambda_p \geq 0$ and $\mathbf{U} = (\mathbf{u}_1,\ldots,\mathbf{u}_p)$ contains the corresponding eigenvectors. For $k = 1,\ldots, p$, let $\mathbf{U}_k = (\mathbf{u}_1,\ldots,\mathbf{u}_k)$ and $\mathbf{Z}_k = \bX \mathbf{U}_k$. Then $\mathbf{Z}_k\in\mathbb{R}^{n\times k}$ contains the first $k$ principal components of $\bX$. By Lemma \ref{lemma: trace_monotonicity_under}, we have
	$$
		\trace[(\mathbf{Z}_k^\T \mathbf{Z}_k)^{-1}\mathbf{U}_k^\T \bSigma \mathbf{U}_k] \leq \trace[(\mathbf{Z}_p^\T \mathbf{Z}_p)^{-1}\mathbf{U}^\T \bSigma \mathbf{U}] = \trace[(\bX^\T \bX)^{-1}\bSigma].
	$$
	Let $\mathbf{z}_{k,i}$ denote the $i$th row of $\mathbf{Z}_k $. Note that $\mathbf{U}_k$ depends on $\bX$ while the coefficients matrix $\mathbf{U}$ in Theorem \ref{thm: df_ls_monotonicity_general} doesn't. Thus, we don't have $\var(\mathbf{z}_{k,i}) = \mathbf{U}_k^\T \bSigma \mathbf{U}_k$ generally. But if $n$ is large relative to $p$, $\mathbf{U}_k^\T \bSigma \mathbf{U}_k$ provides a good estimator of $\var(\mathbf{z}_{k,i})$. Then
	$$
		\dfR(\mathbf{Z}_k) \approx \frac{k}{2} + \frac{n}{2}\trace[(\mathbf{Z}_k^\T \mathbf{Z}_k)^{-1} \mathbf{U}_k^\T \bSigma \mathbf{U}_k] \leq \frac{p}{2} + \frac{n}{2}\trace[(\bX^\T \bX)^{-1}\bSigma] = \dfR(\bX).
	$$
\end{example}

\subsubsection{Overparameterized Regime}
When $p>n$, the minimum-norm least squares method is used for estimation. Due to regularization, the model space is implicitly constrained, and this makes the sequence of constrained model spaces no longer nested beyond the interpolation threshold. As a result, $\dfR$ is not necessarily monotone in the subset size $p$. However, when the features are independent and isotropic, $\dfR$ is shown to be decreasing in $p$. 
\begin{theorem} \label{thm: df_ls_monotonicity_over}
	Assume $\var(\bx_\ast) = \sigma_x^2 \bI_d$ for some $\sigma_x > 0$. Let $\cS_1 \subset \cS_2 \subseteq \mathcal{D}$ with $\vert \cS_1\vert \geq n$. Then
	\begin{equation} \label{eq: df_ls_monotonicity_over}
		\dfR(\cS_2) \leq \dfR(\cS_1).
	\end{equation}
	In particular, if there exists $j \in \cS_2 \backslash \cS_1$ such that $\bx_{(j)}$ is not in the null space of $(\bX_{\cS_1} \bX_{\cS_1}^\T)^{-1}$, then the inequality holds strictly.
\end{theorem}

\begin{remark}[Isotropic normal covariates] \label{rmk: df_over_normal}
	Assume $\bx_i \in \R^p$ with $p>n+1$. If $\bx_i \sim \mathcal{N}(\mathbf{0},\bI_p)$, $(\bX \bX^\T)^{-1}$ follows an inverse-Wishart distribution with degrees of freedom $p$ and scale matrix $\bI_n$. Then $\E[\trace((\bX \bX^\T)^{-1})] = \frac{n}{p-n-1}$ and 
	$$
		\E(\dfR) = \frac{n}{2} + \frac{n}{2} \E[\trace((\bX \bX^\T)^{-1})] = \frac{n(p-1)}{2(p-n-1)},
	$$
	which is clearly decreasing in $p$.
\end{remark}

For a general covariance matrix, we conduct a simulation study as follows. Assume all the variables are normalized to have unit variance so that $\bSigma = \bm{\rho}$, where $\bm{\rho}$ is the correlation matrix of $\bx$. Without loss of generality, assume that the first $p$ columns of $\bX$, denoted by $\bX_p = (\bx_{(1)},\ldots,\bx_{(p)})$, are used to fit the model. Let $\bm{\rho}_p$ be the $p$th leading principal submatrix of $\bm{\rho}$. Then, as a function of $p$, we have
$$
	\dfR(p)=\frac{n}{2}\trace[\bX_p^\T(\bX_p \bX_p^\T)^{-2}\bX_p \bm{\rho}_p] + \frac{n}{2}.
$$
In the simulation, we take $n = 20$ and $d = 100$. We first randomly generate $M = 10,000$ correlation matrices using the method described in \cite{makalic2020efficient}. Then we draw $\bx_i$'s from $\mathcal{N}(\mathbf{0}, \bm{\rho})$. Let $\dfR^{(m)}(p)$ be the predictive model degrees of freedom based on the $m$th correlation matrix. We find that
$$
	\frac{1}{M(d-n)}\sum_{m=1}^{M}\sum_{p = n}^{d-1} \mathbf{1}_{\lbrace\dfR^{(m)}(p+1) \leq \dfR^{(m)}(p)\rbrace} = 0.9839,
$$
which suggests that \eqref{eq: df_ls_monotonicity_over} is still highly likely to hold even with correlated features. Presumably, this has to do with the use of minimum-norm least squares solution. In fact, adding more variables beyond the interpolation threshold has similar effect as imposing more regularization on the least squares problem. An intuitive explanation for this is provided by \cite{hastie2019surprises}. In general, with more variables included in the model, the components of $\hat{\beta}$ can be redistributed and reduced to achieve a smaller $\ell_2$-norm of $\hat{\bbeta}$. Thus, model complexity generally decreases as $p$ increases.

\subsection{Asymptotics} \label{subsec: df_asymptotics}
Let $\lambda_{\text{min}}(\bSigma)$ and $\lambda_{\text{max}}(\bSigma)$ be the smallest and largest eigenvalues of $\bSigma$ respectively. \cite{hastie2019surprises} showed that, for any positive definite $\bSigma\in\mathbb{R}^{p\times p}$ whose spectral distribution $F_{\bSigma}$ converges weakly to a measure $\mathcal{P}$ as $n$ and $p\to\infty$, if there exist $c_1$ and $c_2$ such that $0 < c_1 \leq \lambda_{\text{min}}(\bSigma) \leq \lambda_{\text{max}}(\bSigma) \leq c_2$ for every $p$, then as $n$ and $p\to\infty$ and $\frac{p}{n} \to\gamma$,
\begin{align}
	&\text{for $\gamma < 1$: } \trace[(\bX^\T \bX)^{-1}\bSigma] \xrightarrow{\rm a.s.} \frac{\gamma}{1-\gamma},\label{eq: tr_under_limit}\\
	&\text{for $\gamma > 1$: } \trace[\bX^\T(\bX \bX^\T)^{-2}\bX\bSigma] \xrightarrow{\rm a.s.} \lim_{z\to 0^+}\frac{v_{F_\mathcal{P}, \gamma}'(-z)}{v_{F_\mathcal{P},\gamma}^2(-z)} - 1, \label{eq: tr_over_limit}
\end{align}
where $v_{F_\mathcal{P}, \gamma}$ is the companion Stieltjes transform of the limiting spectral distribution $F_{\mathcal{P},\gamma}$ given by the Marchenko-Pastur theorem \citep{marchenko1967distribution}. 

The results above can be used to approximate $\dfR$ when $p$ and $n$ are large. They could also be used to examine the behavior of $\dfR$ under different orderings of variables. When $p < n$, we can naively substitute $\frac{p}{n}$ for $\gamma$ in \eqref{eq: tr_under_limit}. This yields
\begin{equation}\label{eq: df_under_approx}
	\dfR \approx  \frac{p}{2} + \frac{n}{2}\cdot\frac{p}{n-p} = \frac{p}{2}\left(1+\frac{n}{n-p}\right),
\end{equation}
which is asymptotically equivalent to (\ref{eq: df_approx_norm}) derived under the normality assumption. The above approximation also implies that $\dfR$ doesn't depend on the order in which variables are added to the model.

When $\gamma > 1$, it is usually not easy to write $v_{F_\mathcal{P},\gamma}$ explicitly for an arbitrary $\bSigma$. However, when $\bSigma = (1-\rho) \bI_d + \rho \mathbf{1} \mathbf{1}^\T$ with $0 \leq \rho < 1$. \cite{hastie2019surprises} showed that 
\begin{equation}
	\trace[\bX^\T(\bX \bX^\T)^{-2}\bX\bSigma] \xrightarrow{\rm a.s.} \frac{1}{\gamma - 1}
\end{equation}
as $n$ and $p\to\infty$ and $\frac{p}{n} \to \gamma > 1$. Replacing $\gamma$ with $\frac{p}{n}$, we have
\begin{equation}\label{eq: df_over_approx_equicorr}
	\dfR \approx \frac{n}{2} + \frac{n}{2}\cdot\frac{n}{p-n} = \frac{n}{2}\left(1+\frac{n}{p-n}\right).
\end{equation}
Again, the limit doesn't depend on either $\rho$ or the variable orderings. In fact, under the equicorrelation assumption, (\ref{eq: df_under_approx}) and (\ref{eq: df_over_approx_equicorr}) can be integrated into a single expression
\begin{equation}\label{eq: df_equicor_approx}
	\dfR \approx \frac{\min(p,n)}{2} + \frac{n}{2}\cdot\frac{\min(p,n)}{\vert n-p\vert} = \left(\frac{1}{2}+\frac{n}{2\vert n-p\vert}\right)\dfF, \quad p \neq n.
\end{equation}
This suggests that $\dfR > \dfF$ for $p < 2n$ and $\dfR < \dfF$ for $p>2n$ asymptotically. Figure \ref{fig: df_limit} illustrates the approximation with $\bSigma = \frac{1}{2}\bI_d + \frac{1}{2}\mathbf{1} \mathbf{1}^\T$.
\begin{figure}[t!]
	\centering
	\includegraphics[scale = 0.55]{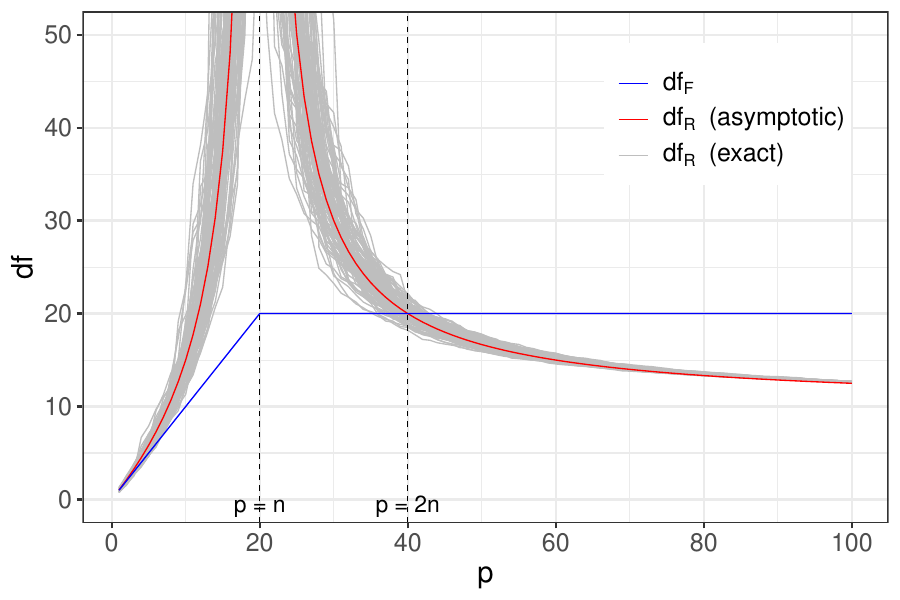}
	\caption{Predictive model degrees of freedom $\dfR$ versus the number of variables $p$ under the equal correlation setting. $\bx_1,\ldots,\bx_n$ are generated from $\mathcal{N}(0,\Sigma)$ with $n = 20$, $d = 100$, and $\bSigma = \frac{1}{2}\bI_d + \frac{1}{2}\mathbf{1} \mathbf{1}^\T$. The gray lines are $\dfR$ based on 100 randomly ordered variable sequences, whereas the red line is the approximate degrees of freedom in (\ref{eq: df_equicor_approx}) for equicorrelated features.}
	\label{fig: df_limit}
\end{figure}

\subsection{The Double Descent Phenomenon} \label{subsec: double_descent}
Recently, \cite{belkin2019two} demonstrated an interesting ``double descent'' phenomenon with least squares method that seems to defy the classical single U-shape risk curve. This has led to several follow-up works that present conditions under which overfitting can be benign or even (near) optimal. For example, \cite{bartlett2020benign} provided a finite sample characterization of overparameterized Gaussian linear models, and \cite{hastie2019surprises} analyzed the prediction risk of linear models asymptotically. In all of these works, the prediction risk is indexed by the total (or per observation) number of parameters as a proxy for model complexity. We argue that this proxy measure of model complexity requires adjustment beyond the interpolation threshold. In this subsection, we show that this fascinating phenomenon can be well reconciled with the classical risk theory if we parameterize the risk with the proposed predictive model degrees of freedom. 

We consider a similar setting as discussed in Figure 2 of \cite{belkin2019two}. Let $n=20$ and $d=100$. Assume that the true model is linear with coefficients $\beta_j$ satisfying $\beta_j \propto \frac{1}{j}$ and $\Vert \beta \Vert^2 = 10$. Also assume that  $\bx_1,\ldots,\bx_n \sim \mathcal{N}(\mathbf{0},\bI_d)$ and $\varepsilon_1,\ldots,\varepsilon_n \sim \mathcal{N}(0,1)$. Consider adding variables to the least squares model in the descending order of the coefficients. As shown in the left panel of Figure \ref{fig: reconcile_double_descent}, the ``double descent'' phenomenon can be observed when we plot the prediction error against the subset size $p$. If we instead align the prediction error against the predictive model degrees of freedom, the two side-by-side U-shape curves become folded into two similar U-shape curves defined over a comparable range of complexity.
\begin{figure}
	\centering
	\includegraphics[scale = 0.6]{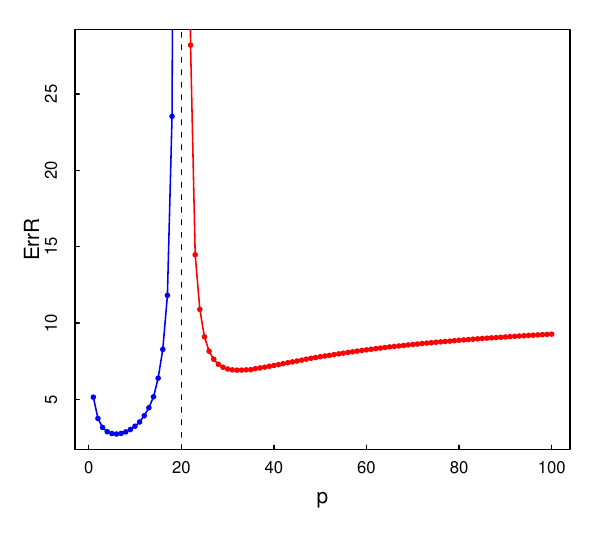}
	\includegraphics[scale = 0.6]{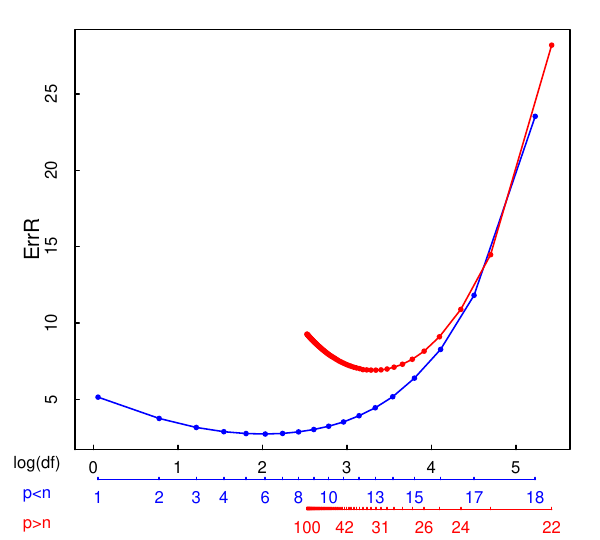}
	\caption{The out-of-sample prediction error as a function of $p$ (left) and $\log(\dfR)$ (right).}
	\label{fig: reconcile_double_descent}
\end{figure}

In terms of the structure of a model space, it is more appropriate to treat the ordinary least squares method and the minimum-norm least squares method as two different procedures, since the model space for the former is nested with an increasing dimension while that for the latter is not due to the implicit regularization. Thus, the ``double descent'' phenomenon can be well reconciled with the classical theory on bias-variance trade-off when we plot the prediction error against the predictive model degrees of freedom separately for the underparameterized and overparameterized regimes.

\section{Prediction Error Estimation For Least Squares Method} \label{sec: pred_err_est}
In this section, we aim to construct out-of-sample prediction error estimators for the least squares method using the predictive model degrees of freedom. Note that
\begin{equation}\label{eq: ErrR_est_decomp}
	\ErrR_\bX = \ErrT_\bX + \Delta B_\bX + \frac{2}{n}\sigma_\varepsilon^2 \dfR.
\end{equation}
Assume that $\sigma_\varepsilon^2$ and $\dfR$ are known. The expected training error $\ErrT_\bX$ can be conveniently estimated by its sample version $\ErrT_{\bX,\by}$. Thus, estimating the prediction error generally requires an estimator of the excess bias $\Delta B_\bX$.

We discuss two scenarios here for estimation of the excess bias. In Section \ref{subsec: risk_est_linear}, we first consider a special case where the true mean function is linear and covariates are multivariate normal. In Section \ref{subsec: risk_est_nonlinear}, we consider a general setting. It is important to note that, while the estimators we develop in this second scenario are specifically for the least squares method, the way we derive them also works for a general linear procedure. 

Throughout this section, we assume that $\cS$ is a subset of variables of size $p$.

\subsection{Linear Mean Function with Gaussian Covariates}\label{subsec: risk_est_linear}

In this subsection, we assume that $p < n$, $\mu(\bx_i;\bbeta)=\bx_i^\T\bbeta$ and $\bx_i \sim \mathcal{N}(\mathbf{0},\bSigma)$. Letting $\cS^\comp = \mathcal{D}\backslash \cS$, we define $\bSigma_\cS = \var(\bx_{i,\cS})$, $\bSigma_{\cS^\comp} = \var(\bx_{i,\cS^\comp})$, $\bSigma_{\cS,\cS^\comp}=\E(\bx_{i,\cS} \bx_{i,\cS^\comp}^\T)$, and $\bSigma_{\cS^\comp\vert \cS} = \var(\bx_{i,\cS^\comp}\vert \bx_{i,\cS}) = \bSigma_{\cS^\comp} - \bSigma_{\cS,\cS^\comp}^\T\bSigma_\cS^{-1}\bSigma_{S,\cS^\comp}$. Consider the subset regression model $\hat{\mu}_i = \bx_{i,\cS}^\T \hat{\bbeta}(\cS)$. The following proposition gives the conditional expectation of the excess bias and training error when $\bX_\cS$ is given. 
\begin{proposition}\label{prop: risk_est_linear_normal}
	Assume $\mu(\bx_i;\bbeta)=\bx_i^\T\bbeta$ and $\bx_i \sim \mathcal{N}(\mathbf{0},\bSigma)$. Then, for each $\cS \subseteq \mathcal{D}$ with $\vert \cS \vert = p < n$,
	$$
		\E(\Delta B_\bX \vert \bX_\cS) = \frac{2}{n} \sigma_\cS^2 \,\dfR(\cS) \quad \text{and} \quad \E(\ErrT_\bX \vert \bX_\cS) = \frac{n-p}{n} \sigma_{\varepsilon,\cS}^2,
	$$
	where $\sigma_{\cS}^2 = \bbeta_{\cS^\comp}^\T\bSigma_{\cS^\comp\vert \cS}\bbeta_{\cS^\comp}$ and $\sigma_{\varepsilon, \cS}^2 = \sigma_\varepsilon^2 + \sigma_\cS^2$.
\end{proposition}

Applying the proposition to \eqref{eq: ErrR_est_decomp}, we immediately have 
\begin{equation}\label{eq: ErrR_Xs}
	\E(\ErrR_\bX \vert \bX_\cS) = \frac{1}{n}\sigma_{\varepsilon,\cS}^2 [n - p + 2\,\dfR(\cS)].
\end{equation}
The proposition also implies that $\hat{\sigma}_{\varepsilon,\cS}^2 = \frac{n}{n-p}\ErrT_{\bX,\by}$ is an unbiased estimator of $\sigma_{\varepsilon,S}^2$ given $\bX_\cS$. Thus, an unbiased estimator of the conditional prediction error $\E(\ErrR_\bX \vert \bX_\cS)$ is given by
\begin{equation}\label{eq: ErrR_tilde}
	\widetilde{\ErrR} = \frac{1}{n}\hat{\sigma}_{\varepsilon,\cS}^2 [n - p + 2\,\dfR(\cS)] = \ErrT_{\bX, \by} + \frac{2}{n}\hat{\sigma}_{\varepsilon,\cS}^2\,\dfR(\cS).
\end{equation}
In particular, if $\cS = \mathcal{D}$, we have $\sigma_\cS^2 = 0$ and $\sigma_{\varepsilon, \cS}^2 = \sigma_\varepsilon^2$. In this case, $\widetilde{\ErrR}$ is an unbiased estimator of $\ErrR_\bX$. We call $\widetilde{\ErrR}$ a $C_p$-type estimator since it is of the same form as Mallows's $C_p$, which adjusts the training error by a model complexity measure. 

We can compare $C_p$ and $\widetilde{\ErrR}$ in terms of the optimal model size chosen by the criteria. Figure \ref{fig: mod_select_compr} illustrates the difference between the two criteria under some assumptions on $\sigma_\cS^2$. Details on the derivation of the optimal model sizes can be found in Appendix \ref{ex: mod_select_compr}. As shown in the figure, $\widetilde{\ErrR}$ generally favors more parsimonious models than $C_p$ does. The difference in the optimal model size is more substantial when the true model is sparse, i.e., only a few variables are significantly related to the response. We believe that this is quite reasonable as out-of-sample prediction generally involves more uncertainty than in-sample prediction. Such uncertainty is reflected in the prediction error estimator through the predictive model degrees of freedom, which results in the selection of a simpler model.
\begin{figure}[t]
	\centering
	\includegraphics[scale = 0.6]{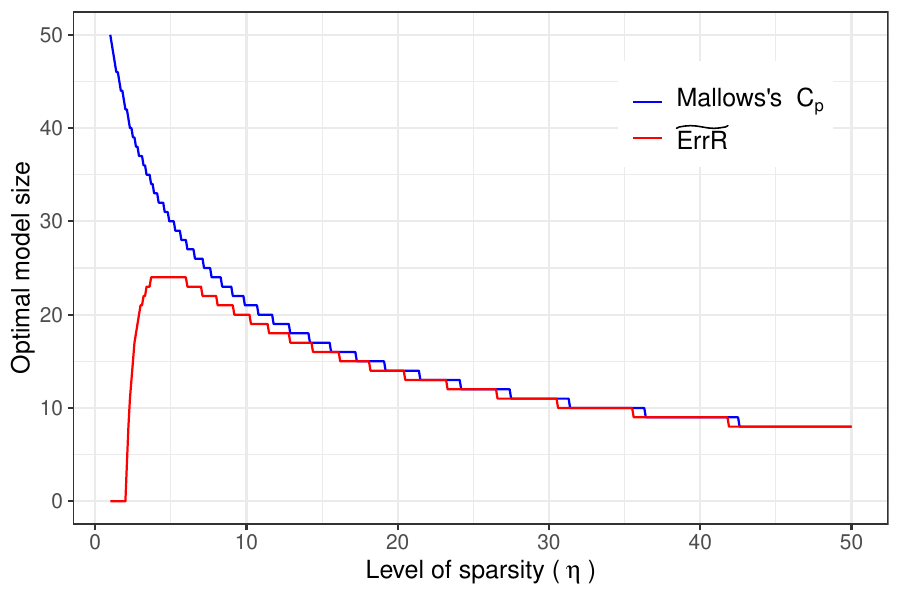}
	\caption{Comparison of the optimal model size identified via the expected Mallows's $C_p$ and $\widetilde{\ErrR}$ under the assumption that $\sigma_\cS^2 = \left(1-\frac{p}{d}\right)^\eta$ with $d=n=100$ and $\sigma_\varepsilon^2=1$ (SNR = 1). $\eta$ controls the sparsity of the model. The larger $\eta$ is, the more quickly $\sigma_\cS^2$ decays, and the more elements in $\bbeta$ are likely to be near 0.}
	\label{fig: mod_select_compr}
\end{figure}

There is also a close connection between $\widetilde{\ErrR}$ and the $S_p$ criterion \citep{tukey1967discussion, hocking1976biometrics, thompson1978selection}, the $\hat{U}_{np}$ statistic \citep{breiman1983many} and the generalized covariance penalty criterion $\widehat{\mathrm{RCp}}$ \citep{rosset2020fixed}. All these criteria estimate the unconditional prediction error $\E[(y_\ast - \hat{\mu}_\ast)^2]$ by
$$
    \hat{U}_{np} = \hat{\sigma}_{\varepsilon, \cS}^2 \left(1 + \frac{p}{n-p-1}\right) = \frac{n(n-1)}{(n-p)(n-p-1)} \ErrT_{\bX,\by}.
$$
In fact, $\widetilde{\ErrR}$ is a conditional version of $\hat{U}_{np}$ given the observed values of the set of variables in the current model. Under the normality assumption, if we replace $\dfR$ in \eqref{eq: ErrR_tilde} with its expectation \eqref{eq: dfR_normal_expectation}, we get exactly the $\hat{U}_{np}$ statistic.

\subsection{Nonlinear Mean Function} \label{subsec: risk_est_nonlinear}

In this subsection, we broaden the scope of our study on prediction error estimation by allowing a general form of the true mean function $\mu(\bx;\bbeta)$ and a general distribution of $\bx$. Our strategy is to estimate the excess bias $\Delta B_\bX$ using the leave-one-out cross validation (LOOCV) technique, which was also adopted by \cite{rosset2020fixed} in a similar context in the underparameterized regime. As a consequence, our risk estimators are closely related to the LOOCV error. In general, the estimators we develop in Section \ref{subsubsec: improved_estimators} exhibit much smaller variance than the LOOCV error. We will demonstrate this difference in Section \ref{sec: numerical_studies} through a series of numerical studies.

Consider a general linear procedure based on $(\bX_\cS, \by)$ with hat matrix $\bH$. Let $\bX_\cS^{-i}$, $\by^{-i}$ and $\bmu^{-i}$ be the corresponding terms with the $i$th record deleted. Let $\bh^{-i}_i$ denote the hat vector of $\bx_{i,\cS}$ based on $\bX_\cS^{-i}$. Conceptually, when $n$ is large, we have
\begin{equation} \label{eq: delta_approx_conceptual}
	\E[(\mu_\ast - \bh_\ast^\T\bmu)^2 \vert \bX] \approx \frac{1}{n}\sum_{i=1}^n(\mu_i - (\bh^{-i}_i)^\T \bmu^{-i})^2.
\end{equation}
Given a set of $\mu_i$'s, the right hand side of \eqref{eq: delta_approx_conceptual} can be evaluated based on the full data model alone for linear procedures. This is due to the well-known LOOCV identity for linear procedures which states that $y_i - (\bh_i^{-i})^\T \by^{-i} = (y_i - \bh_i^\T \by) / (1 - h_{ii})$ \citep{craven1978smoothing}.
Note that $\bh_i$ doesn't depend on either $\by$ nor $\bmu$. Thus, the theorem still holds if we replace $y_i$ and $\by$ with $\mu_i$ and $\bmu$, respectively. Then, we have
\begin{equation}\label{eq: delta_approx}
	\begin{aligned}
		\Delta B_\bX & = \E[(\mu_\ast - \bh_\ast^\T\bmu)^2 \vert \bX] - \frac{1}{n}\Vert \bmu - \bH\bmu\Vert^2\\
		& \approx \frac{1}{n}\sum_{i=1}^n [(\mu_i - (\bh^{-i}_i)^\T \bmu^{-i})^2 - (\mu_i - \bh_i^\T \bmu)^2]\\
		& = \frac{1}{n}\sum_{i=1}^n \left[\frac{(\mu_i - \bh_i^\T\bmu)^2}{(1-h_{ii})^2} - (\mu_i - \bh_i^\T\bmu)^2\right]\\
		& = \frac{1}{n}\bmu^\T \bA \bmu \\
		& = \frac{1}{n}\by^\T \bA \by - \frac{2}{n}\bmu^\T \bA \bm{\varepsilon} - \frac{1}{n}\bm{\varepsilon}^\T \bA \bm{\varepsilon},
	\end{aligned}
\end{equation}
where $\bA= (\bI_n - \bH)^\T \mathbf{D} (\bI_n - \bH)$ and $\mathbf{D} = \mathrm{diag}\left(\frac{1}{(1-h_{ii})^2} - 1\right)$. Since $\E(\bmu^\T \bA\bm{\varepsilon}\vert \bX) = 0$ and $\E(\bm{\varepsilon}^\T \bA \bm{\varepsilon}\vert \bX) = \sigma_\varepsilon^2 \trace(\bA)$, we can estimate $\Delta B_\bX$ by
\begin{equation} \label{eq: delta_hat}
	\hat{\delta} = \frac{1}{n}\by^\T \bA \by - \frac{1}{n}\sigma_\varepsilon^2 \trace(\bA) = \frac{1}{n}\sum_{i=1}^n \frac{h_{ii}(2-h_{ii})}{(1-h_{ii})^2} (y_i - \bh_i^\T \by)^2 - \frac{1}{n}\sigma_\varepsilon^2 \trace(\bA).
\end{equation}
Further, an estimator of $\ErrR_\bX$ is given by
\begin{equation} \label{eq: ErrR_hat}
	\widehat{\ErrR} = \mathrm{ErrT}_{\bX,\by} + \hat{\delta} + \frac{2}{n}\sigma_\varepsilon^2 \dfR =\widehat{\ErrR}_{\rm loocv} + \frac{1}{n} \sigma_\varepsilon^2 \xi_\bX,
\end{equation}
where $\widehat{\ErrR}_{\rm loocv} = \frac{1}{n} \sum_{i=1}^n \frac{(y_i - \bh_i^\T \by)^2}{(1 - h_{ii})^2}$ is the LOOCV error and $\xi_\bX = 2\,\dfR - \trace(\bA)$. Hence, $\widehat{\ErrR}$ is an adjusted version of the LOOCV error.

Since $\widehat{\ErrR}_{\rm loocv}$ is (almost) unbiased in estimating the true out-of-sample prediction error, it is then necessary to understand the meaning of the adjustment $\frac{1}{n} \sigma_\varepsilon^2 \xi_\bX$ and its effect relative to the LOOCV error. In Section \ref{subsubsec: adjustment_interpretation}, we first provide an interpretation of this adjustment in terms of the variance of prediction errors. In Sections \ref{subsubsec: risk_est_under} and \ref{subsubsec: risk_est_over}, we take the least squares method as an example and further examine the properties of the adjustment and $\widehat{\ErrR}$ in the underparameterized and overparameterized regimes separately. In summary, we find that the adjustment is generally negligible when $n$ is large, but could be extremely negative near the interpolation threshold. In Section \ref{subsubsec: improved_estimators}, we discuss possible ways to address this issue and improve our estimators.

\subsubsection{Interpretation of the Adjustment}\label{subsubsec: adjustment_interpretation}
Let $\hat{\bm{\varepsilon}} = (\hat{\varepsilon}_1, \ldots, \hat{\varepsilon}_n)^\T$ and $\hat{\bm{\varepsilon}}_{\rm loocv} = (\hat{\varepsilon}_1^{-1}, \ldots, \hat{\varepsilon}_n^{-n})^\T$ be the residual vectors of the full model and leave-one-out model respectively. Note that $\hat{\bm{\varepsilon}} = (\bI_n - \bH)\by$ and $\hat{\bm{\varepsilon}}_{\rm loocv} = \mathrm{diag}((1-h_{ii})^{-1}) (\bI_n - \bH)\by$. Then
$$
	\begin{aligned}
		\trace(\bA) 
		& = \trace[(\bI_n - \bH)^\T \mathrm{diag}((1-h_{ii})^{-2}) (\bI_n - \bH) ] - \trace[(\bI_n - \bH)^\T (\bI_n - \bH)]\\
		& = \trace[\mathrm{diag}((1-h_{ii})^{-1}) (\bI_n - \bH) (\bI_n - \bH)^\T \mathrm{diag}((1-h_{ii})^{-1})] - \trace[(\bI_n - \bH) (\bI_n - \bH)^\T]\\
		& = \frac{1}{\sigma_\varepsilon^2} (\trace[\var(\hat{\bm{\varepsilon}}_{\rm loocv})] - \trace[\var(\hat{\bm{\varepsilon}})]) \\
		& = \frac{1}{\sigma_\varepsilon^2} \sum_{i=1}^n [\var(\hat{\varepsilon}^{-i}_i \vert \bX) - \var(\hat{\varepsilon}_i \vert \bX)].
	\end{aligned}
$$

On the other hand, the predictive model degrees of freedom is defined through the excess variance of the predictions on the test data against those on the training data. Using the bias-variance decomposition of $\ErrR_\bX$ and $\ErrT_\bX$, it can be expressed as
$$
	\begin{aligned}
		\dfR 
		& = \frac{1}{2\sigma_\varepsilon^2} \left( n \E[\var(\hat{\mu}_\ast \vert \bx_\ast, \bX) \vert \bX] - \sum_{i=1}^n [\var(\hat{\mu}_i \vert \bX) - 2 \cov(y_i, \hat{\mu}_i \vert \bX)]\right)\\
		& = \frac{1}{2\sigma_\varepsilon^2} \left( n \E[\var(\hat{\varepsilon}_\ast \vert \bx_\ast, \bX) \vert \bX] -  \sum_{i=1}^n \var(\hat{\varepsilon}_i \vert \bX) \right).
	\end{aligned}
$$
Then, we have
\begin{equation}\label{eq: ErrR_cv_adjustment}
	\frac{1}{n}\sigma_\varepsilon^2 \xi_\bX = \E[\var(\hat{\varepsilon}_\ast \vert \bx_\ast, \bX) \vert \bX] - \frac{1}{n} \sum_{i=1}^n \var(\hat{\varepsilon}^{-i}_i \vert \bX). 
\end{equation}
While the adjustment $\frac{1}{n} \sigma_\varepsilon^2 \xi_\bX$ is a result of using the leave-one-out trick to estimate the excess bias $\Delta B_\bX$, it can also be interpreted as the difference between the full model and leave-one-out models in the variance of prediction errors. 

Qualitatively, the adjustment should be generally negligible on average when $n$ is large and the model is not too complex. But as model complexity increases, the variance of the prediction errors is expected to increase, and so is the adjustment. We will provide a more quantitative characterization of the expectation and variance of the adjustment in the following.

\subsubsection{Underparameterized Regime} \label{subsubsec: risk_est_under}

Assume $p < n$. For the ordinary least squares method, it is easy to show that $\trace(\bA) = \sum_{i=1}^n \frac{1}{1-h_{ii}} + p - n$ and
$$
	\xi_\bX = 2\,\dfR(\cS) + n - p - \sum_{i=1}^n \frac{1}{1-h_{ii}}.
$$
Plugging them into \eqref{eq: delta_hat} and \eqref{eq: ErrR_hat} respectively, we get
\begin{align*}
	\hat{\delta} & = \frac{1}{n} \sum_{i=1}^n \left[(y_i - \hat{\mu}_i)^2 - (1-h_{ii}) \sigma_\varepsilon^2\right]\left(\frac{1}{(1-h_{ii})^2} - 1\right),\\
	\widehat{\ErrR} & = \widehat{\ErrR}_{\rm loocv} - \frac{1}{n}\sigma_\varepsilon^2 \sum_{i=1}^n \frac{h_{ii}}{1-h_{ii}} + \frac{1}{n}\sigma_\varepsilon^2 (2\,\dfR - p).
\end{align*}
It turns out that our excess bias estimator $\hat{\delta}$ is exactly the same as $\widehat{B^+}$ defined in \citet{rosset2020fixed}. If we further replace $\dfR$ with its approximation \eqref{eq: df_under_approx}, our risk estimator $\widehat{\ErrR}$ is also asymptotically equivalent to the generalized covariance penalty criterion $\mathrm{RCP}^+$ in their work.

\begin{remark}
	Replacing $\dfR$ with its asymptotic approximation \eqref{eq: df_under_approx} in $\xi_\bX$ yields
	$$
		\xi_\bX \approx p \left(1+\frac{n}{n-p}\right) + n - p - \sum_{i=1}^n \frac{1}{1-h_{ii}}.
	$$
	Note that $\sum_{i=1}^n h_{ii} = p$. In particular, when $h_{ii}=\frac{p}{n}$ for all $i=1,\ldots,n$, we have $\xi_\bX \approx 0$. This suggests that $\widehat{\ErrR}$ and $\widehat{\ErrR}_{\rm loocv}$ are asymptotically equivalent when all the observations are equally influential.
\end{remark}

In the following, we study the expectation of $\xi_\bX$ when $\bx_i$'s are multivariate normal. We first state a lemma about the leverage $h_{ii}$ of the hat matrix $\bH$ in this setting.
\begin{lemma}\label{lemma: exp_inv_leverage}
	Assume $5 < p < n-2$. Let $\bx_1,\ldots,\bx_n \in \mathbb{R}^p$ be i.i.d. from a multivariate normal distribution. Let $\bH = (h_{ij})$ be the hat matrix of the ordinary least squares method with design matrix $\bX = (\bx_1,\ldots,\bx_n)^\T$. Then for every $i = 1,\ldots,n$,
	$$
		\E\left(\frac{1}{1-h_{ii}}\right) = \frac{n-1}{n} \frac{n-3}{n-p-2}.
	$$
	If further $p < n-4$, then
	$$
		\var\left(\frac{1}{1-h_{ii}}\right) = \left(\frac{n-1}{n}\right)^2 \frac{2(p-1)(n-3)}{(n-p-4)(n-p-2)^2}.
	$$
\end{lemma}

Using the lemma above, it is easy to obtain the following result about the expected value of $\xi_\bX$.
\begin{theorem} \label{thm: expected_adj_under}
	Assume $5 < p < n-2$. Let $\bx_1,\ldots,\bx_n \in \mathbb{R}^p$ be i.i.d. from a multivariate normal distribution. Then, for the ordinary least squares method with design matrix $\bX = (\bx_1,\ldots,\bx_n)^\T$,
	$$
		\E(\xi_\bX) \to \frac{2 - 3\gamma}{(1 - \gamma)^2},
	$$
	as $n$ and $p\to\infty$ and $\frac{p}{n}\to\gamma < 1$.
\end{theorem} 
\begin{proof}
	Lemma \ref{lemma: exp_inv_leverage} implies that
	$$
		\E[\trace(\bA)] = \frac{(n-1)(n-3)}{n-p-2} + p - n.
	$$
	By \eqref{eq: dfR_normal_expectation}, we have
	$$
		\E(\dfR) = \frac{np}{2(n-p-1)} + \frac{p}{2}.
	$$
	Thus, as $n$ and $p\to\infty$ and $\frac{p}{n}\to\gamma < 1$,
	$$
		\E(\xi_\bX) = 2\E(\dfR) - \E[\trace(\bA)] = \frac{2n^2-3np-5n+3p+3}{(n-p-2)(n-p-1)} \to \frac{2-3\gamma}{(1-\gamma)^2}.
	$$
\end{proof}

Theorem \ref{thm: expected_adj_under} suggests that the adjustment $\frac{1}{n} \sigma_\varepsilon^2 \xi_\bX$ is negligible on average as long as $n$ is large and $p$ is not too close to $n$. But since $\E(\xi_\bX) \to -\infty$ as $\gamma \to 1$, $\widehat{\ErrR}$ could be negative, especially when $p$ gets closer to $n$. 

In terms of the variability of the adjustment, Lemma \ref{lemma: exp_inv_leverage} implies that
$$
	n\var\left(\frac{1}{1-h_{ii}}\right) \to \frac{2 \gamma}{(1 - \gamma)^3}
$$
as $n$ and $p\to\infty$ and $\frac{p}{n}\to\gamma < 1$. As a consequence, the variance of $\trace(\bA)$ and $\xi_\bX$ can be very large near the interpolation threshold, which may further increase the chance of $\widehat{\ErrR}$ being negative. In Section \ref{subsubsec: improved_estimators}, we will further discuss this issue and develop corrections for $\hat{\delta}$ and $\widehat{\ErrR}$.

\subsubsection{Overparameterized Regime} \label{subsubsec: risk_est_over}

When $p > n$, we consider the minimum-norm least squares method. To get the explicit form of the matrix $\bA$, we study the ridge regression problem as a proxy, since the minimum-norm least squares estimator is the limit of the solution to the ridge regression problem \eqref{eq: ridge_regression} as $\lambda\to 0$ \citep{hastie2019surprises}. For design matrix $\bX_\cS$ and regularization parameter $\lambda$, the hat matrix $\bH_\lambda = (h_{\lambda, ij})$ of the ridge regression model is given by
$$
	\bH_\lambda = \bX_\cS \bX_\cS^\T (\bX_\cS \bX_\cS^\T + \lambda \bI_n)^{-1} = \bI_n - \lambda (\bX_\cS \bX_\cS^\T + \lambda \bI_n)^{-1}.
$$
Then, we have $h_{\lambda, ii} = 1 - \lambda [(\bX_\cS \bX_\cS^\T + \lambda \bI)^{-1}]_{ii}$ and
$$
	\begin{aligned}
		\bA 
		& = \lim_{\lambda \to 0} (\bI_n - \bH_\lambda^\T) \mathrm{diag}\left(\frac{h_{\lambda, ii}(2-h_{\lambda, ii})}{(1-h_{\lambda, ii})^2}\right) (\bI_n - \bH_\lambda) \\
		& = \lim_{\lambda \to 0} (\bX_\cS \bX_\cS^\T + \lambda \bI_n)^{-1} \mathrm{diag}\left(\frac{1 - \lambda^2 [(\bX_\cS \bX_\cS^\T + \lambda \bI_n)^{-1}]_{ii}^2}{[(\bX_\cS \bX_\cS^\T + \lambda \bI_n)^{-1}]_{ii}^2}\right) (\bX_\cS \bX_\cS^\T + \lambda \bI_n)^{-1}\\
		& = (\bX_\cS \bX_\cS^\T)^{-1} \mathrm{diag}\left(\frac{1}{[(\bX_\cS \bX_\cS^\T)^{-1}]_{ii}^2}\right)(\bX_\cS \bX_\cS^\T)^{-1}.
	\end{aligned}
$$
The following theorem provides some insight into the expectation of $\xi_\bX$ when $\bx_i$'s follow an isotropic multivariate normal distribution.
\begin{theorem} \label{thm: expected_adj_over}
	Assume $p > n+1$. Let $\bx_1,\ldots,\bx_n\in\mathbb{R}^p$ be i.i.d. from $\mathcal{N}(\mathbf{0},\bI_p)$, and assume that $\bX = (\bx_1,\ldots,\bx_n)^\T$ has full row rank. Define 
	$$
		\bA = (\bX \bX^\T)^{-1} \mathrm{diag}\left(\frac{1}{[(\bX \bX^\T)^{-1}]_{ii}^2}\right)(\bX \bX^\T)^{-1}.
	$$
	Then, for the minimum-norm least squares method with design matrix $\bX$,
	$$
		\E(\xi_\bX) \to \frac{1}{\gamma - 1}
	$$
	as $n$ and $p\to\infty$ and $\frac{p}{n}\to\gamma>1$.
\end{theorem}

According to the theorem above, the adjustment $\frac{1}{n}\sigma_\varepsilon^2 \xi_\bX$ is nontrivial when $p$ is close to $n$, which is similar to the underparameterized regime. Also, we recognize that the expectation of $\xi_\bX$ is positive as $\frac{p}{n}$ approaches $\gamma$. Note that, when $\bx_i \sim \mathcal{N}(\mathbf{0},\bI_p)$, $(\bX\bX^\T)^{-1}$ follows an inverse-Wishart distribution with degrees of freedom $p$ and scale matrix $\bI_n$. Using the results about the variance of an inverse-Wishart matrix presented in \cite{press2005applied}, we can show that
$$
	n \var([(\bX\bX^\T)^{-1}]_{ij}) = \frac{n^3}{(p-n)(p-n-1)(p-n-3)} \to \frac{1}{(\gamma-1)^3}
$$
as $n$ and $p\to\infty$ and $\frac{p}{n}\to\gamma>1$. As a result, $\trace(\bA)$, which depends heavily on $(\bX\bX^\T)^{-1}$, has a large variance near the interpolation threshold. Even though $\E(\xi_\bX) > 0$, it is still possible that a large negative $\xi_\bX$ makes $\hat{\delta}$ and $\widehat{\ErrR}$ negative as well. Therefore, $\hat{\delta}$ and $\widehat{\ErrR}$ need to be corrected for being negative in the overparameterized regime as well.

\subsubsection{Corrections for \texorpdfstring{$\hat{\delta}$}{\textdelta\textasciicircum} and \texorpdfstring{$\widehat{\ErrR}$}{ErrR\textasciicircum}} \label{subsubsec: improved_estimators}

As discussed in Sections \ref{subsubsec: risk_est_under} and \ref{subsubsec: risk_est_over}, $\widehat{\ErrR}$ may be negative around the interpolation threshold. This is due to the large variance of $\trace(\bA)$ that could cause $\hat{\delta} = \frac{1}{n}\by^\T \bA \by - \frac{1}{n} \sigma_\varepsilon^2 \trace(\bA)$ to be very negative. To address this issue, we consider the following family of estimators indexed by $a$ and $b$:
$$
	\hat{\delta}_{a,b} = \frac{a}{n} \by^\T \bA \by - \frac{b}{n} \sigma_\varepsilon^2 \trace(\bA), \quad a \in \mathbb{R}^+, b \in \mathbb{R}.
$$
Our goal is to estimate the excess bias $\Delta B_X$ with $\hat{\delta}_{a,b}$ for some $a$ and $b$ whenever $\hat{\delta} < 0$. Thus, we require
\begin{equation}\label{eq: risk_est_correction_req1}
	\frac{b}{a} \leq \hat{t} \coloneqq \frac{\by^\T \bA \by}{\sigma_\varepsilon^2\trace(\bA)}
\end{equation} 
so that $\hat{\delta}_{a,b} \geq 0$ for all $p$. 

As an alternative estimator of $\Delta B_\bX$, we also want the mean squared error of $\hat{\delta}_{a,b}$ to be as small as possible. Note that \eqref{eq: delta_approx} implies that
$$
	\E(\hat{\delta}) = \frac{1}{n} \bmu^\T \bA \bmu \approx \Delta B_\bX,
$$
Then, we can choose $(a,b)$ to minimize
$$
	\begin{aligned}
		R(a,b) 
		& \coloneqq \E[(n\hat{\delta}_{a,b} - \bmu^\T \bA \bmu)^2 \vert \bX] \\
		& = [(a-1) \bmu^\T \bA \bmu + (a - b) \sigma_\varepsilon^2 \trace(\bA)]^2 + a^2 \var(\by^\T \bA \by\vert X).
	\end{aligned}
$$
Solving $\frac{\partial R}{\partial b} = 0$ gives $a = \frac{b + t}{1 + t}$, where $t  = \frac{\bmu^\T \bA \bmu}{\sigma_\varepsilon^2 \trace(\bA)}$. Since $\bmu$ is unknown, we can use $\hat{t}$ in place of $t$. This leads to the identity
\begin{equation}\label{eq: risk_est_correction_req2}
	a = \frac{b + \hat{t}}{1 + \hat{t}}.
\end{equation} 
Combining the identity with \eqref{eq: risk_est_correction_req1}, we get a set of requirements for $(a,b)$:
$$
	b \leq \hat{t}^2 \; \text{and}\; a = \frac{b + \hat{t}}{1 + \hat{t}}.
$$
We consider two choices of $(a,b)$ below.
\begin{enumerate}
	\item[i.] If we take $b = \hat{t}^2$ and $a = \hat{t}$, then $\hat{\delta}_{a,b} = 0 = \hat{\delta}_{0,0}$. This leads to the estimator
	$$
		\hat{\delta}_+ = \max(\hat{\delta}, 0).
	$$
	
	\item[ii.] Note that the issue with $\hat{\delta}$ appears around the interpolation threshold, where $\trace(\bA)$ is large enough to make $t$ close to 0. Hence, another choice of $(a,b)$ is $b = 0$ and $a = \frac{\hat{t}}{1 + \hat{t}}$. The resulting estimator is
	$$
		\hat{\delta}_{++} = 
		\begin{dcases}
			\hat{\delta}, & \hat{\delta} \geq 0,\\
			\frac{(\by^\T \bA \by)^2}{n[\by^\T \bA \by + \sigma_\varepsilon^2 \trace(\bA)]}, & \hat{\delta} < 0.
		\end{dcases}
	$$
\end{enumerate}
With either correction above, the estimate of the excess bias $\Delta B_\bX$ is guaranteed to be nonnegative. Based on $\hat{\delta}_+$ and $\hat{\delta}_{++}$, we can define $\widehat{\ErrR}_+$ and $\widehat{\ErrR}_{++}$ correspondingly. We will compare these two estimators along with the original estimator $\widehat{\ErrR}$ and the LOOCV error $\widehat{\ErrR}_{\rm loocv}$ in Section \ref{sec: numerical_studies}.

\section{Numerical Studies on Subset Regression} \label{sec: numerical_studies}

We evaluate the performance of the estimators developed in Section \ref{sec: pred_err_est} through simulations and a real data analysis. 

\subsection{Simulations} \label{sec: simulations}
Throughout our experiments, we set $n = 50$ and $d = 120$, and assume $\bx_i \sim \mathcal{N}(\mathbf{0}, \bI_d)$ and $\varepsilon_i \sim \mathcal{N}(0, 1)$. We consider the following two mean functions:
\begin{itemize}
	\item Linear: $\mu(\bx; \bbeta) = \sum_{j=1}^d \beta_j x_j$
	\item Nonlinear: $\mu(\bx; \bbeta) = \sum_{j=1}^d \beta_j \left(e^{x_j/2} - e^{1/8}\right)$
\end{itemize}
Since we assume all features are i.i.d., the magnitude of the coefficients reflect the importance of the corresponding features to the response. We set $\beta_j = \alpha \left(1 - j/d\right)^\kappa$ ($\kappa \geq 1)$ and choose $\alpha$ such that $\Vert \bbeta \Vert^2 = 10$. To see the impact of model sparsity on the performance of the estimators, we examine two cases: $\kappa = 1$ (dense) and $\kappa = 5$ (sparse). In subset regression, we add variables presciently, i.e., from the most important to the least.

We first look at the estimators of the excess bias $\Delta B_\bX$. As shown in Figure \ref{fig: delta_est_compr_sparsity}, the estimator $\hat{\delta}$ (or $\widehat{B^+}$ in \cite{rosset2020fixed}) can indeed be negative when $p$ is close to $n$ regardless of the form and sparsity of the true model. This suggests the necessity of an appropriate correction for $\hat{\delta}$. On average, the two corrections proposed in Section \ref{subsubsec: improved_estimators} work equally well under the dense model but tend to overestimate $\Delta B_\bX$ for $p$ around $n$ when the true model is sparse. Such a difference is more prominent in the linear case. In general, $\hat{\delta}_+$ is less biased than $\hat{\delta}_{++}$, but the difference is fairly small. In the rest of the study, we will only consider $\hat{\delta}_+$ and $\widehat{\ErrR}_+$.
\begin{figure}[!b]
    \centering
	\includegraphics[scale = 0.6]{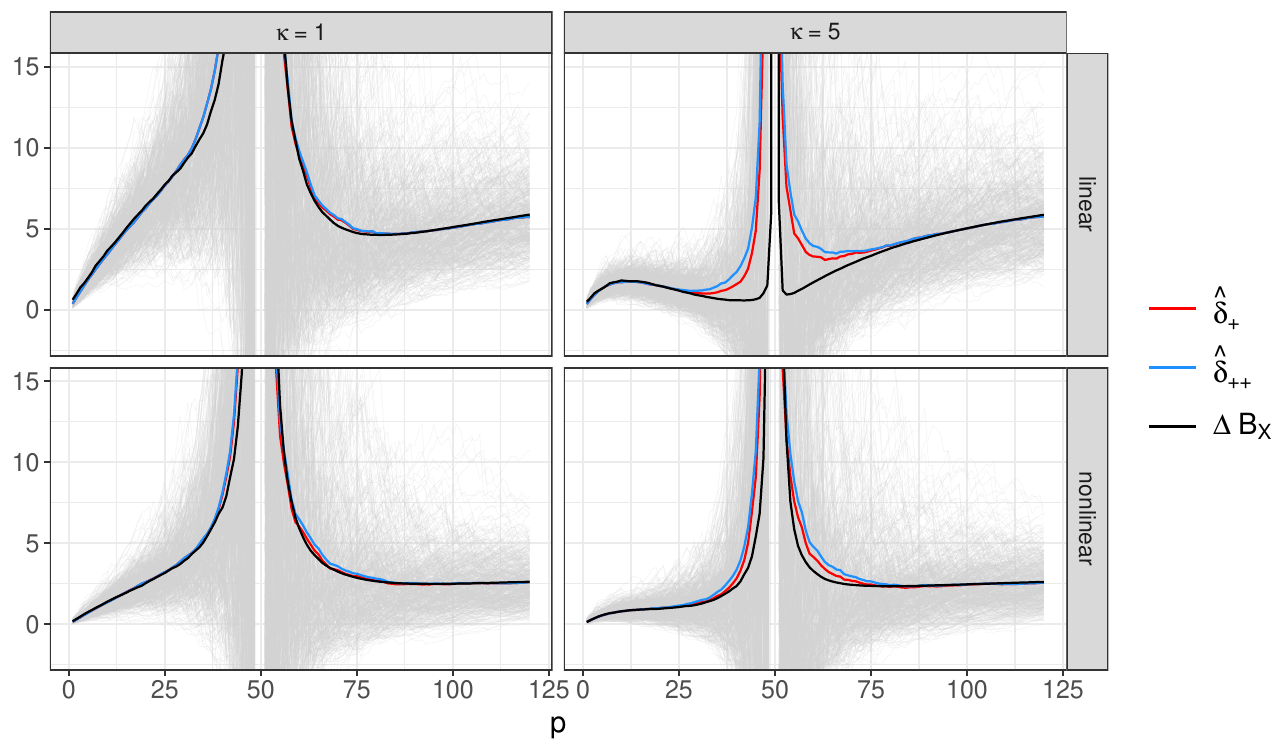}
	\caption{Comparison of $\hat{\delta}$, $\hat{\delta}_+$ and $\hat{\delta}_{++}$ as an estimator of the excess bias $\Delta B_\bX$. The gray lines show $\hat{\delta}$ for 500 random replicates of $(\bX, \by)$. $\hat{\delta}_+$, $\hat{\delta}_{++}$ and the true $\Delta B_\bX$ are averaged over these replicates.} \label{fig: delta_est_compr_sparsity}
\end{figure}

\begin{figure}[!t]
	\centering
	\includegraphics[scale = 0.62]{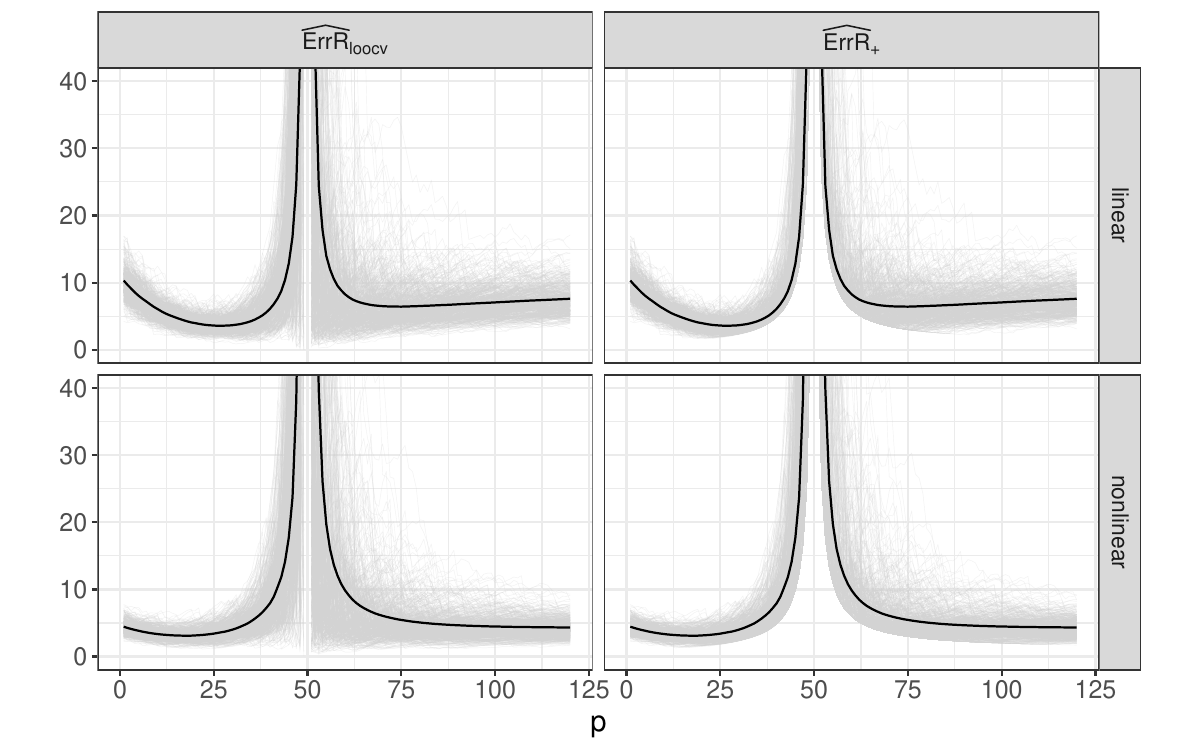}
	\caption{Comparison of $\widehat{\ErrR}_{\rm loocv}$ and $\widehat{\ErrR}_+$. The gray lines are the estimates for 500 random replicates of $(\bX,\by)$. The black line is the average true prediction error.}\label{fig: err_est_compr}
\end{figure}

For the prediction error estimators, we take $\kappa = 5$ as an example. Figure \ref{fig: err_est_compr} compares $\widehat{\ErrR}_{\rm loocv}$ and $\widehat{\ErrR}_+$ over 500 random realizations of $(\bX, \by)$. We see that the two estimators are very close when the model is away from the interpolation threshold. When $p$ is close to $n$, however, $\widehat{\ErrR}_+$ has much smaller variance than $\widehat{\ErrR}_{\rm loocv}$, despite having a slight upward bias that comes from $\hat{\delta}_+$.

To compare the performance of $\widehat{\ErrR}_{\rm loocv}$ and $\widehat{\ErrR}_+$ in estimating the risk quantitatively,  let $\ErrR^{(m)}$, $\widehat{\ErrR}_{\rm loocv}^{(m)}$ and $\widehat{\ErrR}_+^{(m)}$ be the true out-of-sample prediction error and the two estimates based on the $m$th replicate $(\bX^{(m)}, \by^{(m)})$. Define the relative mean squared error of $\widehat{\ErrR}_+$ to $\widehat{\ErrR}_{\rm loocv}$ as
$$
	\Pi(p) = \frac{\sum_{m} [\widehat{\ErrR}_+^{(m)}(p) - \ErrR^{(m)}(p)]^2}{\sum_{m} [\widehat{\ErrR}_{\rm loocv}^{(m)}(p) - \ErrR^{(m)}(p)]^2}, \; \text{for } p \neq n.
$$
As shown in Figure \ref{fig: err_est_compr_rel_mse}, $\widehat{\ErrR}_+$ can have much lower mean squared error than $\widehat{\ErrR}_{\rm loocv}$ around the interpolation threshold. 
\begin{figure}[!t]
	\centering
	\includegraphics[scale = 0.58]{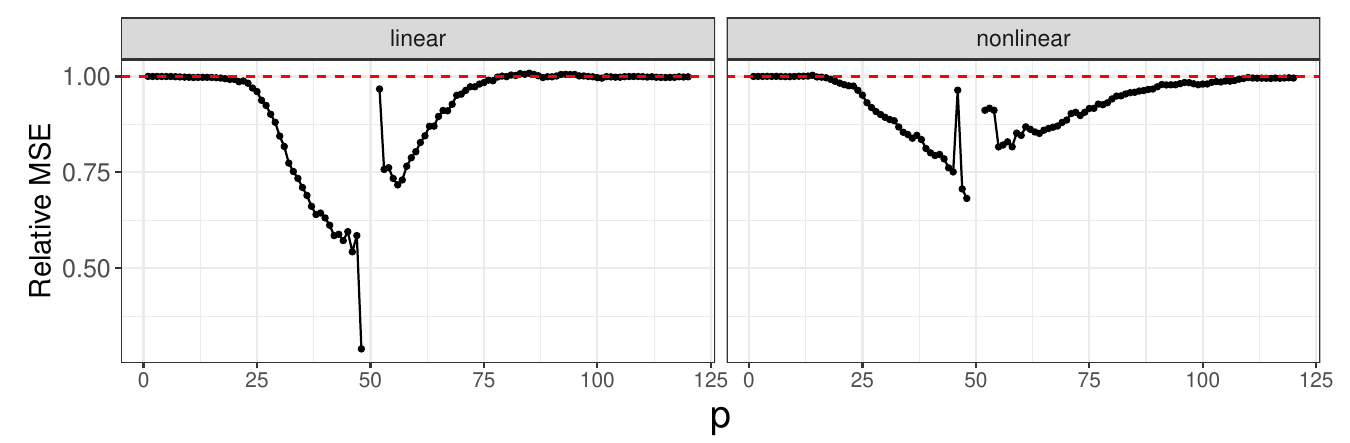}
	\caption{The relative mean squared error of $\widehat{\ErrR}_+$ to $\widehat{\ErrR}_{\rm loocv}$.} \label{fig: err_est_compr_rel_mse}
	\vspace{12pt}
	
	\includegraphics[scale = 0.52]{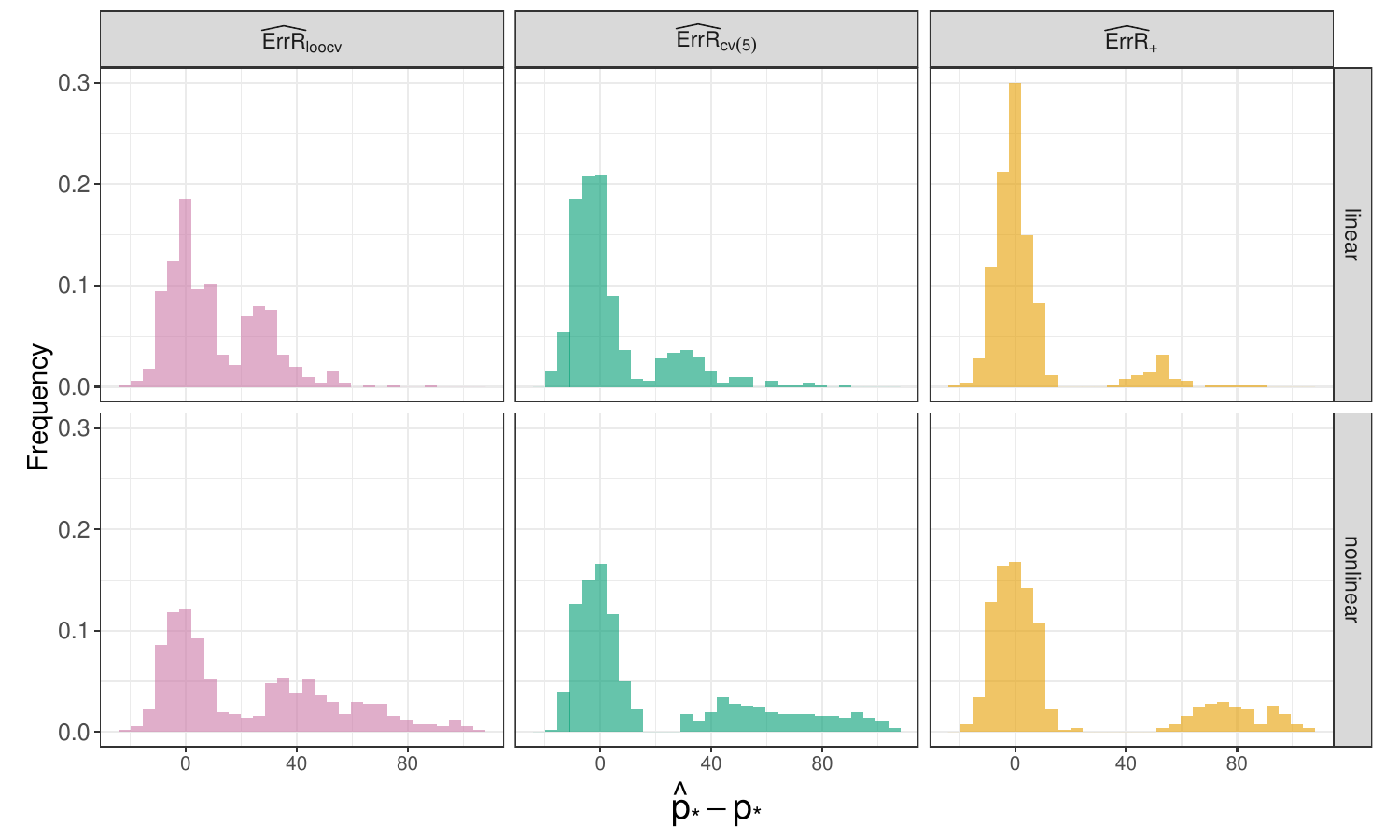}
	\caption{Comparison of $\widehat{\ErrR}_{\rm loocv}$, 5-fold cross validation and $\widehat{\ErrR}_+$ in model selection.}\label{fig: err_est_compr_model_selection}
\end{figure}

In practice, cross validation is often used to select models. Using the same simulated data, we compare the performance of $\widehat{\ErrR}_+$ with the leave-one-out and 5-fold cross validation in model selection. Let $p_\ast^{(m)}$ be the minimizer of the true prediction error $\ErrR^{(m)}$. Let $\hat{p}_+^{(m)}$, $\hat{p}_{\rm loocv}^{(m)}$ and $\hat{p}_{\rm cv(5)}^{(m)}$ be the optimal $p$ identified by $\widehat{\ErrR}_+$, LOOCV and 5-fold cross validation respectively. Figure \ref{fig: err_est_compr_model_selection} shows the histogram of $\hat{p}^{(m)}_{\ast} -  p_\ast^{(m)}$ for the three model selection methods. We can see that $\widehat{\ErrR}_+$ outperforms the other two with more concentration around 0 in the histogram for both the linear and nonlinear cases. We also notice that all histograms exhibit a bimodal pattern with a major mode around 0 followed by a minor one. This second mode is due to some competing interpolating models with comparable risk estimates as the optimal model in the underparameterized regime. In terms of the location of the second mode, $\widehat{\ErrR}_+$ has produced the largest model size. This implies that it tends to select simpler models than the other two methods in the overparameterized regime as illustrated in Figure \ref{fig: reconcile_double_descent}.

\subsection{Real Data Analysis}

We apply our prediction error estimators $\widehat{\ErrR}_+$ to the US cancer mortality data, which is publicly available at \url{https://data.world/nrippner/ols-regression-challenge}. The data were obtained at the county level with the response being the mean \textit{per capita} (100,000) cancer mortalities from 2010 to 2016. Covariates include demographics (population, median age, etc.), socioeconomic status (median income, unemployment rate, etc.), health conditions (coverage, cancer-related clinical trials, etc.) and education levels. Based on exploratory data analysis, 22 continuous independent variables are initially selected for modeling, with missing values imputed by the statewise median and appropriate transformations chosen. To reduce the impact of geographical variations, we focus on states around the Great Lakes area and their neighboring states as shown in Figure \ref{fig: rda_us_map}. 
\begin{figure}[t]
	\centering
	\includegraphics[scale = 0.6]{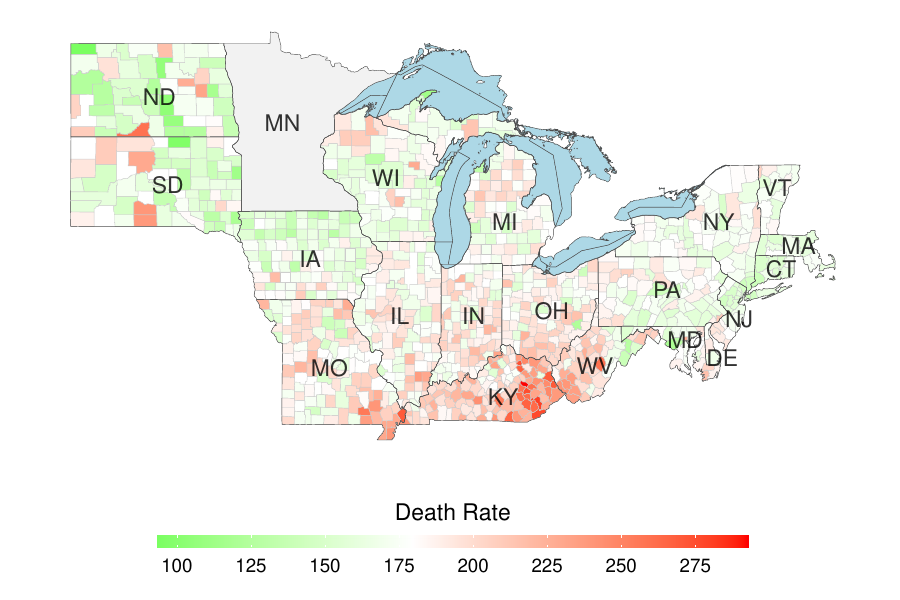}
	\caption{County-level \textit{per capita} (100,000) cancer mortalities of the states around the Great Lakes and their neighboring states.}
	\label{fig: rda_us_map}
\end{figure}
Minnesota is removed from the analysis due to data quality issues. Stratified sampling is used to create two datasets of size 40 and 510 from the total 1,148 counties for training and testing respectively. The remaining 598 counties are used to
\begin{itemize}
	\item Estimate the variable means and covariance matrix $\bSigma$, and center all observations in the training and test data;
	
	\item Estimate the error variance $\sigma_\varepsilon^2$ by fitting the full model with all predictors;
	
	\item Estimate $\dfR$ using $\hat{\bSigma}$ and $\hat{\sigma}_\varepsilon^2$;
	
	\item Starting from the null model, determine the variable sequence for subset regression by adding the variable that reduces the residual sum of squares most until all variables are included.
\end{itemize}

Figure \ref{fig: rda_subset_reg_single} demonstrates the performance of 6 different variable selection methods based on one random data partition. The black curve on the left panel shows the true prediction error over the test data, which is minimized at $p=6$ for this particular partition. We can see that $\widehat{\ErrR}_+$, $\widehat{\ErrR}_{\rm loocv}$, $\widehat{\ErrR}_{\rm cv(5)}$ and $\mathrm{BIC}$ all select the model with the first 7 variables, whereas $\widehat{\ErrF}$ and $\mathrm{AIC}$ both select the one with 19. In addition, $\widehat{\ErrR}_+$ and $\widehat{\ErrR}_{\rm loocv}$ are very close for all $p$ in this example. This is not surprising, since the sample size of the training data $n=150$ is much larger than the total number of predictors $d=22$, which makes the adjustment $\frac{1}{n}\sigma_\varepsilon^2 \xi_\bX$ almost negligible compared to $\widehat{\ErrR}_{\rm loocv}$ in \eqref{eq: ErrR_hat}. For more details about the model selected by $\widehat{\ErrR}_+$, see Appendix \ref{det: rda}.
\begin{figure}[!t]
	\centering
	\includegraphics[scale = 0.62]{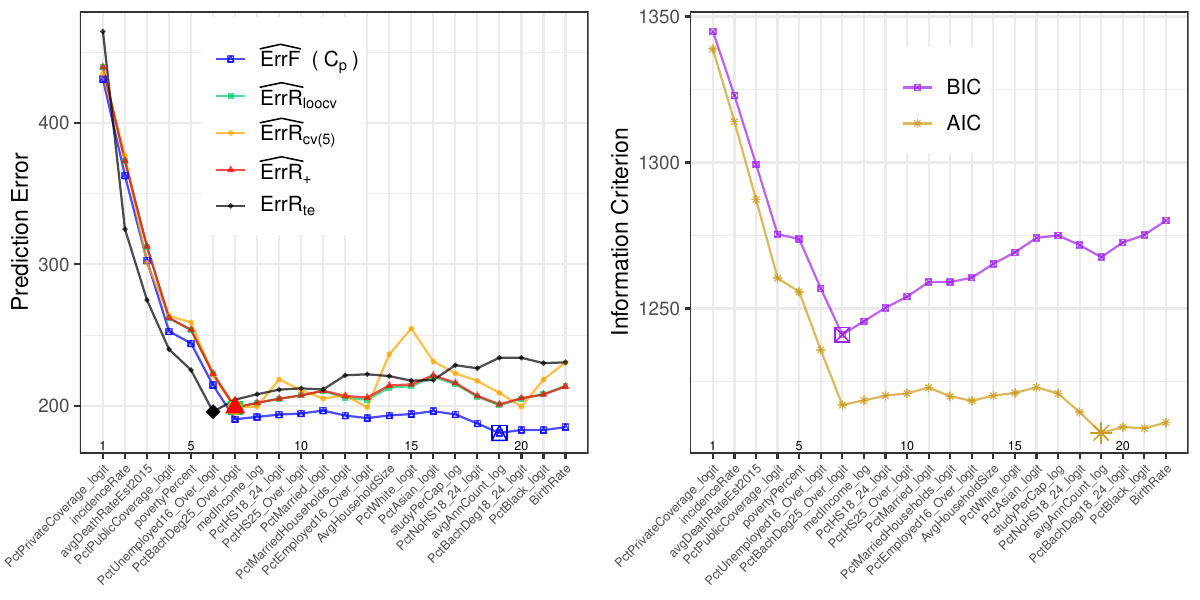}
	\caption{Comparison of various variable selection methods on the cancer mortality data. $\ErrR_{\rm te}$ is the mean prediction error on the test data. The highlighted points mark the minimum for each criterion.}
	\label{fig: rda_subset_reg_single}
\end{figure}
\begin{table}[t!]
	\centering
	\caption{Average performance of different variable selection methods based on 1,000 random data partitions. The sample size of the training data is taken to be 150 and 40 respectively in the two tables. $p_\ast$ and $\hat{p}_\ast$ are the optimal model size indicated by the test data and variable selection criterion respectively. Numbers in parentheses are the standard deviations.}
	\label{tab: rda_subset_reg_avg}
	\begin{tabular}{c|cccccc}
		\multicolumn{7}{c}{$n = 150$}\\
		\hline
		Method & $\widehat{\ErrR}_+$ & $C_p$ & LOOCV & 5-fold CV & AIC & BIC \\ 
		\hline
		$\hat{p}_\ast/p_\ast$ & \makecell{$1.109$\\\small($0.701$)}& \makecell{$1.609$\\\small($0.819$)} & \makecell{$1.136$\\\small($0.714$)} & \makecell{$1.211$\\\small($0.688$)} &  \makecell{$1.685$\\\small($0.837$)} & \makecell{$0.701$\\\small($0.329$)}  \\
		
		$\ErrR_{\rm te}(\hat{p}_\ast)/\ErrR_{\rm te}(p_\ast)$ & \makecell{$1.075$\\\small($0.077$)} & \makecell{$1.108$\\\small($0.110$)} & \makecell{$1.078$\\\small($0.086$)} & \makecell{$1.075$\\\small($0.092$)} & \makecell{$1.110$\\\small($0.111$)} & \makecell{$1.045$\\\small($0.036$)}  \\
		\hline
		\multicolumn{7}{c}{}\\
		\multicolumn{7}{c}{$n = 40$}\\
		\hline
		Method & $\widehat{\ErrR}_+$ & $C_p$ & LOOCV & 5-fold CV & AIC & BIC \\ 
		\hline
		$\hat{p}_\ast/p_\ast$ & \makecell{$0.910$\\\small($0.475$)}& \makecell{$1.303$\\\small($0.754$)} & \makecell{$1.144$\\\small($0.751$)} & \makecell{$1.140$\\\small($0.664$)} &  \makecell{$2.267$\\\small($1.064$)} & \makecell{$0.981$\\\small($0.675$)}  \\
		
		$\ErrR_{\rm te}(\hat{p}_\ast)/\ErrR_{\rm te}(p_\ast)$ & \makecell{$1.172$\\\small($0.206$)} & \makecell{$1.473$\\\small($1.079$)} & \makecell{$1.351$\\\small($0.704$)} & \makecell{$1.290$\\\small($0.651$)} & \makecell{$2.063$\\\small($1.309$)} & \makecell{$1.301$\\\small($0.682$)}  \\
		\hline
	\end{tabular}
\end{table}

With the dataset fixed for estimation of $\bSigma$, $\sigma_\varepsilon^2$, and variable sequence, the average performance of these methods is also studied based on 1,000 partitions of the data into training and test data. Here, in addition to the scenario of $n=150$, we also consider the case of $n=40$. Table \ref{tab: rda_subset_reg_avg} compares the methods in terms of the ratio of the chosen model size to the optimal size determined by test data and the corresponding ratio of the prediction error on the test data. When $n=150$, we see that $\widehat{\ErrR}_+$ and $\widehat{\ErrR}_{\rm loocv}$ have very similar performance, though the former selects slightly fewer variables on average and has marginally smaller variance. Among all 6 criteria, $\mathrm{BIC}$ performs the best in this scenario in terms of the parsimony, prediction error and stability of the selected model.

On the other hand, when $n=40$, $\widehat{\ErrR}_+$ is the most favorable, since the selected models include the least number of features on average while having the smallest prediction error over the test data, and it also has the smallest variance among all six variable selection methods.

\section{Linear Interpolating Models} \label{sec: grad_descient}

So far, our analysis of least squares method has provided us with guidance on how to choose a proper subset of variables for regression. In the overparameterized regime, however, there are infinitely many linear interpolants. While the minimum-norm least squares solution has the smallest norm, it may not necessarily be the best model. In this section, we focus on linear interpolating models and propose a procedure to choose a model among them. Throughout the section, we assume $\bX \in \R^{n \times p}$ and $p > n$.

\subsection{Gradient Descent}
For a given set of data, one may use the gradient descent algorithm to obtain a least squares solution. The gradient of $\frac{1}{2}\Vert \by - \bX\bbeta \Vert^2$ at $\bbeta = \mathbf{b}$ is $-\bX^\T (\by - \bX \mathbf{b})$. Given an initial value $\bbeta^{(0)}$ and a fixed step size $\alpha > 0$, the $k$th iterate of the gradient descent algorithm is given by
\begin{equation}\label{eq: gd_iterate}
	\bbeta^{(k)} = \bbeta^{(k-1)} + \alpha \bX^\T (\by - \bX \bbeta^{(k-1)}), \quad k = 1,2,\ldots.
\end{equation}
It is then easy to show that
\begin{equation} \label{eq: gradient_descent}
	\bbeta^{(k)} = \mathbf{E}^k\bbeta^{(0)} + (\bI_p - \mathbf{E}^{k}) \hat{\bbeta},
\end{equation}
where $\mathbf{E} = \bI_p - \alpha \bX^\T \bX$ and $\hat{\bbeta}$ is the minimum-norm least squares solution. The following lemma gives the condition under which the algorithm converges.
\begin{lemma} \label{lemma: gd_convergence}
	Assume $\by \neq \bX\bbeta^{(0)}$. For $p > n$, the gradient descent algorithm \eqref{eq: gradient_descent} converges if and only if 
	$$
		0 < \alpha < \frac{2}{\lambda_{\max}(\bX^\T \bX)},
	$$
	where $\lambda_{\max}(\bX^\T \bX)$ denotes the largest eigenvalue of $\bX^\T \bX$.
\end{lemma}

The proof of the lemma is given in \ref{pf: gd_convergence}. Assume $\bX$ has singular value decomposition
$$
	\mathbf{U} (\bm{\Psi}, \mathbf{O})
	\begin{pmatrix}
		\bV_1^\T \\
		\bV_2^\T
	\end{pmatrix} = \mathbf{U} \bm{\Psi} \bV_1^\T,\; \text{with } \mathbf{U},\bm{\Psi} \in \R^{n \times n}, \bV_1 \in \R^{p \times n}, \bV_2 \in \R^{p \times (p-n)},
$$
where $\mathbf{U}$ and $\bV = (\bV_1, \bV_2)$ are two orthogonal matrices and $\bm{\Psi} = \mathrm{diag}(\psi_1,\ldots,\psi_n)$ with $\psi_1 \geq \cdots \geq \psi_n > 0$. When the algorithm converges (i.e., $\mathbf{E}^{k} \to \bV_2 \bV_2^\T$), we can write the limit of the gradient descent iterates as
\begin{equation} \label{eq: gd_limit}
	\bbeta^{(\infty)} \coloneqq \lim_{k \to \infty} \bbeta^{(k)} = \hat{\bbeta} + \bV_2 \bV_2^\T \bbeta^{(0)}.
\end{equation}
Since $\bX \bbeta^{(\infty)} = \bX \hat{\bbeta} = \by$, $\bbeta^{(\infty)}$ defines an interpolating model. On the other hand, for any $\tilde{\bbeta}$ that interpolates the training data, setting $\bbeta^{(0)} = \tilde{\bbeta} + \bV_1 \bV_1^\T \mathbf{a}$ for any $\mathbf{a} \in \R^{p}$ yields
$$
	\tilde{\bbeta} = \hat{\bbeta} + \bV_2 \bV_2^\T \bbeta^{(0)}.
$$
Thus, we can obtain all interpolating models by varying $\bbeta^{(0)}$ in \eqref{eq: gd_limit}. Figure \ref{fig: ls_illustration} demonstrates the set of interpolating models when $p = 3$ and $n = 2$. In this case, the set forms a one-dimensional affine space that is parallel to the column space of $\bV_2$, and intersects the column space of $\bV_1$ (or equivalently, the row space of $\bX$) at the minimum-norm least squares solution $\hat{\bbeta}$.
\begin{figure}[t]
	\centering
	\includegraphics[scale = 0.5]{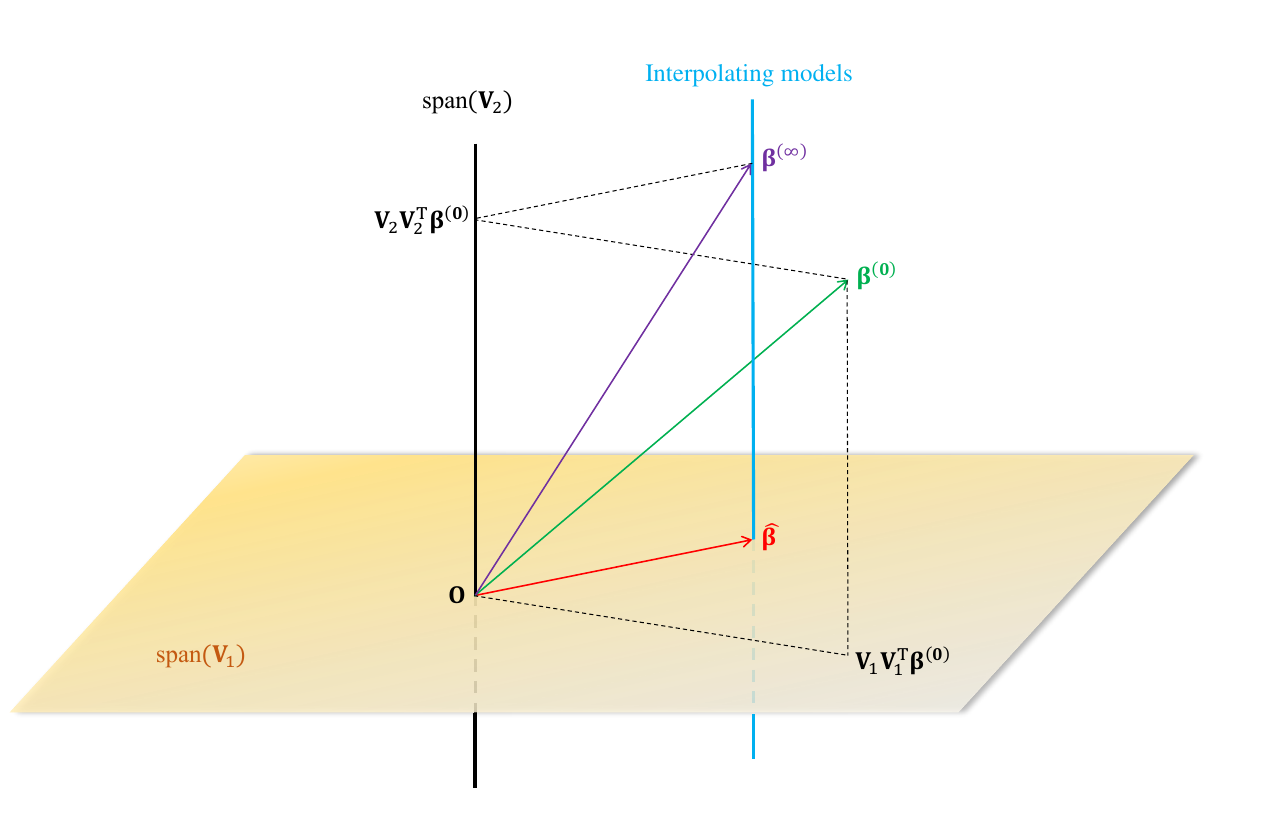}
	\caption{An illustration of the set of interpolating models for $p = 3$ and $n = 2$. The set forms a one-dimensional affine space that is parallel to the column space of $\bV_2$. Its intersection with the column space of $\bV_1$ corresponds to the minimum-norm least squares solution $\hat{\bbeta}$.}
	\label{fig: ls_illustration}
\end{figure}

\subsection{Initialization and Model Selection}
For the gradient descent algorithm \eqref{eq: gd_limit}, we need to specify an initialization scheme to explore different interpolating models. A straightforward way is to initialize $\bbeta^{(0)}$ randomly. While this approach is easy to implement and works well in many applications, it may not be efficient enough to produce a good interpolating model. In this subsection, we consider a data-dependent scheme instead. In particular, we restrict $\bbeta^{(0)}$ to be of the form:
\begin{equation} \label{eq: gd_init}
	\bbeta^{(0)} = \bF \by,
\end{equation}
where $\bF \in \R^{p \times n}$ depends only on $\bX$. One major advantage of this scheme is that the resulting interpolating model $\bbeta^{(\infty)}$ is a linear procedure with hat vector at $\bx_\ast$:
\begin{equation} \label{eq: gd_interpolant_hat_vec}
	\bh_\ast^{(\infty)} = \bh_\ast + \bF^\T \bV_2 \bV_2^\T \bx_\ast,
\end{equation}
where $\bh_\ast = (\bX \bX^\T)^{-1} \bX \bx_\ast$ is the corresponding hat vector of $\hat{\bbeta}$. As a consequence, all results in the previous sections apply to $\bbeta^{(\infty)}$. The following theorem gives the predictive model degrees of freedom of $\bbeta^{(\infty)}$ as well as its relationship with that of $\hat{\bbeta}$.
\begin{theorem} \label{thm: gd_df}
	Let $\bF \in \R^{n \times p}$ be a coefficient matrix that doesn't depend on $\by$, and define $\bbeta^{(0)} = \bF \by$. Let $\bbeta^{(\infty)}$ be the the limit of the gradient descent iterates based on the initial value $\bbeta^{(0)}$. Then, 
	$$
		\dfR(\bbeta^{(\infty)}) =  \frac{n}{2} + \frac{n}{2}\trace\left[(\bF^\T \bV_2 \bV_2^\T + (\bX \bX^\T)^{-1} \bX)^\T (\bF^\T \bV_2 \bV_2^\T + (\bX \bX^\T)^{-1} \bX) \bSigma \right].
	$$
	Further, if $\bSigma = \bI_p$,
	$$
		\dfR(\bbeta^{(\infty)}) \geq \dfR(\hat{\bbeta}),
	$$
	and the equality holds if and only if $\bbeta^{(\infty)} = \hat{\bbeta}$.
\end{theorem}

The proof can be found in \ref{pf: gd_df}. Below we propose an initialization procedure that constructs $\mathbf{F}$ using the simple linear regression coefficients of $\by$ on each column of $\bX$.
\begin{table}[H]
	\centering
	\begin{tabular}{l}
		\Xhline{2.5\arrayrulewidth}
		A data-dependent initialization scheme\\
		\hline
		Step 1: Randomly select $q$ ($q \leq n$) variables. Denote the variable set by $\cS$.\\
		Step 2: Set\\
		\multicolumn{1}{c}{
			$\displaystyle
			\beta_j^{(0)} = \theta_j \frac{\bx_{(j)}^\T \by}{\bx_{(j)}^\T \bx_{(j)}}\mathbf{1}_{\lbrace j \in \cS\rbrace}, \quad j = 1,\ldots,p.
			$}\\
		\hline
	\end{tabular}
\end{table}

We can interpret the above scheme as follows. When $q = 0$ or $\theta_j = 0$ for all $j\in\cS$, $\bbeta^{(0)} = \mathbf{0}$ and the gradient descent ends up with the minimum-norm least squares solution. When $\theta_j = 1$, $\beta_j^{(0)}$ is the simple linear regression coefficient of $\by$ on $\bx_{(j)}$ (without intercept). Thus, choosing a $\theta_j \in (0, 1)$ shrinks that coefficient toward 0. 

We consider a numerical experiment with $n = 20$ and $p = 60$ to evaluate the interpolating models obtained from this initialization scheme. We generate $\bx_i \sim \mathcal{N}(\mathbf{0},\bI_p)$ and $\varepsilon_i \sim \mathcal{N}(0,1)$ for 500 replicates. We assume $\mu(\bx;\bbeta) = \bx^\T \bbeta$ with $\beta_j \propto (1-j/p)^5$ and $\Vert \bbeta \Vert^2 = 10$. Figure \ref{fig: gd_df_single} demonstrates the average predictive model degrees of freedom, excess bias and prediction error of $\bbeta^{(\infty)}$ as a function of $\theta$ when a single variable is selected for initialization ($q = 1$). We see that the importance of the initial variable has little impact on $\dfR$, but can make a big difference in the excess bias and prediction error in two different ways. First, with the minimum-norm least squares model as the reference, using an important variable in the initial value helps reduce the excess bias and risk of the resulting interpolating model, while choosing an unimportant one degrades the performance. Second, the optimal shrinkage parameter $\theta$ generally decreases with the variable importance. So using larger values for important features and smaller values for less relevant ones is recommended.
\begin{figure}[!t]
	\centering
	\includegraphics[scale = 0.55]{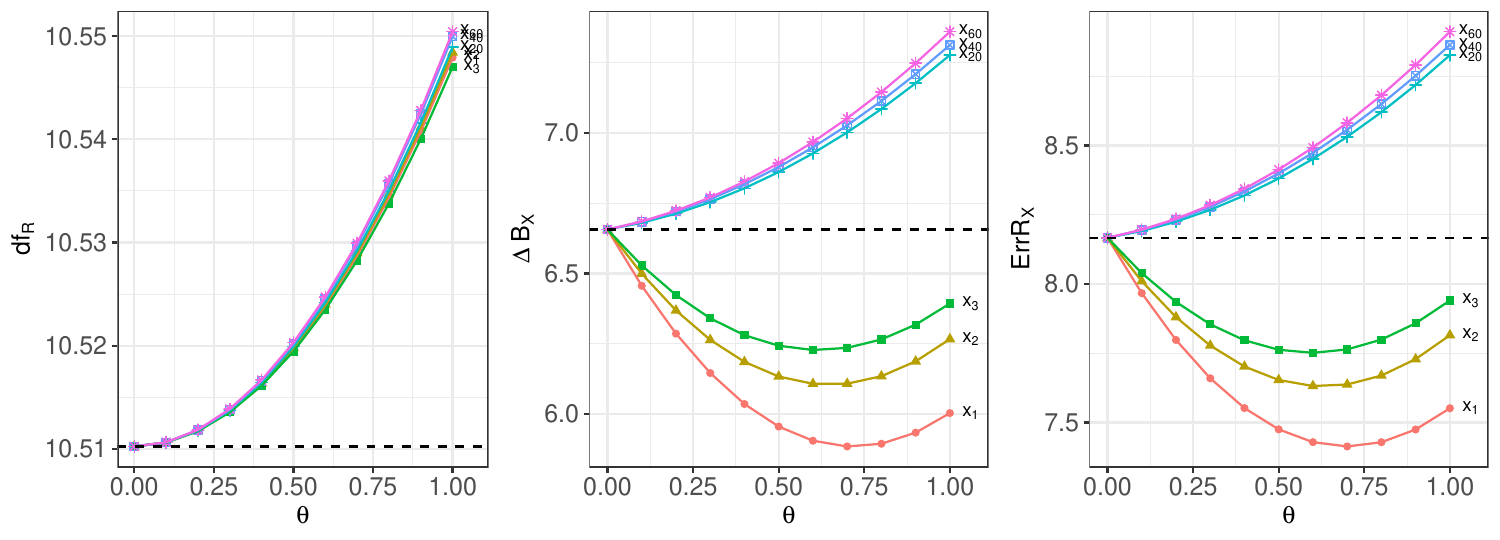}
	\caption{Predictive model degrees of freedom $\dfR$, excess bias $\Delta B_\bX$ and prediction error $\ErrR_\bX$ of $\bbeta^{(\infty)}$ as a function of the shrinkage parameter $\theta$ when a single variable is selected for initialization ($q = 1$). All results are averaged over 500 replicates. The horizontal dashed lines represent the corresponding values of the minimum-norm least squares model.}
	\label{fig: gd_df_single}
	\vspace{10pt}
	\includegraphics[scale = 0.55]{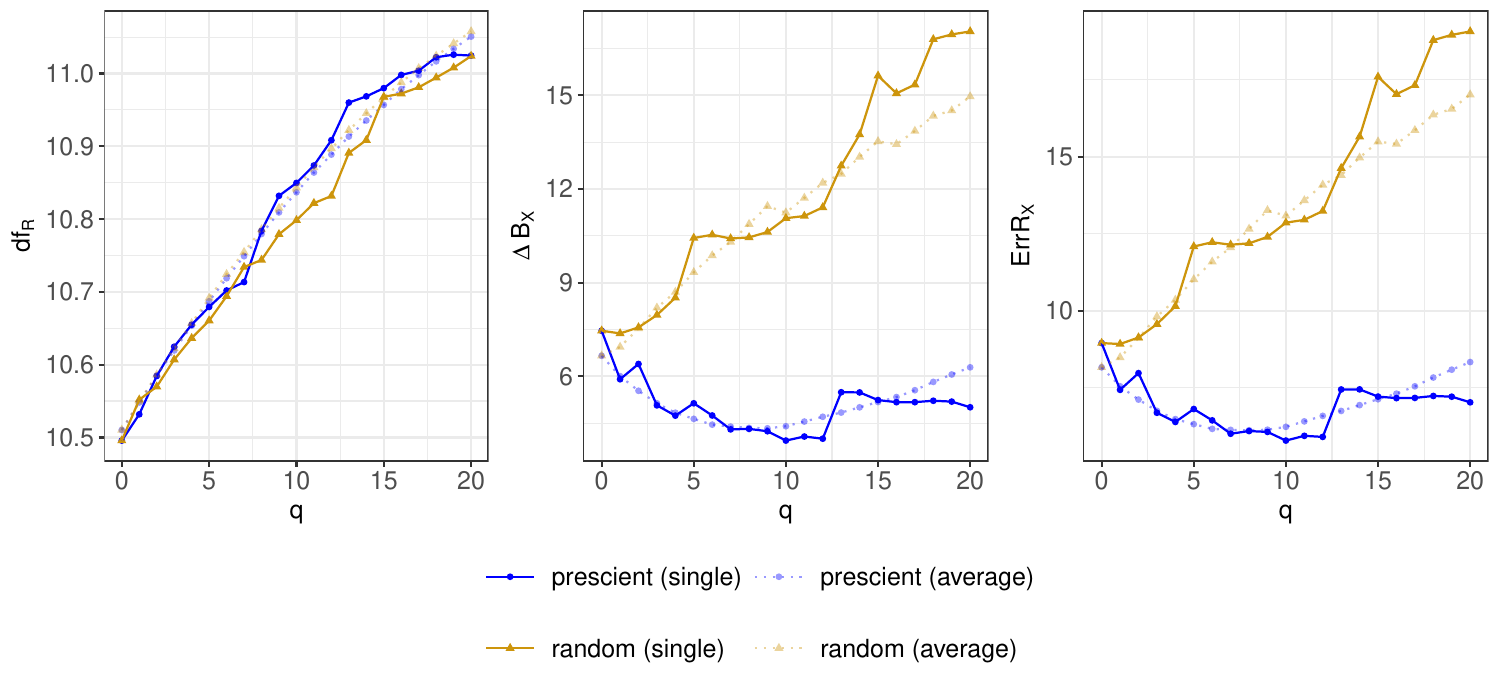}
	\caption{Predictive model degrees of freedom $\dfR$, excess bias $\Delta B_\bX$ and prediction error $\ErrR_\bX$ as a function of the subset size $q$ when the subset expands along a presciently ordered variable sequence and a prespecified randomly ordered one. The solid lines are based on one typical replicate of data whereas the dotted lines are for the average of 500 replicates. }
	\label{fig: gd_df_multiple}
\end{figure}

To examine the effect of subset size $q$, we fix $\theta_j = 1$ for all $j \in \cS$ and increase $q$ by expanding $\cS$ along a presciently ordered variable sequence and a prespecified randomly ordered one for comparison. Figure \ref{fig: gd_df_multiple} shows the results of both a single replicate and the average of 500 replicates. We find that the predictive model degrees of freedom increases as more variables are used in the initial value, and the pattern is consistent across different variable sequences. By contrast, the excess bias and prediction error behave in a completely different manner with the two sequences. When variables are selected presciently in the initial value, the corresponding interpolating models are almost uniformly better than the minimum-norm least squares solution corresponding to $q = 0$, with the optimal choice occurring when the first few most important variables are used. If we select variables at random, however, the resulting models generally get worse as $q$ increases. This once again suggests the influence of variable importance on the performance of the limiting interpolating models through the initial value.

In the following, we give a theoretical explanation for these findings on variable importance. We focus on the excess bias $\Delta B_\bX$ as it is the dominating term in the prediction error for our example. For convenience, we assume $\bx_i \sim \mathcal{N}(\mathbf{0},\bI_p)$ and keep $\theta_j = 1$ for all $j \in \cS$. Let $\mathbf{e}_k \in \R^p$ be the $k$th standard basis vector and $\mathbf{P} = \sum_{k \in \cS} \Vert \bx_{(k)} \Vert^{-2} \mathbf{e}_k \mathbf{e}_k^\T$. Then, we can show that $\bF = \mathbf{P} \bX^\T$
and
\begin{equation} \label{eq: gd_interpolant_excess_bias}
	\Delta B_\bX (\bbeta^{(\infty)}) = \Delta B_\bX (\hat{\bbeta}) - \Vert \bV_2 \bV_2^\T \bbeta \Vert^2 + \Vert \bV_2 \bV_2^\T \mathbf{z} \Vert^2,
\end{equation}
where $\mathbf{z} = \bbeta - \mathbf{P} \bX^\T \bX \bbeta $. For the derivation of \eqref{eq: gd_interpolant_excess_bias}, see Appendix \ref{pf: gd_interpolant_excess_bias}. Note that the third term is the only term that depends on the initial value. Our goal is to make $\Vert \bV_2 \bV_2^\T \mathbf{z} \Vert^2$ as small as possible by choosing $\cS$.

We first consider the case when $q = 1$. Let $\cS = \lbrace j \rbrace$ and define $c_{jk} = \bx_{(j)}^\T \bx_{(k)}/\Vert \bx_{(j)} \Vert^2$ for $k=1,\ldots,p$. When $n$ is large, $c_{jk} \approx 0$ for $j \neq k$ since $\cov(x_{ij}, x_{ik}) = 0$. So we have
$$
	\mathbf{z} = \bbeta - \sum_{k = 1}^p c_{jk} \beta_k \mathbf{e}_j \approx \bbeta - \beta_j \mathbf{e}_j \coloneqq \tilde{\mathbf{z}}.
$$ 
We can look at $\Vert \bV_2 \bV_2^\T \tilde{\mathbf{z}} \Vert^2$ as a proxy for $\Vert \bV_2 \bV_2^\T \mathbf{z} \Vert^2$. Since $\bV_2 \bV_2^\T = \bI_p - \bV_1 \bV_1^\T = \bI_p - \bX^\T (\bX \bX^\T)^{-1} \bX$, and $\bX^\T (\bX \bX^\T)^{-1} \bX$ is the orthogonal projection matrix onto the row space of $\bX$, it follows that
$$
	\Vert \bV_2 \bV_2^\T \tilde{\mathbf{z}} \Vert^2 = \Vert \tilde{\mathbf{z}} - \bX^\T (\bX \bX^\T)^{-1} \bX \tilde{\mathbf{z}} \Vert^2 = \Vert \tilde{\mathbf{z}} \Vert^2 - \Vert \bX^\T (\bX \bX^\T)^{-1} \bX \tilde{\mathbf{z}} \Vert^2.
$$
Also, by the rotational symmetry of the isotropic multivariate normal distribution, we have
$$
	\E(\bX^\T (\bX \bX^\T)^{-1} \bX) = \frac{n}{p} \bI_p.
$$
Therefore,
$$
	\begin{aligned}
		\E(\Vert \bV_2 \bV_2^\T \tilde{\mathbf{z}} \Vert^2) = \left(1 - \frac{n}{p}\right) \Vert \tilde{\mathbf{z}} \Vert^2 = \left(1 - \frac{n}{p}\right) (\Vert \bbeta \Vert^2 - \beta_j^2).
	\end{aligned}
$$
This suggests us to select the most important variable for initialization.

In general, let $\tilde{\bbeta} = \mathbf{P} \bX^\T \bX \bbeta$ so that $\mathbf{z} = \bbeta - \tilde{\bbeta}$. To reduce the excess bias, it suffices to make $\tilde{\bbeta}$ close to $\bbeta$. For simplicity, assume $\Vert \bx_{(j)}\Vert = \sqrt{n}$ for all $j \in \cS$ and denote the corresponding $\tilde{\bbeta}$ by $\tilde{\bbeta}^\ast$ to avoid confusion with the original $\tilde{\bbeta}$. Note that $\tilde{\beta}^\ast_j = (\beta_j + \frac{1}{n}\sum_{k \neq j} \beta_k \bx_{(j)}^\T \bx_{(k)})\mathbf{1}_{\lbrace j \in \cS\rbrace}$. Then we have $\E(\tilde{\beta}^\ast_j) = \beta_j \mathbf{1}_{\lbrace j \in \cS \rbrace}$, $\var(\tilde{\beta}^\ast_j) = \frac{1}{n} (\Vert \bbeta \Vert^2 - \beta_j^2) \mathbf{1}_{\lbrace j \in \cS \rbrace}$ and
\begin{align}
	& b(\tilde{\bbeta}^\ast) \coloneqq \Vert \E(\tilde{\bbeta}^{\ast}) - \bbeta \Vert^2 = \sum_{j = 1}^p (\E(\tilde{\beta}^\ast_j) - \beta_j)^2 = \Vert \bbeta \Vert^2 - \sum_{j \in \cS} \beta_j^2,\\
	& v(\tilde{\bbeta}^\ast) \coloneqq \E \Vert \tilde{\bbeta}^\ast - \E(\tilde{\bbeta}^\ast) \Vert^2 = \sum_{j = 1}^p \var(\tilde{\beta}_j^\ast) = \frac{q}{n} \Vert \bbeta \Vert^2 - \frac{1}{n} \sum_{j \in \cS} \beta_j^2,\\
	& \E\Vert \tilde{\bbeta}^\ast - \bbeta \Vert^2 = b(\tilde{\bbeta}^\ast) + v(\tilde{\bbeta}^\ast) = \frac{n+1}{n} \Vert \bbeta \Vert^2 \left(\frac{n+q}{n+1} - \frac{\sum_{j \in \cS} \beta_j^2}{\Vert \bbeta \Vert^2}\right) \label{eq: exp_norm_beta_betatilde}.
\end{align}
Figure \ref{fig: sq_bias_by_q} illustrates that $\E\Vert \tilde{\bbeta}^\ast - \bbeta \Vert^2$, $\Vert \bbeta-\tilde{\bbeta} \Vert^2$ and $\Vert \bV_2 \bV_2^\T (\bbeta-\tilde{\bbeta}) \Vert^2$ behave similarly along the prescient variable sequence. This justifies our proposal of using $\E\Vert \tilde{\bbeta}^\ast - \bbeta \Vert^2$ as an approximation of $\Vert \bbeta-\tilde{\bbeta} \Vert^2$ to study the behavior of the excess bias. \eqref{eq: exp_norm_beta_betatilde} indicates that, for a fixed $q$, selecting the $q$ most important  variables minimizes the expected distance between $\tilde{\bbeta}^\ast$ and $\bbeta$. As we vary $q$, the size of $\E\Vert \tilde{\bbeta}^\ast - \bbeta \Vert^2$ is governed by the trade-off between $b(\tilde{\bbeta}^\ast)$ and $v(\tilde{\bbeta}^\ast)$. Thus, an ideal set of initial variables should include the most important features while excluding the least important ones. 
\begin{figure}[!t]
	\centering
	\includegraphics[scale = 0.62]{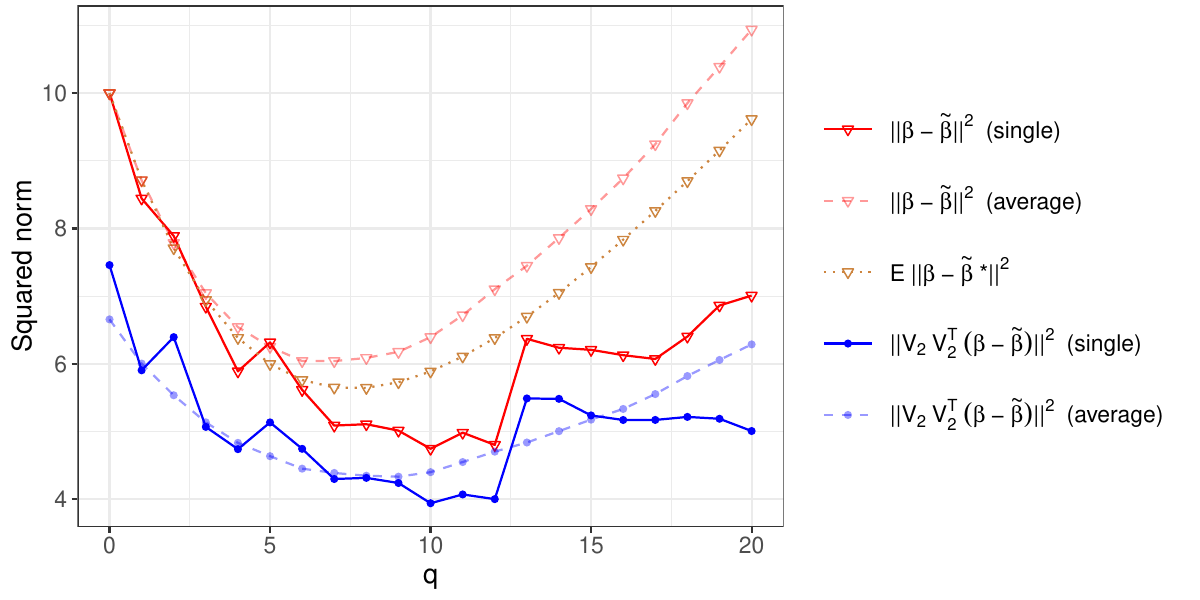}
	\caption{Squared norm of $\bbeta - \tilde{\bbeta}$ and $\bV_2 \bV_2^\T (\bbeta - \tilde{\bbeta})$ as a function of the subset size $q$ in the initial value. The averages are taken over 500 replicates of $\bX$ and the expected value is based on \eqref{eq: exp_norm_beta_betatilde}.}
	\label{fig: sq_bias_by_q}
\end{figure}

\section{Conclusion and Discussion} \label{sec: discussion}
In this work, we have proposed the concept of predictive model degrees of freedom for linear procedures in the standard regression setting. The proposed measure of model complexity targets estimation of out-of-sample prediction error and can differentiate interpolating models. This does not only provide insights into the ``double descent'' phenomenon, but it also allows us to consider a potentially broader choice of models in practice with proper adjustment in the model complexity. Our numerical results on subset regression illustrate benefits of the variable selection criterion based on the predictive model degrees of freedom over other classical methods such as Mallows's $C_p$, $\mathrm{AIC}$, $\mathrm{BIC}$ and cross validation. By accounting for additional uncertainty in out-of-sample prediction, our risk estimator tends to favor more parsimonious models than other classical criteria with reduced variance. The analysis of gradient descent algorithm on least squares problems in the overparameterized regime also reveals some interesting properties of linear interpolating models and sheds light on the effect of initial values on the risk.

There are several extensions worth considering based on the current work. Within the
scope of linear regression procedures that we have focused on, we find it important to develop
efficient methods to estimate the predictive model degrees of freedom using the training
data. Given that heteroskedasticity in data is quite common in practice, it will be also useful to generalize the current framework to incorporate data weights into modeling procedures so as to handle unequal error variances.

Beyond the scope of linear modeling, it will be interesting to extend the framework to a generalized linear model setting with a general loss function. We believe that the squared
covariance adjustment to the classical model degrees of freedom presented in Proposition \ref{prop: dfR_cov_representation} will provide a promising
way to extend the current result to exponential family data and thus deserves further
investigation.

\subsubsection*{Acknowledgments}
We thank Misha Belkin for a lively discussion on the phenomenon of double descent and interpolation in overparametrized regimes, which inspired this work. This research was supported in part by the National Science Foundation Grants DMS-15-13566, DMS-17-12580, DMS-17-21445 and DMS-20-15490.

\bibliographystyle{apalike}
\bibliography{bibliography}

\newpage
\appendix

\section{Proofs}

\subsection{Proof of Propositions \ref{prop: dfR_cov_representation} and \ref{prop: dfR_gdf_representation}}

\textit{Proof.} Since $\hat{\mu}_\ast = \bh_\ast^\T \by$ and $\hat{\mu}_j = \bh_j^\T \by$, we have
$$
	\begin{aligned}
		& \frac{\cov(y_i, \hat{\mu}_\ast \vert \bx_\ast, \bX)}{\sigma_\varepsilon^2} = \frac{\partial \E(\hat{\mu}_\ast\vert \bx_\ast,\bX)}{\partial \mu_i} = h_{\ast,i},\\
		& \frac{\cov(y_i, \hat{\mu}_j \vert \bx_\ast, \bX)}{\sigma_\varepsilon^2} = \frac{\partial \E(\hat{\mu}_j\vert \bX)}{\partial \mu_i} = h_{ji}.
	\end{aligned}
$$
As a result,
$$
	\sum_{i=1}^n \E\left(\frac{\cov^2(y_i, \hat{\mu}_\ast \vert \bx_\ast, \bX)}{(\sigma_\varepsilon^2)^2} \Bigg\vert \bX\right) = \sum_{i=1}^n \E\left(\left(\frac{\partial \E(\hat{\mu}_\ast \vert \bx_\ast, \bX) }{\partial \mu_i}\right)^2 \Bigg\vert \bX\right) = \E(\Vert \bh_\ast \Vert^2 \vert \bX),
$$
and
$$
	\sum_{i=1}^n \sum_{j=1}^n \frac{\cov^2(y_i, \hat{\mu}_j \vert \bX)}{(\sigma_\varepsilon^2)^2} = \sum_{i=1}^n \sum_{j=1}^n \left(\frac{\partial \E(\hat{\mu}_j \vert \bX) }{\partial \mu_i}\right)^2 = \trace(\bH^\T \bH).
$$
The propositions then follow. \hfill\qedsymbol

\subsection{Proof of Lemma \ref{lemma: trace_monotonicity_under}} \label{pf: useful_lemma}

\noindent\textit{Proof.} By the inversion formula for a 2 by 2 partitioned matrix,
$$
	(\tilde{\bX}^\T\tilde{\bX})^{-1} = 
	\begin{pmatrix}
		\mathbf{G} & -g(\bX^\T \bX)^{-1}\bX^\T \mathbf{w}\\
		-g \mathbf{w}^\T \bX(\bX^\T \bX)^{-1} & g
	\end{pmatrix},
$$
where
$$
	\begin{aligned}
		& \mathbf{G} = (\bX^\T \bX)^{-1} + g(\bX^\T \bX)^{-1}\bX^\T \mathbf{w} \mathbf{w}^\T \bX(\bX^\T \bX)^{-1},\\
		& g = \frac{1}{\mathbf{w}^\T (\bI-\bX(\bX^\T \bX)^{-1} \bX^\T) \mathbf{w}}.
	\end{aligned}
$$
$g$ is guaranteed to be positive by the assumption that $\mathrm{rank}(\tilde{\bX})=p+1$. Since $\tilde{\bB}$ is positive semi-definite, there must exist a multivariate normal random vector $\mathbf{v} \in \R^{p+1}$ such that $\var(\mathbf{v}) = \tilde{\bB}$. This implies that $b^2 - \ba^\T \bB^{-1} \ba = \var(v_{p+1}\vert v_1,\ldots,v_p) \geq 0$, with equality if and only if $v_{p+1}$ is a linear combination of $v_1,\ldots,v_p$. Let $\mathbf{u} = \bX\bB^{-1}\ba$, $\mathbf{C} = \bX(\bX^\T \bX)^{-1}\bB(\bX^\T \bX)^{-1}\bX^\T$ and
$$
	\mathbf{K} =
	\begin{pmatrix}
		\bB &  \ba\\
		\ba^\T &  \ba^\T\bB^{-1}\ba 
	\end{pmatrix}.
$$
Then
\begin{equation}\label{eq: pf_trace_inequality}
	\begin{aligned}
		\trace[(\tilde{\bX}^\T\tilde{\bX})^{-1}\tilde{\bB}]
		& = \trace[(\tilde{\bX}^\T\tilde{\bX})^{-1} \mathbf{K}] + g(b^2 - \ba^\T \bB^{-1}\ba)\\
		& \geq \trace[(\tilde{\bX}^\T\tilde{\bX})^{-1} \mathbf{K}]\\
		& = \trace[(\bX^\T \bX)^{-1}\bB] + g \mathbf{w}^\T \bX(\bX^\T \bX)^{-1} \bB (\bX^\T \bX)^{-1}\bX^\T \mathbf{w}\\
		& \qquad\qquad\qquad\quad\;\;\, - 2 g \mathbf{w}^\T \bX(\bX^\T \bX)^{-1} \ba + g \ba^\T\bB^{-1}\ba \\
		& = \trace[(\bX^\T \bX)^{-1}\bB] + g(\mathbf{w}-\mathbf{u})^\T \mathbf{C} (\mathbf{w}-\mathbf{u})\\
		& \geq \trace[(\bX^\T \bX)^{-1}\bB].
	\end{aligned}
\end{equation}
In particular, if $\tilde{\bB}$ is positive definite, $b^2 - \ba^\T \bB \ba >0$, so the strict inequality holds in \eqref{eq: df_ls_inequality}. \hfill\qedsymbol

%

\subsection{Proof of Theorem \ref{thm: dfR_increment}}\label{pf: dfR_increment}
\textit{Proof.} 
Let $\mathbf{u} = \bX_{\cS_1}  \bSigma_{\cS_1}^{-1} \bSigma_{\cS_1, j}$. Following \eqref{eq: pf_trace_inequality} in the proof of Lemma \ref{lemma: trace_monotonicity_under}, we have
$$
	\begin{aligned}
		\trace[(\bX_{\cS_2}^\T \bX_{\cS_2})^{-1} \bSigma_{\cS_2}] & - \trace[(\bX_{\cS_1}^\T
		\bX_{\cS_1})^{-1} \bSigma_{\cS_1}]\\
		& = \frac{(\bx_{(j)} - \mathbf{u})^\T \mathbf{C} (\bx_{(j)} -\mathbf{u}) + \sigma_j^2 - \bSigma_{\cS_1, j}^\T  \bSigma_{\cS_1}^{-1} \bSigma_{\cS_1, j}}{\bx_{(j)}^\T (\bI_n - \bH) \bx_{(j)}} = \frac{\bm{\zeta}^\T \mathbf{C} \bm{\zeta} + 1}{\bm{\zeta}^\T (\bI_n - \bH) \bm{\zeta}}.
	\end{aligned}
$$
The theorem then follows immediately. \hfill\qedsymbol

\subsection{Proof of Theorem \ref{thm: df_ls_monotonicity_general}} \label{pf: df_ls_monotonicity_general}

\textit{Proof.} Since $\mathbf{U}$ has full column rank, we can find $\bV \in \R^{p \times (p-s)}$ such that $\mathbf{Q} = (\mathbf{U}, \bV)\in\R^{p\times p}$ is nonsingular. Define $\tilde{\mathbf{Z}} = \bX \mathbf{Q} = (\mathbf{Z}, \bX \bV)$. Then
$$
	\trace[(\bX^\T \bX)^{-1} \bSigma] = \trace[(\tilde{\mathbf{Z}}^\T \tilde{\mathbf{Z}})^{-1}\mathbf{Q}^\T \bSigma \mathbf{Q}] \geq \trace[(\mathbf{Z}^\T \mathbf{Z})^{-1}\mathbf{U}^\T \bSigma \mathbf{U}],
$$
where we apply Lemma \ref{lemma: trace_monotonicity_under} to the inequality above by noting that $\mathbf{Z}$ is the first $s$ columns of $\tilde{\mathbf{Z}}$ and $\mathbf{U}^\T \tilde{\bSigma}\mathbf{U} = (\mathbf{Q}^\T \bSigma \mathbf{Q})_{1:s,1:s}$. In particular, the equality holds if and only if $s = p$. Also, notice that $\var(\mathbf{z}_i) = \var(\mathbf{U}^\T \bx_i) = \mathbf{U}^\T \bSigma\mathbf{U}$. Thus,
$$
	\dfR(\mathbf{Z}) = \frac{s}{2} + \frac{n}{2}\trace[(\mathbf{Z}^\T \mathbf{Z})^{-1} \mathbf{U}^\T \bSigma \mathbf{U}] \leq \frac{p}{2} + \frac{n}{2}\trace[(\bX^\T \bX)^{-1} \bSigma] = \dfR(\mathbf{X}).
$$
\hfill\qedsymbol

\subsection{Proof of Theorem \ref{thm: df_ls_monotonicity_over}} \label{pf: df_ls_monotonicity_over}
\textit{Proof.} When $\var(\bx_\ast) = \sigma_x^2 \bI_d$, \eqref{eq: dfR_ls} is simplified to
$$
	\dfR = \frac{n}{2} + \frac{n}{2} \sigma_x^2 \trace[(\bX_\cS \bX_\cS^\T)^{-1}], \; \text{for } p \geq n
$$
Similar to the proof of Theorem \ref{thm: df_ls_monotonicity}, it is sufficient to show \eqref{eq: df_ls_monotonicity_over} for $\cS_2 = \cS_1 \cup \left\lbrace j \right\rbrace$, where $j \in \mathcal{D} \backslash \cS_1$. By Sherman-Morrison formula,
$$
	(\bX_{\cS_2} \bX_{\cS_2}^\T)^{-1} = (\bX_{\cS_1} \bX_{\cS_1}^\T)^{-1} - t(\bX_{\cS_1} \bX_{\cS_1}^\T)^{-1}\bx_{(j)} \bx_{(j)}^\T (\bX_{\cS_1} \bX_{\cS_1}^\T)^{-1},
$$
where $t = (1 + \bx_{(j)}^\T (\bX_{\cS_1} \bX_{\cS_1}^\T)^{-1} \bx_{(j)})^{-1}>0$. Taking trace yields
$$
	\trace[(\bX_{\cS_2} \bX_{\cS_2}^\T)^{-1}] = \trace[(\bX_{\cS_1} \bX_{\cS_1}^\T)^{-1}] - t\Vert(\bX_{\cS_1} \bX_{\cS_1}^\T)^{-1} \bx_{(j)} \Vert^2 \leq \trace[(\bX_{\cS_1} \bX_{\cS_1}^\T)^{-1}],
$$
with equality if and only if $\bx_{(j)}$ is in the null space of $(\bX_{\cS_1} \bX_{\cS_1}^\T)^{-1}$. The theorem then follows. \hfill\qedsymbol

\subsection{Proof of Proposition \ref{prop: risk_est_linear_normal}} \label{pf: risk_est_linear_normal}
\textit{Proof.} Let $\bbeta_\cS = (\beta_j)_{j\in \cS}$ and $\bbeta_{\cS^\comp} = (\beta_j)_{j\in \cS^\comp}$. Using the property of the multivariate normal distribution, we have
$$
	\E(\bX_{\cS^\comp}\bbeta_{\cS^\comp}\vert \bX_\cS) = \bX_\cS \bSigma_\cS^{-1} \bSigma_{\cS,\cS^\comp} \bbeta_{\cS^\comp} \quad \text{and} \quad \var(\bX_{\cS^\comp}\bbeta_{\cS^\comp}\vert \bX_\cS) = \sigma_\cS^2 \bI_n.
$$
Note that $\bH = \bX_\cS (\bX_\cS^\T \bX_\cS)^{-1} \bX_\cS^\T$, $\bmu = \bX_\cS \bbeta_\cS + \bX_{\cS^\comp}\bbeta_{\cS^\comp}$ and $\mu_\ast = \bx_{\ast,\cS}^\T\bbeta_\cS + \bx_{\ast,\cS^\comp}^\T \bbeta_{\cS^\comp}$. Then, it is easy to show that
$$
	\begin{aligned}
		\E(\Vert \bmu - \bH\bmu \Vert^2 \vert \bX_\cS) 
		& = \E[\bbeta_{\cS^\comp}^\T \bX_{\cS^\comp}^\T (\bI_n - \bH) \bX_{\cS^\comp}
		\bbeta_{\cS^\comp}\vert \bX_\cS] = \sigma_\cS^2 (n-p),\\
		\E[(\mu_\ast - \bh_\ast^\T \bmu)^2 \vert \bX_\cS] & = \E[(\bx_{\ast,\cS^\comp}^\T \bbeta_{\cS^\comp} - \bx_{\ast,\cS}^\T (\bX_\cS^\T \bX_\cS)^{-1} \bX_\cS^\T \bX_{\cS^\comp} \bbeta_{\cS^\comp})^2 \vert \bX_\cS] \\
		& = \sigma_\cS^2 + \sigma_\cS^2 \trace[(\bX_\cS^\T \bX_\cS)^{-1} \bSigma_\cS].
	\end{aligned}
$$
Therefore,
$$
	\begin{aligned}
		\E(\Delta B_\bX \vert \bX_\cS) & = \E[(\mu_\ast - \bh_\ast^\T \bmu)^2 \vert \bX_\cS] - \frac{1}{n}\E(\Vert \bmu - \bH\bmu \Vert^2 \vert \bX_\cS)\\
		& = \frac{2}{n}\sigma_\cS^2 \left(\frac{p}{2} + \frac{n}{2}\trace[(\bX_\cS^\T \bX_\cS)^{-1} \bSigma_\cS]\right)\\
		& = \frac{2}{n}\sigma_\cS^2 \dfR(\cS),
	\end{aligned}
$$
and
$$
	\begin{aligned}
		\E(\ErrT_\bX \vert \bX_\cS)
		& = \frac{1}{n}\E(\Vert \by - \bH \by \Vert^2 \vert \bX_\cS) \\
		& = \frac{1}{n}\E(\Vert \bmu - \bH\bmu \Vert^2 \vert \bX_\cS) + \frac{1}{n} \E[\mathbf{\varepsilon}^\T (\bI_n -\bH) \mathbf{\varepsilon} \vert \bX_\cS] \\
		& = \frac{n-p}{n}\sigma_\cS^2 + \frac{n-p}{n}\sigma_\varepsilon^2\\
		& = \frac{n-p}{n}\sigma_{\varepsilon, \cS}^2.
	\end{aligned}
$$
\hfill\qedsymbol

\subsection{Proof of Lemma \ref{lemma: exp_inv_leverage}} \label{pf: exp_inv_leverage}
\textit{Proof.} When $\bx_1,\ldots,\bx_n \in \mathbb{R}^p$ ($p < n-2)$ are multivariate normal, \cite{jayakumar2014exact} showed that $h_{11},\ldots, h_{nn}$ are identically distributed with density
$$
	f(h_{ii};n,p) = C \left(h_{ii} - \frac{1}{n}\right)^{\frac{p-1}{2}-1} (1 - h_{ii})^{\frac{n-p}{2}-1}, \;\frac{1}{n} < h_{ii} < 1.
$$
where $C = \mathcal{B}\left(\frac{p-1}{2},\frac{n-p}{2}\right)^{-1}\left(1-\frac{1}{n}\right)^{-\frac{n-3}{2}}$ and $\mathcal{B}(\cdot,\cdot)$ is the beta function. Define
$$
	z_{ii} = \frac{n}{n-1}h_{ii} - \frac{1}{n-1},\quad i=1,\ldots,n.
$$
It is then easy to check that $z_{ii}$ follows a beta distribution with probability density
$$
	f_z(z_{ii}) = \mathcal{B}\left(\frac{p-1}{2},\frac{n-p}{2}\right)^{-1} z_{ii}^{\frac{p-1}{2}-1} (1 - z_{ii})^{\frac{n-p}{2}-1}, \; 0 < z_{ii} < 1.
$$
As a result, we can show that
$$
	\E\left(\frac{1}{1- z_{ii}}\right) = \frac{n-3}{n-p-2} \text{ and }
	\var\left(\frac{1}{1- z_{ii}}\right) = \frac{2(p-1)(n-3)}{(n-p-4)(n-p-2)^2}.
$$
It then follows immediately that
$$
	\begin{aligned}
		\E\left(\frac{1}{1- h_{ii}}\right) & = \frac{n-1}{n} \E\left(\frac{1}{1- z_{ii}}\right) = \frac{n-1}{n} \frac{n-3}{n-p-2},\\
		\var\left(\frac{1}{1- h_{ii}}\right) & = \left(\frac{n-1}{n}\right)^2 \var\left(\frac{1}{1 - z_{ii}}\right) = \left(\frac{n-1}{n}\right)^2 \frac{2(p-1)(n-3)}{(n-p-4)(n-p-2)^2}.
	\end{aligned}
$$
\hfill\qedsymbol

\subsection{Proof of Theorem \ref{thm: expected_adj_over}} \label{pf: expected_adj_over}
\textit{Proof.} Let $\bV = (\bX \bX^\T)^{-1}$. Then $\bV$ follows an inverse-Wishart distribution with scale matrix $\bI_n$ and degrees of freedom $p$, denoted by $\mathcal{W}^{-1}(\bI_n,p)$. Let $v_{ij}$ be the $(i,j)$-th entry of $\bV$. Then we could write $\trace(\bA)$ as
$$
	\trace(\bA) = \sum_{i=1}^n \sum_{j=1}^n \frac{v_{ij}^2}{v_{ii}^2} = n + \sum_{i=1}^n \sum_{j\neq i}\frac{v_{ij}^2}{v_{ii}^2}.
$$
We first show that $\sum_{j\neq i}\frac{v_{ij}^2}{v_{ii}^2}$ are identically distributed for $i=1,\ldots,n$, so that
\begin{equation} \label{eq: expected_trace_A_over}
	\E\left[\trace(\bA)\right] = n + n\E\left(\sum_{j\neq 1} \frac{v_{1j}^2}{v_{11}^2}\right).
\end{equation}
To see this, let $\mathbf{P}_{ik} \in \R^{n\times n}$ be the permutation matrix for switching the $i$th and $k$th rows. Define $\tilde{\bX}=\mathbf{P}_{ik}\bX$ and $\tilde{\bV} = (\tilde{\bX}\tilde{\bX}^\T)^{-1}$. Since the rows of $\bX$ are i.i.d, we have $\tilde{\bX}\stackrel{\rm d}{=}\bX$ and $\tilde{\bV}\stackrel{\rm d}{=}\bV$. Note that $\mathbf{P}_{ik}=\mathbf{P}_{ik}^\T=\mathbf{P}_{ik}^{-1}$. So $\tilde{\bV} = \mathbf{P}_{ik} \bV \mathbf{P}_{ik}$ and
$$
	\sum_{j\neq i}\frac{v_{ij}^2}{v_{ii}^2} = \sum_{j\neq k}\frac{\tilde{v}_{kj}^2}{\tilde{v}_{kk}^2} \stackrel{\rm d}{=} \sum_{j\neq k}\frac{v_{kj}^2}{v_{kk}^2}.
$$
Thus \eqref{eq: expected_trace_A_over} follows. 

Next, we calculate $\E\left(\sum_{j\neq 1} \frac{v_{1j}^2}{v_{11}^2}\right)$. Write
$$
	\bV =
	\begin{pmatrix}
		v_{11} & \mathbf{u}^\T  \\
		\mathbf{u} & \bV_{22} 
	\end{pmatrix}
$$
and define $\bV_{22\cdot 1} = \bV_{22}-\frac{1}{v_{11}}\mathbf{u} \mathbf{u}^\T$. We will use the following results for the inverse-Wishart distribution \citep{bodnar2008properties}:
\begin{enumerate}
	\item[(i)] $\frac{1}{v_{11}}\mathbf{u} \vert \bV_{22\cdot 1} \sim \mathcal{N}(\mathbf{0}, \bV_{22\cdot 1})$;
	\item[(ii)] $\bV_{22\cdot 1} \sim \mathcal{W}^{-1}(\bI_{n-1},p)$.
\end{enumerate}
Then
$$
	\E\left(\sum_{j\neq 1} \frac{v_{1j}^2}{v_{11}^2}\right) = \E\left(\frac{\mathbf{u}^\T \mathbf{u}}{v_{11}^2}\right) = \E\left[\E\left(\frac{\mathbf{u}^\T \mathbf{u}}{v_{11}^2}\Bigg\vert \bV_{22\cdot 1}\right)\right] = \E\left[\trace(\bV_{22\cdot 1})\right] = \frac{n-1}{p-n}.
$$
Using \eqref{eq: expected_trace_A_over}, we get
$$
	\E\left[\trace(\bA)\right] = \frac{n(p-1)}{p-n}.
$$
On the other hand, we have shown in Remark \ref{rmk: df_over_normal} that
$$
	\E(\dfR) = \frac{n}{2}\trace\left[\E\left(\bV^{-1}\right)\right] + \frac{n}{2} = \frac{n(p-1)}{2(p-n-1)}.
$$
Thus, as $n\to\infty$, $p\to\infty$ and $p/n\to\gamma>1$,
$$
	\E(\xi_\bX) = \E\left[2\,\mathrm{df}_{\rm R}(S) - \trace(\bA)\right] = \frac{n(p-1)}{(p-n-1)(p-n)} \to \frac{1}{\gamma - 1}.
$$
\hfill\qedsymbol

\subsection{Proof of Lemma \ref{lemma: gd_convergence}} \label{pf: gd_convergence}
\textit{Proof.} Using the singular value decomposition of $\bX = \mathbf{U} (\bm{\Psi}, \mathbf{O}) (\bV_1,\bV_2)^\T$, we have $\hat{\bbeta} = \bX^\T (\bX \bX^\T)^{-1} \by = \bV_1 \mathbf{\Psi}^{-1} \mathbf{U}^\T \by$ and 
$$
	\mathbf{E}^{k} = (\bI_p - \alpha \bX^\T \bX)^k = (\bV_1,\bV_2)
	\begin{pmatrix}
		(\bI_n -\alpha\bm{\Psi}^2)^k & \mathbf{O}\\
		\mathbf{O} & \bI_{p-n}
	\end{pmatrix}
	\begin{pmatrix}
		\bV_1^\T \\
		\bV_2^\T
	\end{pmatrix}.
$$
Thus, for $\alpha > 0$, \eqref{eq: gradient_descent} converges if and only if 
$$
	-1 < 1-\alpha\psi_j^2 < 1, \text{ for } j = 1,\ldots, n,
$$
which is equivalent to
$$
	0 < \alpha < \frac{2}{\psi_1^2} = \frac{2}{\lambda_{\max}(\bX^\T \bX)}. 
$$
\hfill\qedsymbol

\subsection{Proof of Theorem \ref{thm: gd_df}}\label{pf: gd_df}
\textit{Proof.} By \eqref{eq: gd_interpolant_hat_vec}, we have
$$
	\E(\Vert \bh_\ast \Vert^2 \vert \bX) = \trace[(\bF^\T \bV_2 \bV_2^\T + (\bX \bX^\T)^{-1} \bX)^\T (\bF^\T \bV_2 \bV_2^\T + (\bX \bX^\T)^{-1} \bX) \bSigma].
$$
Since $\bbeta^{(\infty)}$ is an interpolating model, if follows from \eqref{eq: dfR_interpolating_models} that
$$
	\dfR(\bbeta^{(\infty)}) =  \frac{n}{2} + \frac{n}{2}\trace[(\bF^\T \bV_2 \bV_2^\T + (\bX \bX^\T)^{-1} \bX)^\T (\bF^\T \bV_2 \bV_2^\T + (\bX \bX^\T)^{-1} \bX) \bSigma].
$$
Note that $\dfR(\hat{\bbeta}) = \frac{n}{2} + \frac{n}{2} \trace[\bX^\T (\bX \bX^\T)^{-2} \bX \bSigma]$. Then, we have
$$
	\begin{aligned}
		\dfR(\bbeta^{(\infty)}) 
		& = \dfR(\hat{\bbeta}) + \trace(\bV_2 \bV_2^\T \bF \bF^\T \bV_2 \bV_2^\T \bSigma) + 2 \trace\left[\bX^\T (\bX \bX^\T)^{-1} \bF^\T \bV_2 \bV_2^\T \bSigma \right] \\
		& = \dfR(\hat{\bbeta}) + \trace(\bV_2 \bV_2^\T \bF \bF^\T \bV_2 \bV_2^\T \bSigma) + 2 \trace\left[\bV_1 \mathbf{\Psi}^{-1} \mathbf{U}^\T \bF^\T \bV_2 \bV_2^\T \bSigma\right] \\
		& = \dfR(\hat{\bbeta}) + \trace(\bV_2 \bV_2^\T \bF \bF^\T \bV_2 \bV_2^\T \bSigma) + 2 \trace\left[\mathbf{\Psi}^{-1} \mathbf{U}^\T \bF^\T \bV_2 \bV_2^\T \bSigma \bV_1 \right].
	\end{aligned}
$$
The second term above is nonnegative for all $\bSigma$. When $\bSigma = \bI_p$, the third term vanishes due to $\bV_2^\T \bV_1 = \mathbf{O}$. Thus
$$
	\dfR(\bbeta^{(\infty)}) \geq \dfR(\hat{\bbeta}).
$$
The equality holds if and only if $\mathrm{span}(\bF) \subseteq \mathrm{span}(\bV_1)$, i.e., $\bbeta^{(\infty)} = \hat{\bbeta}$. 
\hfill\qedsymbol

\subsection{Derivation of Equation \eqref{eq: gd_interpolant_excess_bias}} \label{pf: gd_interpolant_excess_bias}

Let $\bh_\ast$ and $\bh_\ast^{(\infty)}$ be the hat vectors at $\bx_\ast$ for $\hat{\bbeta}$ and $\bbeta^{(\infty)}$ respectively. Note that both $\hat{\bbeta}$ and $\bbeta^{(\infty)}$ define interpolating models. Hence,
$$
	\Delta B_\bX (\hat{\bbeta}) = \E[(\mu_\ast - \bh_\ast^\T \bmu)^2 \vert \bX] \, \text{ and } \, \Delta B_\bX (\bbeta^{(\infty)}) = \E[(\mu_\ast - (\bh_\ast^{(\infty)})^\T \bmu)^2 \vert \bX].
$$
Since we assume $\bmu(\bx; \bbeta) = \bx^\T \bbeta$, we have
$$
	\mu_\ast - \bh_\ast^\T \bmu = \bx_\ast^\T (\bI_p - \bX^\T (\bX \bX^\T)^{-1} \bX) \bbeta = \bx_\ast^\T \bV_2 \bV_2^\T \bbeta.
$$
Then, using \eqref{eq: gd_interpolant_hat_vec} with $\bF = \mathbf{P} \bX^\T$, we have
$$
	\begin{aligned}
		(\mu_\ast - (\bh_\ast^{(\infty)})^\T \bmu)^2
		= (\mu_\ast - \bh_\ast^\T \bmu)^2 
		& - 2 (\bx_\ast^\T \bV_2\bV_2^\T \bbeta)(\bx_\ast^\T  \bV_2\bV_2^\T \mathbf{P} \bX^\T \bX \bbeta)\\
		& + (\bx_\ast^\T  \bV_2\bV_2^\T \mathbf{P} \bX^\T \bX \bbeta)^2
	\end{aligned}
$$
Under the assumption that $\E(\bx_\ast) = 0$ and $\var(\bx_\ast) = \bI_p$,
$$
	\begin{aligned}
		\Delta B_\bX (\bbeta^{(\infty)}) 
		& = \Delta B_\bX (\hat{\bbeta}) - 2 (\bV_2\bV_2^\T \bbeta)^\T (\bV_2\bV_2^\T \mathbf{P} \bX^\T \bX \bbeta) + \Vert \bV_2\bV_2^\T \mathbf{P} \bX^\T \bX \bbeta \Vert^2\\
		& = \Delta B_\bX (\hat{\bbeta}) - \Vert \bV_2\bV_2^\T \bbeta \Vert^2 + \Vert \bV_2\bV_2^\T \bbeta - \bV_2\bV_2^\T \mathbf{P} \bX^\T \bX \bbeta \Vert^2\\
		& = \Delta B_\bX (\hat{\bbeta}) - \Vert \bV_2\bV_2^\T \bbeta \Vert^2 + \Vert \bV_2\bV_2^\T \mathbf{z} \Vert^2.
	\end{aligned}
$$

\section{Examples}
\subsection{Example \ref{ex: df_wgt_funcs}} \label{det: df_wgt_funcs}
Using the expression for $\bh_\ast$ in the example, we have
$$
	\Vert \bh_\ast \Vert^2 = 
	\begin{cases}
		1, & a \leq x_\ast < x_1 \text{ or } x_n \leq x_\ast \leq b,\\
		K^2(z_{\ast,i}) + (1-K(z_{\ast,i}))^2, & x_i \leq x_\ast < x_{i+1}.
	\end{cases}
$$
Define $C_K=\int_{0}^1 \left[K^2(z) + (1-K(z))^2\right] d z$. For any distribution on $[a,b]$ with continuous and positive density $f$, we have
\begin{equation} \label{eq: ex1_expectation_decomp}
	\E(\Vert \bh_\ast\Vert^2\Vert \bX) = \int_{a}^{x_1}f(x_\ast) d x_\ast + \sum_{i=1}^{n-1}\int_{x_i}^{x_{i+1}}\Vert \bh_\ast\Vert^2 f(x_\ast) d x_\ast + \int_{x_n}^{b}f(x_\ast) d x_\ast.
\end{equation}
By the mean value theorem for definite integrals, there must exist some $\tilde{x}_i \in (x_i,x_{i+1})$ such that
$$
	\begin{aligned}
		\int_{x_i}^{x_{i+1}}\Vert \bh_\ast\Vert^2 f(x_\ast) d x_\ast 
		& = (x_{i+1}-x_i)f(\tilde{x}_i)\int_{x_i}^{x_{i+1}}\Vert \bh_\ast\Vert^2 \frac{1}{x_{i+1}-x_i} d x_\ast \\
		& = C_K (x_{i+1}-x_i)f(\tilde{x}_i).
	\end{aligned}
$$
Note that $x_1\xrightarrow{\rm P} a$, $x_n\xrightarrow{\rm P} b$ as $n \to \infty$. By applying the above result to \eqref{eq: ex1_expectation_decomp} and using the definition of definite integrals as the limit of a Riemann sum, we have
$$
	\begin{aligned}
		\E(\Vert \bh_\ast\Vert^2\Vert \bX)
		& = \int_{a}^{x_1}f(x_\ast) d x_\ast + C_K\sum_{i=1}^{n-1}(x_{i+1}-x_i)f(\tilde{x}_i) + \int_{x_n}^{b}f(x_\ast) d x_\ast\\
		& \xrightarrow{\rm P} C_K \int_{a}^b f(x_\ast) d x_\ast = C_K,\quad \text{as $n\to\infty$}.
	\end{aligned}
$$
Thus,
$$
	\frac{1}{n}\dfR = \frac{1}{2} + \frac{1}{2}\E(\Vert \bh_\ast\Vert^2\Vert \bX) \xrightarrow{\rm P} \frac{C_K + 1}{2}.
$$

\subsection{Example \ref{ex: spline_df}} \label{det: spline_df}

Let $\phi_t(x) = x^t /t!$ for $t=0,1,2,\ldots$ and
$$
	R(u,v) = \int_{0}^{1} \frac{(u-z)^{s-1}_+}{(s-1)!} \frac{(v-z)^{s-1}_+}{(s-1)!} dz, \; \text{for } u,v \in [0,1],
$$
where $(\cdot)_+ = \max(0,\cdot)$. Then, the interpolating polynomial spline of degree $2s-1$ is of the form
\begin{equation}\label{eq: spline_expr}
	\hat{\mu}(x_\ast) = \sum_{t=0}^{s-1} d_t \phi_t(x_\ast) + \sum_{i=1}^n c_i R(x_\ast,x_i).
\end{equation}
To write out $c_i$ and $d_t$, let $\mathbf{c} = (c_1,\ldots,c_n)^\T$ and $\mathbf{d} = (d_0,\ldots,d_{s-1})^\T$. Let $\mathbf{S} = (\phi_t(x_i)) \in \R^{n \times s}$ and $\mathbf{T} = (R(x_i,x_j)) \in \R^{n \times n}$. Assume $\mathbf{S}$ has Q-R decomposition
$$
	\mathbf{S} = (\mathbf{Q}_1, \mathbf{Q}_2)
	\begin{pmatrix}
		\mathbf{R}  \\
		\mathbf{O} 
	\end{pmatrix},
$$
where $\mathbf{Q}_1 \in \R^{n \times s}$, $\mathbf{Q}_2 \in \R^{n \times (n-s)}$, and $\mathbf{R} \in \R^{s \times s}$ is an upper triangular matrix. Define $\mathbf{U} = \mathbf{Q}_2 (\mathbf{Q}_2 ^\T \mathbf{T} \mathbf{Q}_2)^{-1} \mathbf{Q}_2^\T$ and $\bV = \mathbf{R}^{-1} \mathbf{Q}_1^\T (\bI - \mathbf{T}\mathbf{U})$. Then, we have
\begin{equation}\label{eq: spline_coef}
	\mathbf{c} = \mathbf{U} \by \; \text{ and }\; \mathbf{d} = \bV \by.
\end{equation}
For the derivation of \eqref{eq: spline_expr} and \eqref{eq: spline_coef}, see \cite{gu2013smoothing}. For any $x_\ast \in [0,1]$, let $\bm{\phi}_\ast = (\phi_0(x_\ast),\ldots,\phi_{s-1}(x_\ast))^\T$ and $\bm{\rho}_\ast = (R(x_\ast,x_1),\ldots,R(x_\ast,x_n))^\T$. Then, we can rewrite \eqref{eq: spline_expr} as
$$
	\hat{\mu}(x_\ast) = \bm{\phi}_\ast^\T \mathbf{d}  + \bm{\rho}_\ast^\T \mathbf{c} = (\bV^\T \bm{\phi}_\ast + \mathbf{U}^\T \bm{\rho}_\ast)^\T \by.
$$
Since $\mathbf{U}, \bV$, $\bm{\phi}_\ast$ and $\bm{\rho}_\ast$ don't depend on $\by$, we have
$$
	\bh_\ast = \bV^\T \bm{\phi}_\ast + \mathbf{U}^\T \bm{\rho}_\ast.
$$

\subsection{Example \ref{ex: df_vs_bw}} \label{det: df_vs_bw}
Assume $\frac{1}{2} \bar{L} \leq \omega \leq \ubar{L}$. Define $x_0 = a - \omega$ and $x_{n+1} = b + \omega$. By construction of the smoother, we have
$$
	\bh_\ast =
	\begin{cases}
		\mathbf{e}_i, & x_{i-1} + \omega < x_\ast < x_{i+1} - \omega, \;i = 1,\ldots,n\\
		\frac{1}{2}\mathbf{e}_i + \frac{1}{2}\mathbf{e}_{i+1}, & x_{i+1} - \omega \leq x_\ast \leq x_{i} + \omega, \;i = 1,\ldots,n-1.
	\end{cases}
$$
Since $\mathbf{e}_i \in \R^n$ is the $i$th standard basis vector,
$$
	\Vert \bh_\ast \Vert^2 =
	\begin{cases}
		1, & x_{i-1} + \omega < x_\ast < x_{i+1} - \omega, \;i = 1,\ldots,n\\
		\frac{1}{2}, & x_{i+1} - \omega \leq x_\ast \leq x_{i} + \omega, \;i = 1,\ldots,n-1.
	\end{cases}
$$
When $x_\ast \sim \mathrm{Uniform}(a,b)$,
$$
	\begin{aligned}
		\E(\Vert \bh_\ast \Vert^2 \vert \bX) 
		& = \frac{1}{b-a}\sum_{i=1}^n (x_{i+1} - x_{i-1} - 2\omega) + \frac{1}{2(b-a)}\sum_{i=1}^{n-1} (x_i - x_{i+1} + 2\omega)\\
		& = 1 + \frac{x_n - x_1}{2(b-a)} - \frac{n-1}{b-a}\omega.
	\end{aligned}
$$
Thus, by \eqref{eq: dfR_interpolating_models}, we have
$$
	\dfR(\omega) =  n + \frac{n(x_n - x_1)}{4(b-a)} - \frac{n(n-1)}{2(b-a)}\omega.
$$

\subsection{Optimal Model Size in Figure \ref{fig: mod_select_compr}} \label{ex: mod_select_compr}
To study the difference of $C_p$ and $\widetilde{\ErrR}$ in model selection, it is more convenient to look at their expected versions $\E_\bX(C_p) = \ErrF$ and $\E_\bX(\widetilde{\ErrR}) = \ErrR$. Assume $\sigma_\cS^2(p) = \alpha \left(1-\frac{p}{d}\right)^\eta$, where $\alpha > 0$, $\eta \geq 1$ and $d\leq n$. Here, $\eta$ is a parameter that controls the sparsity of the true model. Implications of this assumption include
\begin{enumerate}
	\item[i.] The signal-to-noise ratio is $\frac{\alpha}{\sigma_\varepsilon^2}$.
	\item[ii.] When $\eta = 1$, $\sigma_\cS^2$ decreases linearly in $p$, which mimics the random selection procedure or the situation where all variables have the same degree of importance.
	\item[iii.] When $\eta > 1$, $\sigma_\cS^2$ decays faster as $p$ increases. This portrays the situation where variables are included ``presciently'' from the most important to the least. 
\end{enumerate}

For simplicity, consider the case when $d=n$ and $\sigma_\varepsilon^2=1$. Let $\gamma=\frac{p}{n}$ and $c = \alpha(\eta + 1)$. Here, we allow $\gamma$ to take any real value between 0 and 1. Then, by Proposition \ref{prop: risk_est_linear_normal} and \eqref{eq: dfR_normal_expectation}, we have
$$
	\ErrT = 1-\gamma + \alpha(1-\gamma)^{\eta + 1} \text{ and } \E(\Delta B_\bX) = \alpha \gamma (1-\gamma)^\eta + \alpha \gamma (1-\gamma)^\eta \frac{1}{1-\gamma - \frac{1}{n}}.
$$
Thus,
$$
	\begin{aligned}
		\ErrF &= \ErrT + \frac{2}{n}\dfF = 1 + \gamma + \alpha (1-\gamma)^{\eta+1},\\
		\ErrR &= \ErrT + \E(\Delta B_\bX) + \frac{2}{n}\E(\dfR) \\
		& = 1 + \alpha(1-\gamma)^{\eta+1} + \alpha\gamma(1-\gamma)^{\eta} + \alpha \gamma (1-\gamma)^\eta \frac{1}{1-\gamma-\frac{1}{n}} + \frac{\gamma}{1-\gamma-\frac{1}{n}}.
	\end{aligned}
$$
Assume $0 \leq \gamma < 1- \frac{1}{n}$. Define $\gamma_{\rm F}^\ast = \argmin_\gamma \ErrF = \max(0, 1-c^{-1/\eta})$ and $\gamma_{\rm R}^\ast = \argmin_\gamma \ErrR$. Then,
$$
	\begin{aligned}
		\frac{d\ErrR}{d\gamma} 
		& = -\alpha \eta (1-\gamma)^{\eta - 1}\frac{1-\frac{1}{n}}{1-\gamma-\frac{1}{n}} + \alpha (1-\gamma)^\eta \frac{1-\frac{1}{n}}{\left(1-\gamma-\frac{1}{n}\right)^2} + \frac{1-\frac{1}{n}}{\left(1-\gamma-\frac{1}{n}\right)^2}\\
		& > -\alpha \eta (1-\gamma)^{\eta - 1}\frac{1-\frac{1}{n}}{1-\gamma-\frac{1}{n}} + \alpha (1-\gamma)^\eta \frac{1-\frac{1}{n}}{(1-\gamma)\left(1-\gamma-\frac{1}{n}\right)} + \frac{1-\frac{1}{n}}{(1-\gamma) \left(1-\gamma-\frac{1}{n}\right)}\\
		& = \frac{1-\frac{1}{n}}{(1-\gamma) \left(1-\gamma-\frac{1}{n}\right)} \left[1 - \alpha (\eta - 1) (1 - \gamma)^\eta \right] \\
		& > \frac{1-\frac{1}{n}}{(1-\gamma) \left(1-\gamma-\frac{1}{n}\right)} \left[1 - c (1 - \gamma)^\eta \right].
	\end{aligned}
$$
When $\gamma = \gamma_{\rm F}^\ast$,
$$
	\begin{aligned}
		\frac{d\ErrR}{d\gamma}\Big\vert_{\gamma=\gamma_{\rm F}^\ast} 
		& > \frac{1-\frac{1}{n}}{(1-\gamma_{\rm F}^\ast) \left(1-\gamma_{\rm F}^\ast-\frac{1}{n}\right)} [1 - c (\min(1, c^{-1/\eta}))^\eta] \\
		& = \frac{1-\frac{1}{n}}{(1-\gamma_{\rm F}^\ast) \left(1-\gamma_{\rm F}^\ast-\frac{1}{n}\right)} [1 - \min(c, c^{-1})] \\
		& \geq 0.
	\end{aligned}
$$
Since $\frac{1-\frac{1}{n}}{(1-\gamma) \left(1-\gamma-\frac{1}{n}\right)} > 0$ for $\gamma < 1- \frac{1}{n}$, and $1 - c (1 - \gamma)^\eta$ is an increasing function of $\gamma$, we also have
$$
	\frac{d\ErrR}{d\gamma} > \frac{1-\frac{1}{n}}{(1-\gamma) \left(1-\gamma-\frac{1}{n}\right)} \left[1 - c (1 - \gamma)^\eta \right] > 0, \; \text{for } \gamma \in \left(\gamma_{\rm F}^\ast,1 - \frac{1}{n}\right).
$$
This implies that $\gamma_{\rm R}^\ast \leq \gamma_{\rm F}^\ast$ for all $\alpha > 0$ and $\eta \geq 1$. In particular, when $\frac{1}{\eta+1} \leq \alpha < \frac{1}{\eta-1}$, 
$$
	\gamma_{\rm R}^\ast = 0 \; \text{and}\; \gamma_{\rm F}^\ast = 1-c^{-\frac{1}{\eta}} > 0.
$$
This means that $\ErrR$ favors a null model in this setting while $\ErrF$ doesn't. Such a difference could be large in terms of the number of selected variables when $n$ is large.
On the other hand, for any fixed $\alpha>0$,
$$
	\lim_{\eta\to\infty}\frac{d\ErrR}{d\gamma}\Big\vert_{\gamma=\gamma_{\rm F}^\ast} = 0,
$$
which suggests that $\gamma_{\rm R}^\ast$ approaches $\gamma_{\rm F}^\ast$ from the left as the level of sparsity increases.

\subsection{Cancer Mortality Data Analysis} \label{det: rda}

Based on the data partition of Figure \ref{fig: rda_subset_reg_single}, the first 6 and 7 variables were respectively selected by $\widehat{\ErrR}_+$ and $\ErrR_{\rm te}$. Table \ref{tab: rda_coef} lists the corresponding variables with their estimated coefficients. In Figure \ref{fig: rda_diagnostics}, we illustrate the prediction results of the model identified by $\widehat{\ErrR}_+$ on all 1,148 counties along with the population map. We observe that large prediction errors are more likely to occur on less populated counties. We then check the mean and standard deviation of the prediction errors by quantiles of the county-level population in Table \ref{tab: pred_err_by_pop} and find that 20\% of counties with smallest population tend to be overestimated and have much more variable prediction errors. This implies that the assumption of constant variance in our regression model does not hold exactly. Since death due to cancer is a relatively less likely event, we should expect generally larger variance for the death rate of those less populous counties.
\begin{table}[t]
	\centering
	\caption{Estimated coefficients in the optimal models identified by $\widehat{\ErrR}_+$ and $\ErrR_{\rm te}$ respectively.}
	\label{tab: rda_coef}
	\begin{threeparttable}
		\fontsize{7}{8.4}\selectfont
		\begin{tabular}{p{0.16\textwidth}p{0.43\textwidth}>{\centering}p{0.12\textwidth}>{\centering}p{0.06\textwidth}>{\centering\arraybackslash}p{0.06\textwidth}}
			\hline
			\multirow{2}{*}{Variable} & \multirow{2}{*}{Description} & \multirow{2}{*}{Transformation} & \multicolumn{2}{c}{\addstackgap[2pt]{Coefficient}} \\ 
			\cline{4-5}
			& & & \addstackgap[2pt]{$\widehat{\ErrR}_+$} & $\ErrR_{\rm te}$ \\
			\hline
			\texttt{PctPrivateCoverage} & Percent of county residents with private health coverage  & logit &  $-6.613$ & $-3.338$ \\ 
			
			\texttt{incidenceRate} & Mean \textit{per capita} (100,000) cancer diagnoses & - & $0.118$ \quad* & $0.121$ \quad* \\ 
			
			\texttt{avgDeathRateEst2015} & Estimated death rate based on 2015 population estimates and average number of reported cancer mortalities from 2010 to 2016 & - & $0.381$ *** & $0.360$ *** \\
			
			\texttt{PctPublicCoverage} & Percent of county residents with government provided health coverage & logit & $-45.562$ \quad** & $-42.542$ \quad* \\
			
			\texttt{povertyPercent} & Percent of population in poverty & - & $1.513$ \quad$\cdot$ & $1.427$ \quad$\cdot$ \\
			
			\texttt{PctUnemployed16\_Over} & Percent of county residents aged 16 and over and unemployed & logit & $23.788$ *** & $23.222$ *** \\
			
			\texttt{PctBachDeg25\_Over} & Percent of county residents aged 25 and over with bachelor's degree as the highest education attained & logit & - & $-4.110$ \\
			\hline
			& & \multicolumn{1}{r}{$\hat{\sigma}$} & $12.71$ & $12.83$ \\
			& & \multicolumn{1}{r}{adj.$R^2$} & $0.715$ & $0.709$ \\
			\cline{3-5}
		\end{tabular}
		
		\begin{tablenotes}\footnotesize
			\item \hfill Significance codes: `***' $ < 0.001$, `**' $0.001 \sim 0.01$, `*' $0.01 \sim 0.05$, `$\cdot$' $0.05 \sim 0.1$, ` ' $>0.1$.
		\end{tablenotes}
	\end{threeparttable}
\end{table}

\begin{figure}[b]
	\centering
	\includegraphics[scale = 0.54]{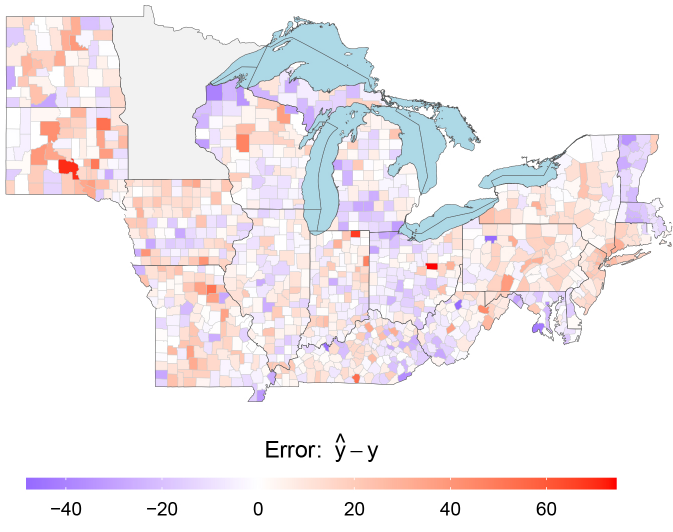}
	\includegraphics[scale = 0.54]{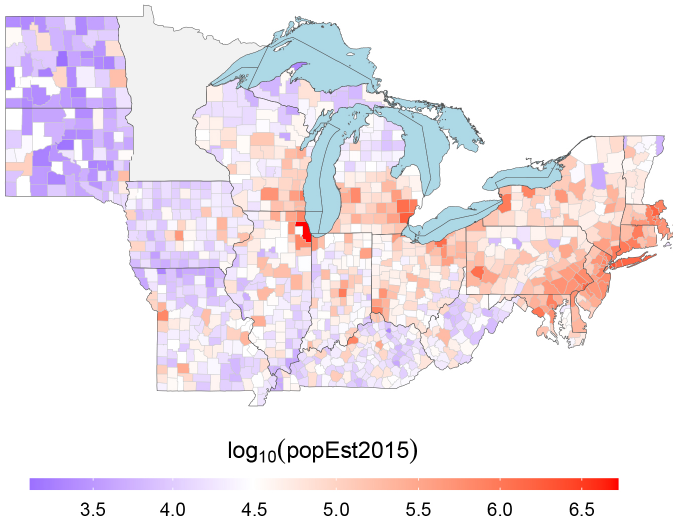}
	\caption{Prediction error (left) and population (right) maps for counties in the focused states.}
	\label{fig: rda_diagnostics}
\end{figure}

We also find a few overestimated counties in central New York, central Pennsylvania and northeast West Virginia. Geographical factors can be one possible explanation, as these counties are mostly situated in the Appalachian Mountains region, which may have different natural and socioeconomic conditions from the majority of the counties in our focused area.
\begin{table}[t]
	\centering
	\caption{Mean and standard deviation of the prediction errors ($\hat{y}-y$) by population quantiles.}\label{tab: pred_err_by_pop}
	\begin{tabular}{c|ccccc}
		\hline
		Population quantile & $[0, 0.2]$ & $(0.2, 0.4]$ & $(0.4, 0.6]$ & $(0.6, 0.8]$ & $(0.8, 1]$ \\
		\hline
		Mean& $3.537$ & $-0.511$ & $-2.365$ & $-2.123$ & $2.229$ \\
		SD &  $19.768$ & $14.779$ & $13.343$ & $12.827$ & $12.191$ \\
		\hline
	\end{tabular}
\end{table}

\end{document}